\documentclass[12pt,oneside,notitlepage]{book} 

\usepackage{amssymb,latexsym,amsmath,setspace}
\usepackage{amsthm,amsfonts}
\usepackage{graphicx}
\usepackage [left =1.50in, top =1.4in, bottom =1.3in,
    right=1.22in]{geometry}

\newcommand{\norm}[1]{\left\Vert#1\right\Vert}

\newcommand{\abs}[1]{\left\vert#1\right\vert}
\newcommand{\Lip}[1]{\text {Lip}\left(#1\right)}
\newcommand{\dia}[1]{\text {diameter}\left(#1\right)}

\newcommand{\C}{\text{const}}

\newcommand{\divergence}[1]{\text {div}#1}
\newcommand{\set}[1]{\left\{#1\right\}}

\theoremstyle{plain}
\newtheorem{thm}{Theorem}[section]
\newtheorem{cor}[thm]{Corollary}
\newtheorem{lem}[thm]{Lemma}
\newtheorem{prop}[thm]{Proposition}
\theoremstyle{definition}
\newtheorem{defn}{Definition}[section]

\theoremstyle{remark}
\newtheorem{rem}{Remark}[section]
\theoremstyle{plain}
\newtheorem*{main_assu}{Main Assumption}

\newcommand {\marginal} [1] {$\bigstar $\leavevmode\marginpar
    {\tiny\raggedright#1\par}}
\newcommand{\comment}[1]{}
\numberwithin{equation}{chapter}   

\begin{document}

\frontmatter

\pagestyle{empty}
\begin{center}
\rule{165pt}{0pt} \\
\vspace{0.5cm}
\LARGE{\textbf{Rigorous Results for the Periodic Oscillation of an Adiabatic Piston}}\\
\vspace{0.4cm}
\normalsize{by} \\
\vspace{0.2cm}
\large{Paul Wright } \\
\vspace{2cm}
\normalsize{A dissertation submitted in partial fulfillment} \\
\normalsize{of the requirements for the degree of} \\
\normalsize{Doctor of Philosophy}  \\
\normalsize{Department of Mathematics} \\
\normalsize{New York University} \\
\normalsize{May 2007} \\
\vspace{1.0cm}
\end{center}
\begin{flushright}
\vspace{.6cm}
\rule{160pt}{.6pt} \\
\normalsize{Lai-Sang Young --- Advisor}\\
\vspace{.8cm}
\end{flushright}
\pagebreak


\pagestyle{empty}           
\setcounter{page}{0}   
\qquad \pagebreak

\setcounter{page}{3} \pagestyle{plain}
\addcontentsline{toc}{section}{Dedication} \qquad \vspace{2cm}
\begin{center}
To Elizabeth
\end{center}
\pagebreak

\setcounter{page}{4} \pagestyle{plain}
\addcontentsline{toc}{section}{Acknowledgements}
\begin{center}
\huge{\textbf{Acknowledgements}}\\
\end{center}
\vspace{1cm}
\par
\normalsize{

This dissertation would not have been possible without the help and
support of many people.  First and foremost, I would like to thank
my advisor, Lai-Sang Young.  I have benefited greatly from both the
breadth of her interests and the depth of her understanding.  She
has been an excellent teacher and mentor, and I am indebted to her
for much of my understanding of mathematics.

I would also like to thank a number of other people at the Courant
Institute.  Professors Henry McKean and Jalal Shatah guided me
through much of my early time at Courant.  Oscar Lanford patiently
listened to and critiqued many of my arguments, and George Zaslavsky
contributed to my understanding of physics.  Postdoc's Kevin Lin and
Will Ott gave me constant encouragement and support, and both helped
to critique much of my writing.  P\'eter B\'alint generously
proofread a manuscript of my piston results in higher dimensions.

Of course, graduate school would never have been the same without my
fellow travelers, the other graduate students, to whom I am
grateful.  I would especially like to thank the other dynamics
students, Jos\'e Koiller, Stan Mintchev, and Tanya Yarmola, for
being good friends and study partners.  I would also like to thank
Tom Alberts, Gil Ariel, and Paris Pender, and my many great
officemates, Hantaek Bae, Umberto Hryniewicz, Helga Schaffrin, and
Ross Tulloch.

I must also thank Dmitry Dolgopyat of the University of Maryland,
who first suggested the adiabatic piston problem to me, and who
generously shared with me his unpublished notes on averaging.  In
addition, I am grateful to Marco Lenci and Luca Bussolari of the
Stevens Institute of Technology for their participation in Courant's
dynamical systems seminar, and to Albert Fathi of the ENS Lyon for
arranging for me to be a visitor there during the fall of 2005.
During graduate school, I was partially supported by a National
Science Foundation Graduate Research Fellowship.

Finally, I would like to thank my family, who have supported me
throughout my studies and enabled me to become who I am today.  My
mother shared with me her love of learning, and my father shared the
beauty of mathematics.  My wife has strengthened me in countless
ways.  I cannot thank her enough for sharing with me her writing
skills and her expert advice.  I dedicate this work to her.

} \pagebreak

\pagestyle{plain}
\addcontentsline{toc}{section}{Abstract}
\begin{center}
\huge{\textbf{Abstract}}\\
\end{center}
\vspace{1cm}
\par
\normalsize{

We study a heavy piston of mass $M$ that moves in one dimension. The
piston separates two gas chambers, each of which contains finitely
many ideal, unit mass gas particles moving in $d$ dimensions, where
$ d\geq 1$. Using averaging techniques, we prove that the actual
motions of the piston converge in probability to the predicted
averaged behavior on the time scale $M^ {1/2} $ when $M$ tends to
infinity while the total energy of the system is bounded and the
number of gas particles is fixed. Neishtadt and Sinai previously
pointed out that an averaging theorem due to Anosov should extend to
this situation.

When $ d=1$, the gas particles move in just one dimension, and we
prove that the rate of convergence of the actual motions of the
piston to its averaged behavior is $\mathcal{O} (M^ {-1/2}) $ on the
time scale $M^ {1/2} $. The convergence is uniform over all initial
conditions in a compact set. We also investigate the piston system
when the particle interactions have been smoothed. The convergence
to the averaged behavior again takes place uniformly, both over
initial conditions and over the amount of smoothing.

In addition, we prove generalizations of our results to $N$ pistons
separating $N+1$ gas chambers.  We also provide a general discussion
of averaging theory and the proofs of a number of previously known
averaging results.  In particular, we include a new proof of
Anosov's averaging theorem for smooth systems that is primarily due
to Dolgopyat.

} \pagebreak


\tableofcontents


\listoffigures
\addcontentsline{toc}{section}{List of Figures}





\mainmatter

\chapter{Introduction}\label{chp:intro}

What can be rigorously understood about the nonequilibrium dynamics
of chaotic, many particle systems?  Although much progress has been
made in understanding the infinite time behavior of such systems,
our understanding on finite time scales is still far from complete.
Systems of many particles contain a large number of degrees of
freedom, and it is often impractical or impossible to keep track of
their full dynamics.  However, if one is only interested in the
evolution of macroscopic quantities, then these variables form a
small subset of all of the variables. The evolution of these
quantities does not itself form a closed dynamical system, because
it depends on events happening in all of the (very large) phase
space. We must therefore develop techniques for describing the
evolution of just a few variables in phase space. Such descriptions
are valid on limited time scales because a large amount of
information about the dynamics of the full system is lost. However,
the time scales of validity can often be long enough to enable a
good prediction of the observable dynamics.

Averaging techniques help to describe the evolution of certain
variables in some physical systems, especially when the system has
components that move on different time scales.  The primary results
of this thesis involve applying averaging techniques to chaotic
microscopic models of gas particles separated by an adiabatic piston
for the purposes of justifying and understanding macroscopic laws.

This thesis is organized as follows.  In
Section~\ref{sct:piston_intro} we briefly introduce the the
adiabatic piston problem and our results.  In
Section~\ref{sct:heuristic} we review the physical motivations for
our results.  The following three chapters may each be read
independently. Chapter~\ref{chp:averaging} presents an introduction
to averaging theory and the proofs of a number of averaging theorems
for smooth systems that motivate our later proofs for the piston
problem.  Chapter~\ref{chp:1Dpiston} contains our results for piston
systems in one dimension, and Chapter~\ref{chp:dDpiston} contains
our results for the piston system in dimensions two and three.

\section{The adiabatic piston}\label{sct:piston_intro}

Consider the following simple model of an adiabatic piston
separating two gas containers:  A massive piston of mass $M\gg 1$
divides a container in $\mathbb{R}^d$, $ d=1, 2, \text{ or } 3$,
into two halves. The piston has no internal degrees of freedom and
can only move along one axis of the container. On either side of the
piston there are a finite number of ideal, unit mass, point gas
particles that interact with the walls of the container and with the
piston via elastic collisions. When $M=\infty $, the piston remains
fixed in place, and each gas particle performs billiard motion at a
constant energy in its sub-container. We make an ergodicity
assumption on the behavior of the gas particles when the piston is
fixed. Then we study the motions of the piston when the number of
gas particles is fixed, the total energy of the system is bounded,
but $M$ is very large.

Heuristically, after some time, one expects the system to approach a
steady state, where the energy of the system is equidistributed
amongst the particles and the piston. However, even if we could show
that the full system is ergodic, an abstract ergodic theorem says
nothing about the time scale required to reach such a steady state.
Because the piston will move much slower than a typical gas
particle, it is natural to try to determine the intermediate
behavior of the piston by averaging techniques. By averaging over
the motion of the gas particles on a time scale chosen short enough
that the piston is nearly fixed, but long enough that the ergodic
behavior of individual gas particles is observable, we will show
that the system does not approach the expected steady state on the
time scale $M^ {1/2} $. Instead, the piston oscillates periodically,
and there is no net energy transfer between the gas particles.

The results of this thesis follow earlier work by Neishtadt and
Sinai~\cite{Sin99,NS04}. They determined that for a wide variety of
Hamiltonians for the gas particles, the averaged behavior of the
piston is periodic oscillation, with the piston moving inside an
effective potential well whose shape depends on the initial position
of the piston and the gas particles' Hamiltonians.  They pointed out
that an averaging theorem due to Anosov~\cite{Ano60,LM88}, proved
for smooth systems, should extend to this case. The main result of
the present work, Theorem~\ref{thm:dDpiston}, is that Anosov's
theorem does extend to the particular gas particle Hamiltonian
described above. Thus, if we examine the actual motions of the
piston with respect to the slow time $\tau=t/M^ {1/2}$, then, as
$M\rightarrow\infty $, in probability (with respect to Liouville
measure) most initial conditions give rise to orbits whose actual
motion is accurately described by the averaged behavior for
$0\leq\tau\leq 1$, i.e.~for $0\leq t\leq M^ {1/2}$.

A recent study involving some similar ideas by Chernov and
Dolgopyat~\cite{CD06} considered the motion inside a two-dimensional
domain of a single heavy, large gas particle (a disk) of mass $M\gg
1$ and a single unit mass point particle.  They assumed that for
each fixed location of the heavy particle, the light particle moves
inside a dispersing (Sinai) billiard domain. By averaging over the
strongly hyperbolic motions of the light particle, they showed that
under an appropriate scaling of space and time the limiting process
of the heavy particle's velocity is a (time-inhomogeneous) Brownian
motion on a time scale $\mathcal{O} (M^ {1/2}) $. It is not clear
whether a similar result holds for the piston problem, even for gas
containers with good hyperbolic properties, such as the Bunimovich
stadium.  In such a container the motion of a gas particle when the
piston is fixed is only nonuniformly hyperbolic because it can
experience many collisions with the flat walls of the container
immediately preceding and following a collision with the piston.

The present work provides a weak law of large numbers, and it is an
open problem to describe the sizes of the deviations for the piston
problem~\cite{CD06b}.  Although our result does not yield concrete
information on the sizes of the deviations, it is general in that it
imposes very few conditions on the shape of the gas container. Most
studies of billiard systems impose strict conditions on the shape of
the boundary, generally involving the sign of the curvature and how
the corners are put together.  The proofs in this work require no
such restrictions.  In particular, the gas container can have cusps
as corners and need satisfy no hyperbolicity conditions.

If the piston divides a container in $\mathbb{R}^2$ or
$\mathbb{R}^3$ with axial symmetry, such as a rectangle or a
cylinder, then our ergodicity assumption on the behavior of the gas
particles when the piston is fixed does not hold. In this case, the
interactions of the gas particles with the piston and the ends of
the container are completely specified by their motions along the
normal axis of the container. Thus, this system projects onto a
system inside an interval consisting of a massive point particle,
the piston, which interacts with the gas particles on either side of
it. These gas particles make elastic collisions with the walls at
the ends of the container and with the piston, but they do not
interact with each other.  For such one-dimensional containers, the
effects of the gas particles are quasi-periodic and can be
essentially decoupled, and we recover a strong law of large numbers
with a uniform rate, reminiscent of classical averaging over just
one fast variable in $S^1$: The convergence of the actual motions to
the averaged behavior is uniform over all initial conditions, with
the size of the deviations being no larger than $\mathcal{O} (M^
{-1/2}) $ on the time scale $M^ {+1/2} $.  See
Theorem~\ref{thm:1Dpiston1}. Gorelyshev and
Neishtadt~\cite{GorNeishtadt06} independently obtained this result.

For systems in $ d=1$ dimension, we also investigate the behavior of
the system when the interactions of the gas particles with the walls
and the piston have been smoothed, so that Anosov's theorem applies
directly.  Let $\delta\geq 0$ be a parameter of smoothing, so that
$\delta=0 $ corresponds to the hard core setting above. Then the
averaged behavior of the piston is still a periodic oscillation,
which depends smoothly on $\delta$.  We show that the deviations of
the actual motions of the piston from the averaged behavior are
again not more than $\mathcal{O} (M^ {-1/2}) $ on the time scale $M^
{1/2} $. The size of the deviations is bounded uniformly, both over
initial conditions and over the amount of smoothing,
Theorem~\ref{thm:1D_smooth_uniform}.

Our results for a single heavy piston separating two gas containers
generalize to the case of $N$ heavy pistons separating $N+1$ gas
containers.  Here the averaged behavior of the pistons has them
moving like an $N$-dimensional particle inside an effective
potential well.  Compare Section~\ref{sct:apps_generalizations}.

The systems under consideration in this work are simple models of an
adiabatic piston. The general adiabatic piston problem~\cite {Ca63},
well-known from physics, consists of the following: An insulating
piston separates two gas containers, and initially the piston is
fixed in place, and the gas in each container is in a separate
thermal equilibrium. At some time, the piston is no longer
externally constrained and is free to move. One hopes to show that
eventually the system will come to a full thermal equilibrium, where
each gas has the same pressure and temperature. Whether the system
will evolve to thermal equilibrium and the interim behavior of the
piston are mechanical problems, not adequately described by
thermodynamics~\cite{Gru99}, that have recently generated much
interest within the physics and mathematics communities following
Lieb's address~\cite{Lie99}. One expects that the system will evolve
in at least two stages. First, the system relaxes deterministically
toward a mechanical equilibrium, where the pressures on either side
of the piston are equal.  In the second, much longer, stage, the
piston drifts stochastically in the direction of the hotter gas, and
the temperatures of the gases equilibrate.  See for
example~\cite{GPL03,CL02,Che05} and the references therein.
Previously, rigorous results have been limited mainly to models
where the effects of gas particles recolliding with the piston can
be neglected, either by restricting to extremely short time
scales~\cite{CLS02,CLS02b} or to infinite gas
containers~\cite{Che05}.

\section{Physical motivation for the results}\label{sct:heuristic}

In this section, we briefly review the physical motivations for our
results on the adiabatic piston.

Consider a massive, insulating piston of mass $M$ that separates a
gas container $\mathcal{D} $ in $\mathbb{R}^d$, $ d= 1,2,\text { or
}3$. See Figure~\ref{fig:intro_domain}. Denote the location of the
piston by $Q$ and its velocity by $dQ/dt=V$.  If $Q$ is fixed, then
the piston divides $\mathcal{D} $ into two subdomains,
$\mathcal{D}_1(Q) =\mathcal{D}_1 $ on the left and $\mathcal{D}_2(Q)
=\mathcal{D}_2  $ on the right.  By $\abs{\mathcal{D}_i} $ we denote
the area (when $ d=2$, or length, when $ d=1$, or volume, when $
d=3$) of $\mathcal{D}_i$.  Define
\[
    \ell : =\frac {\partial\abs{\mathcal{D}_1(Q)}} {\partial Q}=
    -\frac {\partial\abs{\mathcal{D}_2(Q)}} {\partial Q},
\]
so that $\ell$ is the piston's cross-sectional length (when $ d=2$,
or area, when $ d=3$).  If $ d=1$, then $\ell=1$.  By $E_i$ we
denote the total energy of the gas inside $\mathcal{D}_i$.

\begin{figure}
    \begin {center}
    \setlength{\unitlength}{0.9 cm}
    \begin{picture}(15,6)
        \put(2.5,1){\line(1,0){10}}
        \put(2.5,5){\line(1,0){10}}
        \qbezier(2.5,1)(0.5,1)(0.5,3)
        \qbezier(0.5,3)(0.5,5)(2.5,5)
        \qbezier(12.5,1)(14.5,1)(14.5,3)
        \qbezier(14.5,3)(14.5,5)(12.5,5)
        \put(2.2,4.55){$\mathcal{D}_1(Q)$}
        \put(11.5,4.55){$\mathcal{D}_2(Q)$}
        \put(3.1,1.3){$E_1$}
        \put(12,1.3){$E_2$}
        \put(1.0,5.4){$\mathcal{D} =\mathcal{D}_1(Q)\sqcup\mathcal{D}_2(Q)$}
        \linethickness {0.15cm}
        \put(8.5,1){\line(0,1){4}}
        \thinlines
        \put(8.5,3){\vector(1,0){.35}}
        \put(9.0,2.8){$V=\varepsilon W$}
        \put(8.5,1){\line(0,-1){0.15}}
        \put(8.35,0.33){$Q$}
        \put(8.35,5.2){$M=\varepsilon ^ {-2}\gg 1$}
        \put(5,4){\circle*{.15}}
        \put(5, 4){\vector(1,1){0.7}}
        \put(11.5,2.5){\circle*{.15}}
        \put(11.5, 2.5){\vector(1,-1){0.6}}
        \put(2.7,1.4){\circle*{.15}}
        \put(2.7, 1.4){\vector(1,2){.4}}
        \put(7.3,4.5){\circle*{.15}}
        \put(7.3, 4.5){\vector(-2,0){1.15}}
        \put(3.8,4){\circle*{.15}}
        \put(3.8, 4){\vector(1,3){.20}}
        \put(5.4,2.1){\circle*{.15}}
        \put(5.4, 2.1){\vector(-2,1){.65}}
        \put(9.0,4.1){\circle*{.15}}
        \put(9.0, 4.1){\vector(1,3){.25}}
        \put(11.2,4){\circle*{.15}}
        \put(11.2,4){\vector(1,0){1.45}}
        \put(14.1,3.0){\circle*{.15}}
        \put(14.1, 3.0){\vector(-2,-1){.55}}
        \put(0,1){\line(0,1){4}}
        \put(0,1){\line(1,0){0.15}}
        \put(0,5){\line(1,0){0.15}}
        \put(0,3.0){\line(-1,0){0.15}}
        \put(-0.5,2.88){$\ell$}
    \end{picture}
    \normalsize
    \end {center}
    \caption{A gas container $\mathcal{D} $ in $d=2$ dimensions separated by
        an adiabatic piston.}
    \label{fig:intro_domain}
\end{figure}
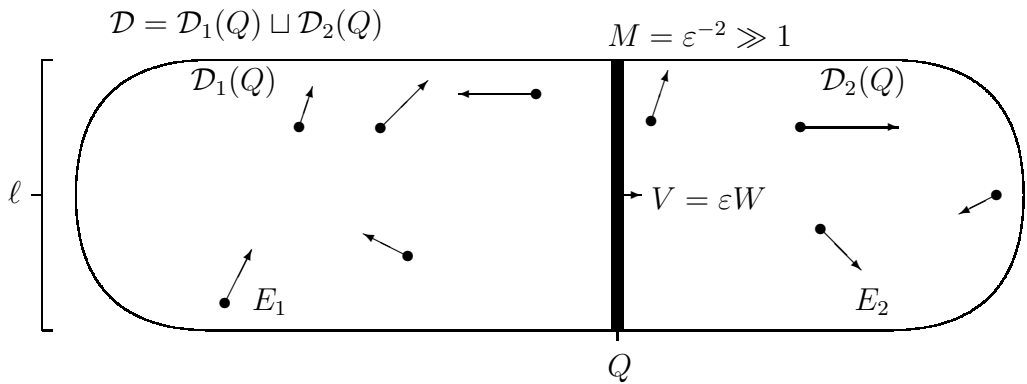

We are interested in the dynamics of the piston when the system's
total energy is bounded and $M\rightarrow \infty $.  When
$M=\infty$, the piston remains fixed in place, and each energy $E_i$
remains constant. When $M$ is large but finite, $MV^2/2$ is bounded,
and so $V=\mathcal{O} (M^ {-1/2}) $.  It is natural to define
\[
\begin {split}
    \varepsilon &=M^ {-1/2},\\ W &=\frac {V} {\varepsilon},
\end {split}
\]
so that $W$ is of order $1$ as $\varepsilon\rightarrow 0$.  This is
equivalent to scaling time by $\varepsilon$, and so we introduce the
slow time
\[
    \tau=\varepsilon t.
\]

If we let $P_i$ denote the pressure of the gas inside
$\mathcal{D}_i$, then heuristically the dynamics of the piston
should be governed by the following differential equation:
\begin {equation}\label{eq:piston_force}
\begin {split}
    \frac {dQ} {dt}&=V,\qquad M\frac {dV} {dt}=P_1\ell-P_2\ell,
    \\
    &\qquad\text {i.e.}
    \\
    \frac {dQ} {d\tau}&= W,\qquad
    \frac {dW} {d\tau}= P_1\ell- P_2\ell.
\end {split}
\end {equation}
To find differential equations for the energies of the gases, note
that in a short amount of time $dt$, the change in energy should
come entirely from the work done on a gas, i.e.~the force applied to
the gas times the distance the piston has moved, because the piston
is adiabatic. Thus, one expects that
\begin {equation}\label{eq:work}
\begin {split}
    \frac {dE_1} {dt}&=-V P_1\ell,
    \qquad
    \frac {dE_2} {dt} = +V  P_2\ell,
    \\
    &\qquad\text {i.e.}
    \\
    \frac {dE_1} {d\tau}&=- WP_1\ell,
    \qquad
    \frac {dE_2} {d\tau} = + W P_2\ell.
\end {split}
\end {equation}

To obtain a closed system of differential equations, it is necessary
to insert an expression for the pressures.   $P_i \ell$ should be
the average force from the gas particles in $\mathcal{D}_i$
experienced by the piston when it is held fixed in place.  Whether
such an expression, depending only on $E_i$ and $\mathcal{D}_i(Q)$,
exists and is the same for (almost) every initial condition of the
gas particles depends strongly on the microscopic model of the gas
particle dynamics.  Sinai and Neishtadt~\cite{Sin99, NS04} pointed
out that for many microscopic models where the pressures are well
defined, the solutions of Equations~\eqref{eq:piston_force}
and~\eqref{eq:work} have the piston moving according to a
model-dependent effective Hamiltonian.

Because the pressure of an ideal gas in $d$ dimensions is
proportional to the energy density, with the constant of
proportionality $2/d$, we choose to insert
\[
    P_i=\frac {2E_i}{d\abs{\mathcal{D}_i}} .
\]
Later, we will make assumptions on the microscopic gas particle
dynamics to justify this substitution. However, if we accept this
definition of the pressure, we obtain the following ordinary
differential equations for the four macroscopic variables of the
system:
\begin {equation}\label{eq:heuristic_avg_eq}
    \frac{d}{d\tau}
    \begin {bmatrix}
    Q\\
    W\\
    E_1\\
    E_2\\
    \end {bmatrix}
    =
    \begin {bmatrix}
    \displaystyle W\\
    \displaystyle \frac{2E_1\ell}{d\abs{\mathcal{D}_1(Q)}}
    -\frac{2E_2\ell}{d\abs{\mathcal{D}_2(Q)}}\\
    \displaystyle -\frac{2WE_1\ell}{d\abs{\mathcal{D}_1(Q)}}\\
    \displaystyle +\frac{2WE_2\ell}{d\abs{\mathcal{D}_2(Q)}}\\
    \end {bmatrix}.
\end {equation}
For these equations, one can see the effective Hamiltonian as
follows.  Since
\[
    \frac {d\ln(E_i)}{d\tau}=
    -\frac{2}{d}\frac {d\ln(\abs{\mathcal{D}_i(Q)})} {d\tau},
\]
\[
    E_i(\tau)=E_i(0)\left (\frac{\abs{\mathcal{D}_i(Q(0))}}
        {\abs{\mathcal{D}_i(Q(\tau))}}\right) ^ {2/d}.
\]
Hence
\[
    \frac {d^2Q(\tau)} {d\tau^2}=
    \frac {2\ell} {d}
    \frac{E_1(0)\abs{\mathcal{D}_1(Q(0))}^{2/d}}{\abs{\mathcal{D}_1(Q(\tau))}^{1+2/d}}
    -\frac {2\ell} {d}
    \frac{E_2(0)\abs{\mathcal{D}_2(Q(0))}^{2/d}}{\abs{\mathcal{D}_2(Q(\tau))}^{1+2/d}},
\]
and so $(Q, W) $ behave as if they were the coordinates of a
Hamiltonian system describing a particle undergoing motion inside a
potential well. The effective Hamiltonian may be expressed as
\begin {equation}\label{eq:d_dpot}
    \frac {1} {2}W^2+
    \frac{E_1(0)\abs{\mathcal{D}_1(Q(0))}^{2/d}}
    {\abs{\mathcal{D}_1(Q)}^{2/d}}+
    \frac{E_2(0)\abs{\mathcal{D}_2(Q(0))}^{2/d}}
    {\abs{\mathcal{D}_2(Q)}^{2/d}}.
\end {equation}

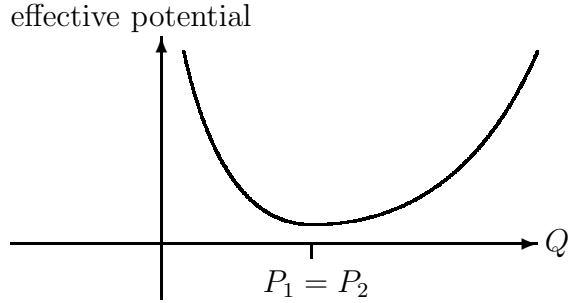
\begin{figure}
    \begin {center}
    \setlength{\unitlength}{1.0 cm}
    \begin{picture}(10,4)
        \thicklines
        \put(1,0.75){\vector(1,0){7}}
        \put(8.1,0.65){$Q$}
        \put(1,3.65){effective potential}
        \put(3,0){\vector(0,1){3.5}}
        \put(5,0.75){\line(0,-1){0.2}}
        \put(4.35,0.1){$P_1= P_2$}
        \qbezier(3.3,3.3)(3.8,1)(5,1)
        \qbezier(5,1)(7,1)(8,3.3)
    \end{picture}
    \end {center}
    \caption{An effective potential.}
    \label{fig:potential}
\end{figure}

The question is, do the solutions of
Equation~\eqref{eq:heuristic_avg_eq} give an accurate description of
the actual motions of the macroscopic variables when $M$ tends to
infinity?  The main result of this thesis,
Theorem~\ref{thm:dDpiston}, is that, for an appropriately defined
system, the answer to this question is affirmative for $0\leq t\leq
M^ {1/2}$, at least for most initial conditions of the microscopic
variables. Observe that one should not expect the description to be
accurate on time scales much longer than $\mathcal{O} (M^ {1/2})
=\mathcal{O} (\varepsilon^ {-1})$.  The reason for this is that,
presumably, there are corrections of size $\mathcal{O}
(\varepsilon)$ in Equation~\eqref{eq:heuristic_avg_eq} that we are
neglecting.  For $\tau=\varepsilon t>\mathcal{O} (1) $, these
corrections should become significant.  Such higher order
corrections for the adiabatic piston were studied by Crosignani
\emph{et al.}~\cite{CD96}.

\chapter{Background Averaging Material}\label{chp:averaging}

In this chapter, we present a number of well-known classical
averaging results for smooth systems, as well as a proof of Anosov's
averaging theorem, which is the first general multi-phase averaging
result. All of these theorems are at least 45 years old. However, we
present them here because our proofs of the classical results are at
least slightly novel, and the ideas in them lend themselves well to
certain higher-dimensional generalizations. In particular, they are
fairly close to the ideas in the proof we give for our piston
results in one dimension.  The proof of Anosov's theorem is a new
and unpublished proof due mainly to Dolgopyat, with some further
simplifications made. The ideas in this proof underly the ideas we
will use to prove the weak law of large numbers for our piston
system in dimensions two and three.

We begin by giving a discussion of a framework for general averaging
theory and some averaging results.  A number of classical averaging
theorems are then proved, followed by the proof of Anosov's theorem.

\section{The averaging framework}
\label{sct:anos}

In this section, consider a family of ordinary differential
equations
\begin {equation}\label{eq:ode}
\frac {dz} {dt}=Z(z,\varepsilon)
\end {equation}
on a smooth, finite-dimensional Riemannian manifold $\mathcal{M}$,
which is indexed by the real parameter $\varepsilon\in
[0,\varepsilon_0]$. Assume
\begin {itemize}
\item
     {\em  Regularity:} the functions $Z$ and
            $\partial Z/\partial \varepsilon$ are both
            $\mathcal{C}^1$ on $\mathcal{M}\times [0,\varepsilon_0] $.
\end {itemize}
We denote the flow generated by $Z(\cdot,\varepsilon) $ by
$z_\varepsilon(t,z)=z_\varepsilon(t) $. We will usually suppress the
dependence on the initial condition $z=z_\varepsilon(0,z) $.  Think
of $z_\varepsilon(\cdot) $ as being a random variable whose domain
is the space of initial conditions for the differential equation
\eqref{eq:ode} and whose range is the space of continuous paths
(depending on the parameter $t$) in $\mathcal{M}$.
\begin {itemize}
\item
     {\em  Existence of smooth integrals:} $z_0(t) $ has $m$ independent
            $\mathcal{C}^2$ first integrals
            $h=(h_1,\dotsc,h_m):\mathcal{M}\rightarrow\mathbb{R}^m$.
\end {itemize}
Then $h$ is conserved by $z_0(t) $, and at every point the linear
operator $\partial h/\partial z$ has full rank.  It follows from the
implicit function theorem that each level set
\[
    \mathcal{M}_c :=\{h=c\}
\]
is a smooth submanifold of co-dimension $m$, which is invariant
under $z_0(t) $. Further, assume that there exists an open ball
$\mathcal{U}\subset\mathbb{R}^m$ satisfying:
\begin {itemize}
    \item
        {\em Compactness:} $\forall c\in\mathcal{U},\: \mathcal{M}_c$ is compact.
    \item
        {\em Preservation of smooth measures:} $\forall c\in\mathcal{U}$,
        $z_0(t)\arrowvert_{\mathcal{M}_c} $ preserves a smooth measure
        $\mu_c$ that varies smoothly with $c$, i.e.~there exists a
        $\mathcal{C}^1 $ function $g:\mathcal{M}\rightarrow\mathbb{R}_{>0} $ such that
        $g\arrowvert_{\mathcal{M}_c} $ is the density of $\mu_c$ with respect
        to the restriction of Riemannian volume.
\end {itemize}

Set
\[
    h_\varepsilon(t,z) =h_\varepsilon(t) :=h(z_\varepsilon(t)).
\]
Again, think of $h_\varepsilon(\cdot) $ as being a random variable
that takes initial conditions $z\in\mathcal{M}$ to continuous paths
(depending on the parameter $t$) in $\mathcal{U}$. Since $dh_0/dt
\equiv0 $, Hadamard's Lemma allows us to write
\[
\frac {dh_\varepsilon} {dt} =\varepsilon
H(z_\varepsilon,\varepsilon)
\]
for some $\mathcal{C} ^1$ function $H:\mathcal{M}\times
[0,\varepsilon_0]\rightarrow\mathcal{U}$. Observe that
\[
    \frac {dh_\varepsilon} {dt}(t) =
    Dh(z_\varepsilon(t))Z(z_\varepsilon(t),\varepsilon)=
    Dh(z_\varepsilon(t))\bigl(Z(z_\varepsilon(t),\varepsilon)
    -Z(z_\varepsilon(t),0)\bigr),
\]
so that
\[
H(z,0) =\mathcal{L}_{\frac {\partial Z}
{\partial\varepsilon}\arrowvert_{\varepsilon=0}}h.
\]
Here $\mathcal{L}$ denotes the Lie derivative.

Define the averaged vector field $\bar{H} $ by
\begin{equation}
\label{eq:Anosov_ avg}
    \bar {H} (h)
    =\int_{\mathcal{M}_h}H(z,0)d\mu_h(z).
\end{equation}
Then $\bar {H} $ is $\mathcal{C}^1$. Fix a compact set
$\mathcal{V}\subset\mathcal{U}$, and introduce the slow time
\[
    \tau=\varepsilon t.
\]
Let $\bar{h} (\tau,z)=\bar{h} (\tau) $ be the random variable that
is the solution of
\[
\frac {d\bar{h}}{d\tau} =\bar {H} (\bar {h}),\qquad \bar {h} (0)
=h_\varepsilon(0).
\]
We only consider the dynamics in a compact subset of phase space, so
for initial conditions $z\in h^ {-1} \mathcal{U}$, define the
stopping time
\[
    T_\varepsilon(z) =T_\varepsilon=\inf \{\tau\geq 0: \bar {h}
    (\tau)\notin \mathcal{V} \text { or } h_\varepsilon(\tau
    /\varepsilon) \notin \mathcal{V} \}.
\]

Heuristically, think of the phase space $\mathcal{M}$ as being a
fiber bundle whose base is the open set $\mathcal{U}$ and whose
fibers are the compact sets $\mathcal{M}_h$. See Figure
\ref{fig:phase_space}. Then the vector field $Z(\cdot,0) $ is
perpendicular to the base, so its orbits $z_0(t) $ flow only along
the fibers.  Now when $0< \varepsilon\ll 1 $, the vector field
$Z(\cdot,\varepsilon) $ acquires a component of size $\mathcal{O}
(\varepsilon) $ along the base, and so its orbits $z_\varepsilon(t)
$ have a small drift along the base, which we can follow by
observing the evolution of $h_\varepsilon(t) $. Because of this, we
refer to $h$ as consisting of the slow variables. Other variables,
used to complete $h$ to a parameterization of (a piece of) phase
space, are called fast variables. Note that $h_\varepsilon(t)$
depends on all the dimensions of phase space, and so it is not the
flow of a vector field on the $m$-dimensional space $\mathcal{U}$.
However, because the motion along each fiber is relatively fast
compared to the motion across fibers, we hope to be able to average
over the fast motions and obtain a vector field on $\mathcal{U}$
that gives a good description of $h_\varepsilon(t)$ over a
relatively long time interval, independent of where the solution
$z_\varepsilon(t) $ started on $\mathcal{M}_{h_\varepsilon(0)} $.
Because our averaged vector field, as defined by Equation
\eqref{eq:Anosov_ avg}, only accounts for deviations of size
$\mathcal{O}(\varepsilon) $, we cannot expect this time interval to
be longer than size $\mathcal{O}(1/\varepsilon)$.  In terms of the
slow time $\tau=\varepsilon t$, this length becomes $\mathcal{O}(1)
$.  In other words, the goal of the first-order averaging method
described above should be to show that, in some sense,
$\sup_{0\leq\tau\leq 1\wedge
T_\varepsilon}\abs{h_\varepsilon(\tau/\varepsilon)-\bar{h}(\tau)}\rightarrow
0$ as $\varepsilon\rightarrow 0$.  This is often referred to as the
averaging principle.

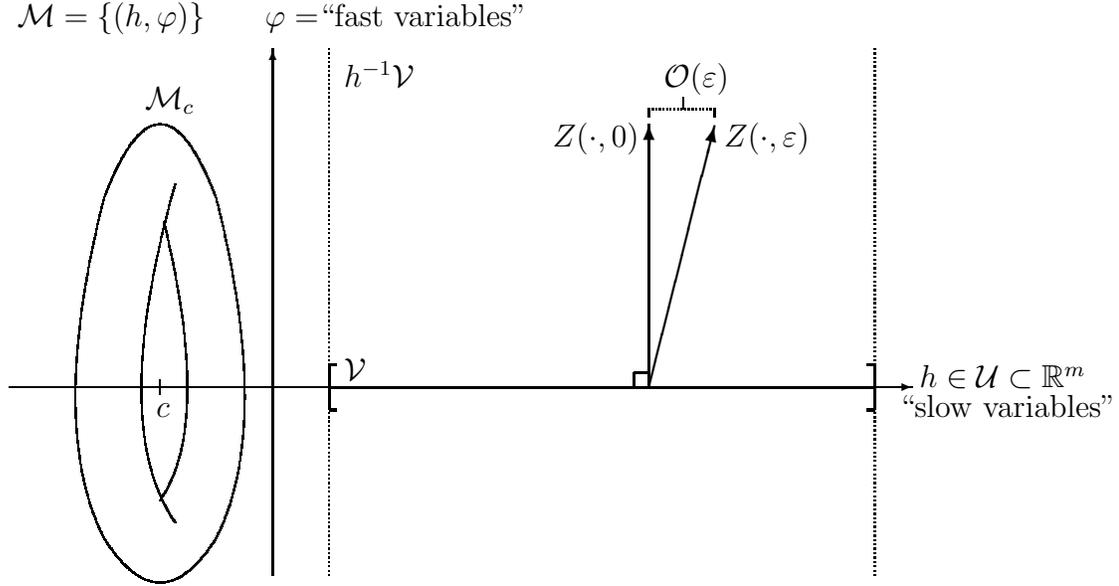
\begin{figure}
    \begin {center}
    \setlength{\unitlength}{1 cm}
    \begin{picture}(15,8)
        \put(0.1,7.3){$\mathcal{M}=\{(h,\varphi)\}$}
        \put(0,2.5){\vector(1,0){12}}
        \put(12.1,2.5){$h\in\mathcal{U}\subset\mathbb{R}^m$}
        \put(12.1,2.1){\!\!\!\!``slow variables''}
        \put(3.5,0){\vector(0,1){7}}
        \put(3.4,7.3){$\varphi=$``fast variables''}
        \thicklines
        \put(8.5,2.5){\vector(0,1){3.5}}
        \put(8.5,2.5){\vector(1,4){0.87}}
        \put(8.3,2.7){\line(1,0){.2}} \put(8.3,2.5){\line(0,1){.2}}
        \put(7.22,5.7){$Z(\cdot,0) $}
        \put(9.5,5.7){$Z(\cdot,\varepsilon) $}
        \thinlines
        \qbezier[18](8.5,6.2)(8.93,6.2)(9.37,6.2)
        \put(8.95,6.2){\line(0,1){.15}}
        \put(8.5,6.2){\line(0,-1){0.1}}
        \put(9.37,6.2){\line(0,-1){0.1}}
        \put(8.7,6.5){$\mathcal{O} (\varepsilon)$}
        \put(2,2.6){\line(0,-1){0.2}}
        \put(1.95,2.1){$c$}
        \qbezier(2.2,5.2)(1.3,2)(2.2,0.7)
        \qbezier(2.05,4.7)(2.7,2)(2,1)
        \qbezier(1.25,5)(.5,2)(1.25,0.5)
        \qbezier(2.75,5)(3.5,2)(2.75,0.5)
        \qbezier(1.25,5)(2,7)(2.75,5)
        \qbezier(1.25,0.5)(2,-0.7)(2.75,0.5)
        \put(1.8,6.2){$\mathcal{M}_c$}
        \qbezier[120](4.25,0)(4.25,3.5)(4.25,7)
        \qbezier[120](11.5,0)(11.5,3.5)(11.5,7)
        \put(4.45,2.6){$\mathcal{V}$}
        \put(4.45,6.5){$h^ {-1}\mathcal{V}$}
        \thicklines
        \put(4.25,2.5){\line(1,0){7.25}}
        \put(4.25,2.2){\line(0,1){.6}}
        \put(4.25,2.2){\line(1,0){.1}}
        \put(4.25,2.8){\line(1,0){.1}}
        \put(11.5,2.2){\line(0,1){.6}}
        \put(11.5,2.2){\line(-1,0){.1}}
        \put(11.5,2.8){\line(-1,0){.1}}
    \end{picture}
    \end {center}
    \caption{A schematic of the phase space $\mathcal{M}$.
        Note that although the level set $\mathcal{M}_c=\set { h=c}$
        is depicted as
        a torus, \emph{it need not be a torus}.  It could be any compact, co-dimension
        $m$ submanifold.}\label{fig:phase_space}
\end{figure}

Note that the assumptions of regularity, existence of smooth
integrals, compactness, and preservation of smooth measures above
are not sufficient for the averaging principle to hold in any form.
As an example of just one possible obstruction, the level sets
$\mathcal{M}_c$ could separate into two completely disjoint sets,
$\mathcal{M}_c=\mathcal{M}_c^+\sqcup\mathcal{M}_c^-$. If this were
the case, then it would be implausible that the solutions of the
averaged vector field defined by averaging over all of
$\mathcal{M}_c$ would accurately describe $h_\varepsilon(t,z) $,
independent of whether $ z\in\mathcal{M}_c^+$ or $
z\in\mathcal{M}_c^-$.

\subsubsection{Some averaging results}

So far, we are in a general averaging setting.  Frequently, one also
assumes that the invariant submanifolds, $\mathcal{M}_h$, are tori,
and that there exists a choice of coordinates
\[
    z=(h,\varphi)
\]
on $\mathcal{M} $ in which the differential equation \eqref{eq:ode}
takes the form
\[
    \frac {dh} {dt}=\varepsilon H(h,\varphi,\varepsilon),\qquad
    \frac {d\varphi} {dt}=\Phi(h,\varphi,\varepsilon).
\]
Then if $\varphi\in S^1$ and the differential equation for the fast
variable is regular, i.e.~$\Phi(h,\varphi,0)$ is bounded away from
zero for $h\in\mathcal{U} $,
\[
    \sup_{\substack {\text {initial conditions}\\
    \text {s.t. } h_\varepsilon(0)\in  \mathcal{V} }}\;
    \sup_{0\leq\tau\leq 1\wedge
    T_\varepsilon}\abs{h_\varepsilon(\tau/\varepsilon)-\bar{h}(\tau)}
    =\mathcal{O} (\varepsilon )\text { as }\varepsilon\rightarrow 0.
\]
See for example Chapter 5 in \cite{SV85}, Chapter 3 in \cite{LM88},
or Theorem~\ref{thm:simple_averaging2} in the following section.

When the differential equation for the fast variable is not regular,
or when there is more than one fast variable, the typical averaging
result becomes much weaker than the uniform convergence above.  For
example, consider the case when $\varphi\in\mathbb{T}^n$, $n>1$, and
the unperturbed motion is quasi-periodic,
i.e.~$\Phi(h,\varphi,0)=\Omega(h) $.  Also assume that
$H\in\mathcal{C} ^ {n+2}$ and that $\Omega$ is nonvanishing and
satisfies a nondegeneracy condition on $\mathcal{U} $ (for example,
$\Omega :\mathcal{U}\rightarrow\mathbb{T}^n$ is a submersion).  Let
$P$ denote Riemannian volume on $\mathcal{M} $.
Neishtadt~\cite{LM88,Nei76} showed that in this situation, for each
fixed $\delta>0$,
\[
    P\left (\sup_{0\leq\tau\leq 1\wedge
    T_\varepsilon}\abs{h_\varepsilon(\tau/\varepsilon)-\bar{h}(\tau)}
    \geq \delta
    \right)
    =\mathcal{O} (\sqrt\varepsilon /\delta),
\]
and that this result is optimal.  Thus, the averaged equation only
describes the actual motions of the slow variables in probability on
the time scale $1/\varepsilon $ as $\varepsilon\rightarrow 0$.

Neishtadt's result was motivated by a general averaging theorem for
smooth systems due to Anosov. This theorem requires none of the
additional assumptions in the averaging results above. Under the
conditions of regularity, existence of smooth integrals,
compactness, and preservation of smooth measures, as well as
\begin {itemize}
    \item
        {\em Ergodicity:} for Lebesgue almost every
        $c\in\mathcal{U}$, $(z_0(\cdot),\mu_c) $ is ergodic,
\end {itemize}
Anosov showed that $\sup_{0\leq\tau\leq 1\wedge
T_\varepsilon}\abs{h_\varepsilon(\tau/\varepsilon)-\bar{h}(\tau)}\rightarrow
0$ in probability (w.r.t. Riemannian volume on initial conditions)
as $\varepsilon\rightarrow 0$, i.e.
\begin {thm}[Anosov's averaging theorem~\cite{Ano60}]
\label{thm:anosov}

For each $T>0$ and for each fixed $\delta>0$,
\[
    P\left (\sup_{0\leq\tau\leq T\wedge
    T_\varepsilon}\abs{h_\varepsilon(\tau/\varepsilon)-\bar{h}(\tau)}
    \geq \delta
    \right)
    \rightarrow 0
\]
as $\varepsilon\rightarrow 0$.

\end {thm}
\noindent We present a recent proof of this theorem in
Section~\ref{sct:Anosov_proof} below.

If we consider $h_\varepsilon(\cdot) $ and $\bar h(\cdot) $ to be
random variables, Anosov's theorem is a version of the weak law of
large numbers. In general, we can do no better: There is no general
strong law in this setting. There exists a simple example due to
Neishtadt (which comes from the equations for the motion of a
pendulum with linear drag being driven by a constant torque) where
for no initial condition in a positive measure set do we have
convergence of $h_\varepsilon(t)$ to $\bar{h}(\varepsilon t)$ on the
time scale $1/\varepsilon $ as $\varepsilon\rightarrow
0$~\cite{Kif04}. Here, the phase space is $\mathbb{R}\times S^1$,
and the unperturbed motion is (uniquely) ergodic on all but one
fiber.

\section{Some classical averaging results}

In this section we present some simple, well-known averaging
results. See for example Chapter 5 in~\cite{SV85} or Chapter 3
in~\cite{LM88}.

\subsection{Averaging for time-periodic vector fields}

Consider a family of time dependent ordinary differential equations
\begin {equation}\label {eq:ode1}
\frac {dh} {dt}=\varepsilon H(h,t,\varepsilon),
\end {equation}
indexed by the real parameter $\varepsilon\geq 0$, where
$h\in\mathbb{R}^m$. Fix $\mathcal{V}\subset\subset\mathcal{U}
\subset\mathbb{R}^m$, and suppose
\begin {itemize}
\item
     {\em  Regularity:}
     $H\in\mathcal{C}^1(\mathcal{U}\times \mathbb{R}\times [0,\infty))$.
\item
     {\em  Periodicity:} There exists $\mathcal{T} >0 $ such that
     for each $h\in\mathcal{U} $,
     $H(h,t,0)$ is $\mathcal{T}$-periodic in time.
\end {itemize}
Then
\[
    \frac {dh} {dt}=\varepsilon H(h,t,0) +\mathcal{O} (\varepsilon^2).
\]

Let $h_\varepsilon(t) $ denote the solution of Equation
\eqref{eq:ode1}.  We seek a time independent vector field whose
solutions approximate $h_\varepsilon(t) $, at least for a long
length of time.  It is natural to define the averaged vector field
$\bar{H} $ by
\begin{equation*}
    \bar {H} (h)
    =\frac {1} {\mathcal{T}}\int_{0}^{\mathcal{T}} H(h,s,0)ds.
\end{equation*}
Then $\bar {H} \in\mathcal{C}^1(\mathcal{U})$. Let $\bar{h} (\tau) $
be the solution of
\[
\frac {d\bar{h}}{d\tau} =\bar {H} (\bar {h}),\qquad \bar {h} (0)
=h_\varepsilon(0).
\]
It is reasonable to hope that $\bar h (\varepsilon t) $ and
$h_\varepsilon(t) $ are close together for $0\leq t\leq \varepsilon^
{-1} $.  We only consider the dynamics in a compact subset of phase
space, so for initial conditions in $\mathcal{U} $, we define the
stopping time
\[
    T_\varepsilon=\inf \{\tau\geq 0: \bar {h} (\tau)\notin \mathcal{V}
    \text { or } h_\varepsilon(\tau /\varepsilon) \notin \mathcal{V} \}.
\]

\begin {thm}[Time-periodic averaging]
\label{thm:simple_averaging1}

For each $T>0$,
\begin {equation*}
    \sup_{h_\varepsilon(0)\in  \mathcal{V} }\;
    \sup_{0\leq\tau\leq T\wedge
    T_\varepsilon}\abs{h_\varepsilon(\tau/\varepsilon)-\bar{h}(\tau)}
    =\mathcal{O} (\varepsilon )\text { as }\varepsilon\rightarrow 0.
\end {equation*}

\end {thm}

\begin {proof}

We divide our proof into three essential steps.

\paragraph*{Step 1:  Reduction using Gronwall's Inequality.}

Now, $\bar {h}(\tau) $ satisfies the integral equation
\[
\bar {h}(\tau) -\bar h(0) = \int_0^{\tau}\bar H(\bar
h(\sigma))d\sigma,
\]
while $h_\varepsilon(\tau/\varepsilon)$ satisfies
\[
\begin {split}
    h_\varepsilon(\tau/\varepsilon)-h_\varepsilon(0)
    &=
    \varepsilon\int_0^{\tau/\varepsilon}
    H(h_\varepsilon(s),s,\varepsilon)ds
    \\
    &=
    \mathcal{O}(\varepsilon) +\varepsilon\int_0^{\tau/\varepsilon}
    H(h_\varepsilon(s),s,0)ds\\
    &=\mathcal{O}(\varepsilon) +
    \varepsilon\int_0^{\tau/\varepsilon}
    H(h_\varepsilon(s),s,0)-
    \bar H(h_\varepsilon(s))ds+
    \int_0^{\tau}\bar H( h_\varepsilon(\sigma/\varepsilon))d\sigma
\end {split}
\]
for $0\leq\tau\leq T\wedge T_\varepsilon$.

Define
\[
    e_\varepsilon(\tau) =\varepsilon\int_0^{\tau/\varepsilon}
    H(h_\varepsilon(s),s,0)- \bar H(h_\varepsilon(s))ds.
\]
It follows from Gronwall's Inequality that
\begin {equation*}
    \sup_{0\leq \tau\leq T\wedge T_\varepsilon}
    \abs{\bar h(\tau)-h_\varepsilon(\tau/\varepsilon)}\leq
    \left(\mathcal{O}(\varepsilon)+
    \sup_{0\leq \tau\leq T\wedge T_\varepsilon}
    \abs{e_\varepsilon(\tau)}\right)e^{ \Lip{\bar
    H\arrowvert _\mathcal{V}}T}.
\end {equation*}

\paragraph*{Step 2:  A sequence of times adapted for ergodization.}

\emph{Ergodization} refers to the convergence along an orbit of a
function's time average to its space average.  We define a sequence
of times $t_{k} $ for $k\geq 0$ by $t_k=k\mathcal{T}$.  This
sequence of times is motivated by the fact that
\[
    \frac {1} {t_{k+1}-t_k}
    \int_{t_k} ^{t_{ k+1}} H(h_0(s),s,0)ds=\bar H(h_0).
\]
Note that $h_0(t) $ is independent of time. Thus,
\begin {equation}
\label {eq:simple_avg1}
    \sup_{0\leq \tau\leq
    T\wedge T_\varepsilon}
    \abs{e_\varepsilon(\tau)}
    \leq
    \mathcal{O} (\varepsilon)+
    \varepsilon\sum_{t_{k+1}\leq
    \frac {T\wedge T_\varepsilon}{\varepsilon}}
    \abs{
    \int_{t_{k}}^{t_{k+1}}
    H(h_\varepsilon(s),s,0)-\bar H(h_\varepsilon(s))ds}.
\end {equation}

\paragraph*{Step 3:  Control of individual terms by comparison with
        solutions of the $\varepsilon=0$ equation.}

The sum in Equation \eqref{eq:simple_avg1} has no more than
$\mathcal{O} (1/\varepsilon)$ terms, and so it suffices to show that
each term $\int_{t_{k}}^{t_{k+1}}H(h_\varepsilon(s),s,0)-\bar
H(h_\varepsilon(s))ds$ is no larger than $\mathcal{O} (\varepsilon)
$.  We can accomplish this by comparing the motions of
$h_\varepsilon(t)$ for $t_{k}\leq t\leq t_{k+1}$ with
$h_{k,\varepsilon} (t) $, which is defined to be the solution of the
$\varepsilon=0$ ordinary differential equation satisfying
$h_{k,\varepsilon} (t_k)=h_\varepsilon(t_k)$, i.e.
$h_{k,\varepsilon} (t)\equiv h_\varepsilon(t_k)$.

\begin {lem}

If $t_{k+1}\leq\frac{T\wedge T_\varepsilon}{\varepsilon}$, then $
    \sup_{t_{k}\leq t\leq t_{k+1}}\abs{h_{k,\varepsilon}(t)-
    h_\varepsilon(t)}=\mathcal{O}(\varepsilon ).$
\end {lem}

\begin {proof}
$dh_\varepsilon/dt=\mathcal{O}(\varepsilon)$.

\end {proof}

Using that $H$ and $\bar H$ are Lipschitz continuous, we conclude
that
\[
\begin {split}
    \int_{t_{k}}^{t_{k+1}}
    &
    H(h_\varepsilon(s),s,0)-\bar H(h_\varepsilon(s))ds
    \\
    =&
    \int_{t_{k}}^{t_{k+1}}
    H(h_\varepsilon(s),s,0)-H(h_{k,\varepsilon}(s),s,0)ds
    \\&+
    \int_{t_{k}}^{t_{k+1}}
    H(h_{k,\varepsilon}(s),s,0)-\bar H(h_{k,\varepsilon}(s))ds
    \\
    &
    +
    \int_{t_{k}}^{t_{k+1}}
    \bar H(h_{k,\varepsilon}(s))-\bar H(h_\varepsilon(s))ds
    \\
    =&
    \mathcal{O} (\varepsilon) +0+\mathcal{O} (\varepsilon)\\
    =&
    \mathcal{O} (\varepsilon).
\end {split}
\]

Thus we see that $
    \sup_{0\leq \tau\leq T\wedge T_\varepsilon}
    \abs{h_\varepsilon(\tau/\varepsilon)-\bar
    h(\tau)}
    \leq\mathcal{O}(\varepsilon ),
$ independent of the initial condition
$h_\varepsilon(0)\in\mathcal{V}$.

\end {proof}

\begin {rem}
Note that the $\mathcal{O} (\varepsilon) $ control in
Theorem~\ref{thm:simple_averaging1} on a time scale $t=\mathcal{O}
(\varepsilon^ {-1}) $ is generally optimal. For example, take
$H(h,t,\varepsilon) =\cos(t) +\varepsilon$.

\end {rem}

\subsection{Averaging for vector fields with one regular fast variable}

For $h\in\mathbb{R}^m$ and $\varphi\in S^1= [0,1]/0\sim 1$, consider
the family of ordinary differential equations
\begin {equation}\label {eq:ode2}
\frac {dh} {dt}=\varepsilon H(h,\varphi,\varepsilon),\qquad \frac
{d\varphi} {dt}=\Phi(h,\varphi,\varepsilon),
\end {equation}
indexed by the real parameter $\varepsilon\geq 0$. With
$z=(h,\varphi) $, we write this family of differential equations as
$dz/dt=Z(z,\varepsilon)$.

Fix $\mathcal{V}\subset\subset\mathcal{U} \subset\mathbb{R}^m$, and
suppose
\begin {itemize}
\item
     {\em  Regularity:}
     $Z\in\mathcal{C}^1(\mathcal{U}\times S^1\times [0,\infty))$.
\item
     {\em  Regular fast variable:} $\Phi(h,\varphi,0)$ is bounded away
     from $0$
     for $h\in\mathcal{U} $, i.e.
     \[
        \inf_{(h,\varphi)\in\mathcal{U}\times S^1} \abs{\Phi(h,\varphi,0)} >0.
     \]
     Without loss of generality, we assume
     that $\Phi(h,\varphi,0)> 0$.
\end {itemize}

Let $z_\varepsilon(t) = (h_\varepsilon(t),\varphi_\varepsilon(t)) $
denote the solution of Equation \eqref{eq:ode2}.  Then $z_0(t) $
leaves invariant the circles $\mathcal{M}_c=\{h=c\} $ in phase
space.  In fact, $z_0(t) $ preserves an uniquely ergodic invariant
probability measure on $\mathcal{M}_c$, whose density is given by
\[
    d\mu_c=\frac{1}{K_c}\frac{d\varphi}{\Phi(c,\varphi,0)},
\]
where $K_c=\int_0^1 \frac{d\varphi}{\Phi(c,\varphi,0)}$ is a
normalization constant.

The averaged vector field $\bar{H} $ is defined by averaging
$H(h,\varphi,0) $ over $\varphi$:
\begin{equation*}
    \bar {H} (h)
    =\int_{0}^1 H(h,\varphi,0)d\mu_h(\varphi)
    =\frac{1}{K_h}\int_{0}^1 \frac{H(h,\varphi,0)}{\Phi(h,\varphi,0)}d\varphi.
\end{equation*}
Then $\bar {H} \in\mathcal{C}^1(\mathcal{U})$. Let $\bar{h} (\tau) $
be the solution of
\[
\frac {d\bar{h}}{d\tau} =\bar {H} (\bar {h}),\qquad \bar {h} (0)
=h_\varepsilon(0).
\]
For initial conditions in $\mathcal{U}\times S^1 $, we have the
usual stopping time $T_\varepsilon=\inf \{\tau\geq 0: \bar {h}
(\tau)\notin \mathcal{V} \text { or } h_\varepsilon(\tau
/\varepsilon) \notin \mathcal{V} \}$.

\begin {thm}[Averaging over one regular fast variable]
\label{thm:simple_averaging2}

For each $T>0$,
\begin {equation*}
    \sup_{\substack {\text {initial conditions}\\
    \text {s.t. } h_\varepsilon(0)\in  \mathcal{V} }}\;
    \sup_{0\leq\tau\leq T\wedge
    T_\varepsilon}\abs{h_\varepsilon(\tau/\varepsilon)-\bar{h}(\tau)}
    =\mathcal{O} (\varepsilon )\text { as }\varepsilon\rightarrow 0.
\end {equation*}

\end {thm}

\begin {rem}
This result encompasses Theorem~\ref{thm:simple_averaging1} for
time-periodic averaging. For example, if $\mathcal{T} =1$, simply
take $\varphi=t \text { mod } 1$ and
$\Phi(h,\varphi,\varepsilon)=1$.
\end {rem}

\begin {rem}
Many of the proofs of the above theorem of which we are aware hinge
on considering $\varphi$ as a time-like variable.  For example, one
could write
\[
    \frac{dh}{d\varphi}=\frac{dh}{dt}\frac{dt}{d\varphi}
    =\varepsilon\frac{H(h,\varphi,0)}{\Phi(h,\varphi,0)}
    +\mathcal{O}(\varepsilon^2),
\]
and this looks very similar to the time-periodic situation
considered previously.  However, it does take some work to justify
such arguments rigorously, and the traditional proofs do not easily
generalize to averaging over multiple fast variables.  Our proof
essentially uses $\varphi$ to mark off time, and it will immediately
generalize to a specific instance of multiphase averaging.
\end {rem}

\begin {proof}

Again, we have three steps.

\paragraph*{Step 1:  Reduction using Gronwall's Inequality.}

Now
\[
\bar {h}(\tau) -\bar h(0) = \int_0^{\tau}\bar H(\bar
h(\sigma))d\sigma,
\]
and
\[
\begin {split}
    h_\varepsilon(\tau/\varepsilon)-h_\varepsilon(0)
    &=
    \varepsilon\int_0^{\tau/\varepsilon}
    H(z_\varepsilon(s),\varepsilon)ds
    =
    \mathcal{O}(\varepsilon) +\varepsilon\int_0^{\tau/\varepsilon}
    H(z_\varepsilon(s),0)ds\\
    &=\mathcal{O}(\varepsilon) +
    \varepsilon\int_0^{\tau/\varepsilon}
    H(z_\varepsilon(s),0)-
    \bar H(h_\varepsilon(s))ds+
    \int_0^{\tau}\bar H( h_\varepsilon(\sigma/\varepsilon))d\sigma
\end {split}
\]
for $0\leq\tau\leq T\wedge T_\varepsilon$.

Define
\[
    e_\varepsilon(\tau) =\varepsilon\int_0^{\tau/\varepsilon}
    H(z_\varepsilon(s),0)- \bar H(h_\varepsilon(s))ds.
\]
It follows from Gronwall's Inequality that
\begin {equation*}
    \sup_{0\leq \tau\leq T\wedge T_\varepsilon}
    \abs{\bar h(\tau)-h_\varepsilon(\tau/\varepsilon)}\leq
    \left(\mathcal{O}(\varepsilon)+
    \sup_{0\leq \tau\leq T\wedge T_\varepsilon}
    \abs{e_\varepsilon(\tau)}\right)e^{ \Lip{\bar
    H\arrowvert _\mathcal{V}}T}.
\end {equation*}

\paragraph*{Step 2:  A sequence of times adapted for ergodization.}

Now for each initial condition in our phase space and for each fixed
$\varepsilon$, we define a sequence of times $t_{k,\varepsilon} $
and a sequence of solutions $z_{k,\varepsilon}(t) $ inductively as
follows: $t_{0,\varepsilon} =0$ and $z_{0,\varepsilon}(t)=z_{0}(t)
$.  For $k>0$, $t_{k,\varepsilon} =\inf\{t>t_{k-1,\varepsilon}:
\varphi_{k-1,\varepsilon}(t)=\varphi_\varepsilon(0)\}$, and
$z_{k,\varepsilon}(t) $ is defined as the solution of
\[
    \frac{dz_{k,\varepsilon}}{dt}=Z(z_{k,\varepsilon},0)
    =(0,\Phi(z_{k,\varepsilon},0)),
    \qquad
    z_{k,\varepsilon}(t_{k,\varepsilon})=z_{\varepsilon}(t_{k,\varepsilon}).
\]
This sequence of times is motivated by the fact that
\[
    \frac {1} {t_{k+1,\varepsilon}-t_{k,\varepsilon}}
    \int_{t_k,\varepsilon} ^{t_{ k+1,\varepsilon}}
    H(z_{k,\varepsilon}(s),0)ds=\bar H(h_{k,\varepsilon}).
\]
Recall that $h_{k,\varepsilon}(t) $ is independent of time. The
elements of this sequence of times are approximately uniformly
spaced, i.e.~if we fix $\omega >0$ such that $z\in\mathcal{V}\times
S^1 \Rightarrow 1/\omega <\Phi(z,0)<\omega$, then if $t_{
k+1,\varepsilon}\leq(T\wedge T_\varepsilon)/\varepsilon$, $1/\omega
<t_{ k+1,\varepsilon}-t_{ k,\varepsilon}<\omega$.

Thus,
\begin {equation*}
    \sup_{0\leq \tau\leq
    T\wedge T_\varepsilon}
    \abs{e_\varepsilon(\tau)}
    \leq
    \mathcal{O} (\varepsilon)+
    \varepsilon\sum_{t_{k+1,\varepsilon}\leq
    \frac {T\wedge T_\varepsilon}{\varepsilon}}
    \abs{
    \int_{t_{k,\varepsilon}}^{t_{k+1,\varepsilon}}
    H(z_\varepsilon(s),0)-\bar H(h_\varepsilon(s))ds},
\end {equation*}
where the sum in in this equation has no more than $\mathcal{O}
(1/\varepsilon)$ terms.

\paragraph*{Step 3:  Control of individual terms by comparison with
        solutions along fibers.}

It suffices to show that each term
$\int_{t_{k,\varepsilon}}^{t_{k+1,\varepsilon}}H(z_\varepsilon(s),0)-\bar
H(h_\varepsilon(s))ds$ is no larger than $\mathcal{O} (\varepsilon)
$.  We can accomplish this by comparing the motions of
$z_\varepsilon(t)$ for $t_{k,\varepsilon}\leq t\leq
t_{k+1,\varepsilon}$ with $z_{k,\varepsilon} (t) $.

\begin {lem}

If $t_{k+1,\varepsilon}\leq\frac{T\wedge
T_\varepsilon}{\varepsilon}$, then $
    \sup_{t_{k,\varepsilon}\leq t\leq t_{k+1,\varepsilon}}\abs{z_{k,\varepsilon}(t)-
    z_\varepsilon(t)}=\mathcal{O}(\varepsilon ).$
\end {lem}

\begin {proof}

Without loss of generality, we take $k=0$, so that
$z_{k,\varepsilon}(t)=z_{0}(t)$. Since $h_0(t) =h_\varepsilon(0) $
and $dh_\varepsilon/dt=\mathcal{O}(\varepsilon)$,
$\sup_{t_{0,\varepsilon}\leq t\leq t_{1,\varepsilon}}\abs{h_{0}(t)-
h_\varepsilon(t)}=\mathcal{O}(\varepsilon )$.

Now $\varphi_\varepsilon(t) -\varphi_\varepsilon(0) =\int_0^t \Phi
(h_\varepsilon(s),\varphi_\varepsilon(s),\varepsilon)ds $, and
because $\Phi$ is Lipschitz, we find that
\[
    \abs{\varphi_\varepsilon(t) -\varphi_0(t)}\leq \mathcal{O}
    (\varepsilon) +\Lip{\Phi}\int_0^t\abs{\varphi_\varepsilon(s)
    -\varphi_0(s)}ds
\]
for $0\leq t\leq \omega $. The result follows from Gronwall's
Inequality.

\end {proof}

Using that $H$ and $\bar H$ are Lipschitz continuous, we conclude
that
\[
\begin {split}
    \int_{t_{k,\varepsilon}}^{t_{k+1,\varepsilon}}
    &
    H(z_\varepsilon(s),0)-\bar H(h_\varepsilon(s))ds
    \\
    =&
    \int_{t_{k,\varepsilon}}^{t_{k+1,\varepsilon}}
    H(z_\varepsilon(s),0)-H(z_{k,\varepsilon}(s),0)ds
    \\&+
    \int_{t_{k,\varepsilon}}^{t_{k+1,\varepsilon}}
    H(z_{k,\varepsilon}(s),0)-\bar H(h_{k,\varepsilon}(s))ds
    \\
    &
    +
    \int_{t_{k,\varepsilon}}^{t_{k+1,\varepsilon}}
    \bar H(h_{k,\varepsilon}(s))-\bar H(h_\varepsilon(s))ds
    \\
    =&
    \mathcal{O} (\varepsilon) +0+\mathcal{O} (\varepsilon)
    \\=&
    \mathcal{O} (\varepsilon).
\end {split}
\]

Thus we see that $
    \sup_{0\leq \tau\leq T\wedge T_\varepsilon}
    \abs{h_\varepsilon(\tau/\varepsilon)-\bar
    h(\tau)}
    =\mathcal{O}(\varepsilon ),
$ independent of the initial condition
$(h_\varepsilon(0),\varphi_\varepsilon(0))\in\mathcal{V}\times S^1$.

\end {proof}

\subsection{Multiphase averaging for vector fields with separable,
regular fast variables}

As explained in Section~\ref{sct:anos}, when the differential
equation for the fast variable is not regular, or when there is more
than one fast variable, the typical averaging result becomes much
weaker than the uniform convergence in Theorems
\ref{thm:simple_averaging1} and \ref{thm:simple_averaging2} above.
Nonetheless, if the differential equations under consideration
satisfy some very specific hypotheses, the proof in the previous
section immediately generalizes to yield uniform convergence.

For $h\in\mathbb{R}^m$ and $\varphi=(\varphi^1,\cdots,\varphi^n)\in
\mathbb{T}^n= ([0,1]/0\sim 1) ^n$, consider the family of ordinary
differential equations
\begin {equation}\label {eq:ode3}
\frac {dh} {dt}=\varepsilon H(h,\varphi,\varepsilon),\qquad \frac
{d\varphi} {dt}=\Phi(h,\varphi,\varepsilon),
\end {equation}
indexed by the real parameter $\varepsilon\geq 0$. We also write
$z=(h,\varphi) $ and $dz/dt=Z(z,\varepsilon)$.

Fix $\mathcal{V}\subset\subset\mathcal{U} \subset\mathbb{R}^m$, and
suppose
\begin {itemize}
\item
     {\em  Regularity:}
     $Z\in\mathcal{C}^1(\mathcal{U}\times \mathbb{T}^n\times [0,\infty))$.
\item
     {\em  Separable fast variables:} $H(h,\varphi,0)$ and $\Phi(h,\varphi,0)$
     have the following specific forms:
     \begin {itemize}
     \item
        There exist $\mathcal{C}^1$ functions $H_j(h,\varphi^j)$ such that
        $H(h,\varphi,0)=\sum_{j=1}^{n}H_j(h,\varphi^j)$.  This can
        be thought of as saying that, to first order in $\varepsilon$, each fast
        variable affects the slow variables independently of the
        other fast variables.
     \item
        The components $\Phi^j$ of $\Phi$ satisfy
        $\Phi^j(h,\varphi,0)=\Phi^j(h,\varphi^j,0)$, i.e.~the
        unperturbed motion has each fast variable moving
        independently of the other fast variables.  Note that
        this assumption is satisfied if the unperturbed motion is quasi-periodic,
        i.e.~$\Phi(h,\varphi,0)=\Omega(h) $.
     \end {itemize}
\item
     {\em  Regular fast variables:} For each $j$,
     \[
        \inf_{(h,\varphi^j)\in\mathcal{U}\times S^1} \abs{\Phi^j(h,\varphi^j,0)} >0.
     \]
\end {itemize}

Let $z_\varepsilon(t) = (h_\varepsilon(t),\varphi_\varepsilon(t)) $
denote the solution of Equation \eqref{eq:ode3}.  Then $z_0(t) $
leaves invariant the tori $\mathcal{M}_c=\{h=c\} $ in phase space.
In fact, $z_0(t) $ preserves a (not necessarily ergodic) invariant
probability measure on $\mathcal{M}_c$, whose density is given by
\[
    d\mu_c=\prod_{ j=1} ^n
    \frac{1}{K_c^j}\frac{d\varphi^j}{\abs{\Phi^j(c,\varphi^j,0)}},
\]
where $K_c^j=\int_0^1
\frac{d\varphi^j}{\abs{\Phi^j(c,\varphi^j,0)}}$.

The averaged vector field $\bar{H} $ is defined by
\begin{equation*}
\begin {split}
    \bar {H} (h)
    &=\int_{\mathcal{M}_c} H(h,\varphi,0)d\mu_h(\varphi)
    =\sum_{j=1} ^n \int_{\mathcal{M}_c} H_j(h,\varphi^j)d\mu_h(\varphi)
    \\
    &=\sum_{j=1} ^n \frac{1}{K_h^j}\int_{0}^1
        \frac{H_j(h,\varphi^j)}{\abs{\Phi^j(h,\varphi^j,0)}}d\varphi^j
    : = \sum_{j=1} ^n \bar {H}_j (h).
\end {split}
\end{equation*}
Let $\bar{h} (\tau) $ be the solution of
\[
\frac {d\bar{h}}{d\tau} =\bar {H} (\bar {h}),\qquad \bar {h} (0)
=h_\varepsilon(0),
\]
and the stopping time $T_\varepsilon=\inf \{\tau\geq 0: \bar {h}
(\tau)\notin \mathcal{V} \text { or } h_\varepsilon(\tau
/\varepsilon) \notin \mathcal{V} \}$.

\begin {thm}[Averaging over multiple separable, regular fast variables]
\label{thm:simple_averaging3}

For each $T>0$,
\begin {equation*}
    \sup_{\substack {\text {initial conditions}\\
    \text {s.t. } h_\varepsilon(0)\in  \mathcal{V} }}\;
    \sup_{0\leq\tau\leq T\wedge
    T_\varepsilon}\abs{h_\varepsilon(\tau/\varepsilon)-\bar{h}(\tau)}
    =\mathcal{O} (\varepsilon )\text { as }\varepsilon\rightarrow 0.
\end {equation*}

\end {thm}

\begin {proof}

The proof is essentially the same as the proof of Theorem
\ref{thm:simple_averaging2}.  As before, we need only show that
$\sup_{0\leq \tau\leq T\wedge
T_\varepsilon}\abs{e_\varepsilon(\tau)}=\mathcal{O}(\varepsilon)$,
where
\[
    e_\varepsilon(\tau) =\varepsilon\int_0^{\tau/\varepsilon}
    H(z_\varepsilon(s),0)- \bar H(h_\varepsilon(s))ds.
\]
But by our separability assumptions, it suffices to show that for
each $j$,
\[
    \sup_{0\leq \tau\leq T\wedge
    T_\varepsilon}\abs{e_{j,\varepsilon}(\tau)}=\mathcal{O}(\varepsilon),
\]
where $e_{j,\varepsilon}(\tau)$ is defined by
\[
    e_{j,\varepsilon}(\tau) =\varepsilon\int_0^{\tau/\varepsilon}
    H_j(h_\varepsilon(s),\varphi_\varepsilon^j(s))- \bar H_j(h_\varepsilon(s))ds.
\]
Thus, we have effectively separated the effects of each fast
variable, and now the proof can be completed by essentially
following steps 2 and 3 in the proof of Theorem
\ref{thm:simple_averaging2}.

\end {proof}

\section{A proof of Anosov's theorem}
\label{sct:Anosov_proof}

Anosov's original proof of Theorem~\ref{thm:anosov} from 1960 may be
found in~\cite {Ano60}. An exposition of the theorem and Anosov's
proof in English may be found in~\cite {LM88}.  Recently,
Kifer~\cite{Kif04b} proved necessary and sufficient conditions for
the averaging principle to hold in an averaged with respect to
initial conditions sense.  He also showed explicitly that his
conditions are met in the setting of Anosov's theorem.  The proof of
Anosov's theorem given here is mainly due to Dolgopyat~\cite
{Dol05}, although some further simplifications have been made.

\begin{proof}[Proof of Anosov's theorem]
We begin by showing that without loss of generality we may take
$T_\varepsilon=\infty$. This is just for convenience, and not an
essential part of the proof.  To accomplish this, let $\psi(h) $ be
a  smooth bump function satisfying
\begin {itemize}
    \item
         $\psi(h)=1 \text { if } h\in\mathcal{V}$,
    \item
        $\psi(h)>0 \text { if } h\in\text
        {interior}(\mathcal{\tilde{V}})$,
    \item
         $\psi(h)=0 \text { if } h\notin\mathcal{\tilde{V}}$,
\end {itemize}
where $\mathcal{\tilde{V}}$ is a compact set chosen such that
$\mathcal{V}\subset\subset \text
{interior}(\mathcal{\tilde{V}})\subset\subset\mathcal{U}$. Next, set
$\tilde Z(z,\varepsilon)=\psi(h(z))Z(z,\varepsilon)$. Because the
bump function was chosen to depend only on the slow variables, our
assumption about preservation of measures is still satisfied; on
each fiber, $\tilde Z(z,0)$ is a scaler multiple of $Z(z,0)$.
Furthermore, the flow of $\tilde
Z(\cdot,0)\arrowvert_{\mathcal{M}_h} $ is ergodic for almost every
$h\in\mathcal{\tilde V}$. Then it would suffice to prove our theorem
for the vector fields $\tilde {Z}(z,\varepsilon)$ with the set
$\mathcal{\tilde{V}}$ replacing $\mathcal{V} $.  We assume that this
reduction has been made, although we will not use it until Step 5
below.

\paragraph*{Step 1:  Reduction using Gronwall's Inequality.}

Observe that $\bar {h}(\tau) $ satisfies the integral equation
\[
\bar {h}(\tau) -\bar h(0) = \int_0^{\tau}\bar H(\bar
h(\sigma))d\sigma,
\]
while $h_\varepsilon(\tau/\varepsilon)$ satisfies
\[
\begin {split}
    h_\varepsilon(\tau/\varepsilon)-h_\varepsilon(0)&=
    \varepsilon\int_0^ {\tau/\varepsilon}
    H(z_\varepsilon(s),\varepsilon)ds
    \\
    &=\mathcal{O}(\varepsilon) +\varepsilon\int_0^{\tau/\varepsilon}
    H(z_\varepsilon(s),0)ds\\
    &=\mathcal{O}(\varepsilon) +
    \varepsilon\int_0^{\tau/\varepsilon}
    H(z_\varepsilon(s),0)-
    \bar H(h_\varepsilon(s))ds+
    \int_0^{\tau}\bar H( h_\varepsilon(\sigma/\varepsilon))d\sigma
\end {split}
\]
for $0\leq\tau\leq T\wedge T_\varepsilon$.  Here we have used the
fact that $h^ {-1} \mathcal{V}\times [0,\varepsilon_0] $ is compact
to achieve uniformity over all initial conditions in the size of the
$\mathcal{O}(\varepsilon) $ term above. We use this fact repeatedly
in what follows.  In particular, $H$, $\bar H$, and $Z$ are
uniformly bounded and have uniform Lipschitz constants on the
domains of interest.

Define
\[
    e_\varepsilon(\tau) =\varepsilon\int_0^{\tau/\varepsilon}
    H(z_\varepsilon(s),0)- \bar H(h_\varepsilon(s))ds.
\]
It follows from Gronwall's Inequality that
\begin {equation}\label {eq:hinfnorm}
    \sup_{0\leq \tau\leq T\wedge T_\varepsilon}
    \abs{\bar h(\tau)-h_\varepsilon(\tau/\varepsilon)}\leq
    \left(\mathcal{O}(\varepsilon)+
    \sup_{0\leq \tau\leq T\wedge T_\varepsilon}
    \abs{e_\varepsilon(\tau)}\right)e^{ \Lip{\bar
    H\arrowvert_\mathcal{V}}T}.
\end {equation}

\paragraph*{Step 2:  Introduction of a time scale for ergodization.}

Choose a real-valued function $L(\varepsilon) $ such that
$L(\varepsilon)\rightarrow\infty,\: L(\varepsilon)=\text {o}
(\log\varepsilon^ {-1})$ as $\varepsilon\rightarrow 0$. Think of
$L(\varepsilon) $ as being a time scale which grows as
$\varepsilon\rightarrow 0$ so that \emph{ergodization}, i.e.~the
convergence along an orbit of a function's time average to a space
average, can take place. However, $L(\varepsilon) $ doesn't grow too
fast, so that on this time scale $z_\varepsilon(t) $ essentially
stays on one fiber, where we have our ergodicity assumption.  Set
$t_{k,\varepsilon}=kL(\varepsilon) $, so that
\begin {equation}\label {eq:einfnorm}
    \sup_{0\leq \tau\leq T\wedge T_\varepsilon}\abs{e_\varepsilon(\tau)}
    \leq \mathcal{O}(\varepsilon L(\varepsilon))+
    \varepsilon\sum_{k=0}^{\frac{T\wedge T_\varepsilon}{\varepsilon
    L(\varepsilon)}-1}\abs{\int_{t_{k,\varepsilon}}^{t_{k+1,\varepsilon}}H(z_\varepsilon(s),0)-\bar
    H(h_\varepsilon(s))ds}.
\end {equation}

\paragraph*{Step 3:  A splitting for using the triangle inequality.}

Now we let $z_{k,\varepsilon} (s) $ be the solution of
\[
    \frac {dz_{k,\varepsilon}}{dt}=Z(z_{k,\varepsilon},0),
    \qquad z_{k,\varepsilon}(t_{k,\varepsilon})=z_{\varepsilon}(t_{k,\varepsilon}).
\]
Set $h_{k,\varepsilon} (t) =h (z_{k,\varepsilon} (t)) $. Observe
that $h_{k,\varepsilon} (t) $ is independent of $t $. We break up
the integral
$\int_{t_{k,\varepsilon}}^{t_{k+1,\varepsilon}}H(z_\varepsilon(s),0)
-\bar H(h_\varepsilon(s))ds$ into three parts:
\[
\begin {split}
    \int_{t_{k,\varepsilon}}^{t_{k+1,\varepsilon}}
    &
    H(z_\varepsilon(s),0)
    -\bar H(h_\varepsilon(s))ds
    \\
    =&
    \int_{t_{k,\varepsilon}}^{t_{k+1,\varepsilon}}H(z_\varepsilon(s),0) -
    H(z_{k,\varepsilon}(s),0)ds
    \\
    &+\int_{t_{k,\varepsilon}}^{t_{k+1,\varepsilon}} H(z_{k,\varepsilon}(s),0)-
     \bar H(h_{k,\varepsilon}(s))ds\\
    &+\int_{t_{k,\varepsilon}}^{t_{k+1,\varepsilon}} \bar H(h_{k,\varepsilon}(s))ds-
    \bar H(h_{\varepsilon}(s))ds\\
    :=&
    I_{k,\varepsilon}+II_{k,\varepsilon}+III_{k,\varepsilon}.
\end {split}
\]

The term $II_{k,\varepsilon}$ represents an ``ergodicity term'' that
can be controlled by our assumptions on the ergodicity of the flow
$z_0(t) $, while the terms $I_{k,\varepsilon}$ and
$III_{k,\varepsilon}$ represent ``continuity terms'' that can be
controlled using the following control on the drift from solutions
along fibers.

\paragraph*{Step 4:  Control of drift from solutions along fibers.}

\begin {lem}
\label {lem:anosov1}

If $0<t_{k+1,\varepsilon}\leq\frac{T\wedge
T_\varepsilon}{\varepsilon}$,
\[
    \sup_{t_{k,\varepsilon}\leq t\leq t_{k+1,\varepsilon}}\abs{z_{k,\varepsilon}(t)-
    z_\varepsilon(t)}\leq \mathcal{O}(\varepsilon L(\varepsilon)
    e^{ \Lip{Z} L(\varepsilon)})
\]
\end {lem}
\begin {proof}
Without loss of generality we may set $k=0$, so that
$z_{k,\varepsilon}(t)=z_0(t)$. Then for $0\leq t\leq
L(\varepsilon)$,
\[
\begin {split}
    \abs{z_0(t)-z_\varepsilon(t)} &= \abs{\int_0^t Z(z_0(s),0)
    -Z(z_\varepsilon(s),\varepsilon)ds}
    \\
    &\leq \Lip{Z}\int_0^t \abs{\varepsilon}
    +\abs{z_0(s)-z_\varepsilon(s)}ds\\
    &=\mathcal{O}(\varepsilon L(\varepsilon))+
    \Lip{Z}\int_0^t
    \abs{z_0(s)-z_\varepsilon(s)}ds.
\end {split}
\]
The result follows from Gronwall's Inequality.
\end {proof}

From Lemma \ref{lem:anosov1} we find that
$I_{k,\varepsilon},III_{k,\varepsilon}=\mathcal{O}(\varepsilon
L(\varepsilon)^2 e^{\Lip{Z} L(\varepsilon)})$.

\paragraph*{Step 5:  Use of ergodicity along fibers to
control $II_{k,\varepsilon} $.}

From Equations \eqref{eq:hinfnorm} and \eqref{eq:einfnorm} and the
triangle inequality, we already know that
\begin{equation}\label{eq:hinfnorm2}
\begin {split}
    \sup_{0\leq \tau\leq T\wedge T_\varepsilon}
    &
    \abs{\bar h(\tau)-h_\varepsilon(\tau/\varepsilon)}
    \\
    &\leq
    \mathcal{O}(\varepsilon)+
    \mathcal{O}(\varepsilon L(\varepsilon))+
    \varepsilon \frac{T}{\varepsilon L(\varepsilon)}
    \mathcal{O}(\varepsilon L(\varepsilon)^2
    e^{\Lip{Z} L(\varepsilon)})
    +\mathcal{O}\left (\varepsilon
    \sum_{k=0}^{\frac{T\wedge T_\varepsilon}{\varepsilon
    L(\varepsilon)}-1}\abs{II_{k,\varepsilon}}\right)
    \\
    & =
    \mathcal{O}(\varepsilon L(\varepsilon) e^{\Lip{Z}
    L(\varepsilon)})+
    \mathcal{O}\left (\varepsilon
    \sum_{k=0}^{\frac{T\wedge T_\varepsilon}{\varepsilon
    L(\varepsilon)}-1}\abs{II_{k,\varepsilon}}\right).
\end {split}
\end {equation}
Fix $\delta > 0$.  Recalling that $T_\varepsilon=\infty $, it
suffices to show that
\[
    P\left(\varepsilon\sum_{k=0}^{\frac{T}{\varepsilon
    L(\varepsilon)}-1}\abs{II_{k,\varepsilon}}\geq
    \delta\right)\rightarrow 0
\]
as $\varepsilon\rightarrow 0$.

For initial conditions $z\in \mathcal{M}$ and for $0\leq
k\leq\frac{T}{\varepsilon L(\varepsilon)}$ define
\[
\begin {split}
    \mathcal{B}_{k,\varepsilon} & =\set{z:\frac{1}{L(\varepsilon)}
    \abs{II_{k,\varepsilon}}
    >\frac{\delta}{2T}},\\
    \mathcal{B}_{z,\varepsilon} & =\set{k:z\in
    \mathcal{B}_{k,\varepsilon}} .
\end {split}
\]
Think of these sets as describing ``bad ergodization.''  For
example, roughly speaking, $z\in\mathcal{B}_{k,\varepsilon}$ if the
orbit $z_{\varepsilon}(t)$ starting at $z$ spends the time between
$t_{k,\varepsilon}$ and $t_{k+1,\varepsilon} $ in a region of phase
space where the function $H(\cdot,0) $ is ``poorly ergodized'' on
the time scale $L(\varepsilon) $ by the flow $z_0(t) $ (as measured
by the parameter $\delta/2T$).  As $II_{k,\varepsilon}$ is clearly
never larger than $\mathcal{O} (L(\varepsilon))$, it follows that
\[
    \varepsilon
    \sum_{k=0}^{\frac{T}{\varepsilon
    L(\varepsilon)}-1}\abs{II_{k,\varepsilon}}\leq
    \frac{\delta}{2}+\mathcal{O}(\varepsilon L(\varepsilon)
    \# (\mathcal{B}_{z,\varepsilon})).
\]
Therefore it suffices to show that
\[
    P\left(\#(\mathcal{B}_{z,\varepsilon})\geq\frac{\delta}{\C\, \varepsilon
    L(\varepsilon)}\right)\rightarrow 0
\]
as $\varepsilon\rightarrow 0$. By Chebyshev's Inequality, we need
only show that
\[
    E (\varepsilon L(\varepsilon)\# (\mathcal{B}_{z,\varepsilon})) =
    \varepsilon L(\varepsilon)\sum_{k=0}^{\frac{T}{\varepsilon
    L(\varepsilon)}-1}P(\mathcal{B}_{k,\varepsilon})
\]
tends to $0$ with $\varepsilon$.

In order to estimate the size of $P(\mathcal{B}_{k,\varepsilon})$,
it is convenient to introduce a new measure $P^f$ that is uniformly
equivalent to the restriction of Riemannian volume $P$ to $h^
{-1}\mathcal{V} $.  Here the $f$ stands for ``factor,'' and $P^f$ is
defined by
\[
    dP^f=dh\cdot d\mu_h,
\]
where $dh$ represents integration with respect to the uniform
measure on $\mathcal{V} $.

Observe that $\mathcal{B}_{0,\varepsilon} =z_\varepsilon
(t_{k,\varepsilon})\mathcal{B}_{k,\varepsilon}$.  In words, the
initial conditions giving rise to orbits that are ``bad'' on the
time interval $[t_{k,\varepsilon},t_{k+1,\varepsilon}] $, moved
forward by time $t_{k,\varepsilon}$, are precisely the initial
conditions giving rise to orbits that are ``bad'' on the time
interval $[t_{0,\varepsilon},t_{1,\varepsilon}] $.  Because the flow
$z_0(\cdot) $ preserves the measure $P^f$, we expect
$P^f(\mathcal{B}_{0,\varepsilon})$ and
$P^f(\mathcal{B}_{k,\varepsilon})$ to have roughly the same size.
This is made precise by the following lemma.

\begin {lem}
\label {lem:anosov2} There exists a constant $K$ such that for each
Borel set $B\subset \mathcal{M}$ and each $t\in
[-T/\varepsilon,T/\varepsilon] $, $P^f(z_\varepsilon(t)B)\leq
e^{KT}P^f(B)$.
\end {lem}
\begin {proof}
Assume that $P^f(B) > 0$, and set $\gamma (t)=\ln
\bigr(P^f(z_\varepsilon(t)B)/P^f(B)\bigl)$.  Then $\gamma (0) =0$,
and
\[
\begin {split}
    \frac {d\gamma }{dt}(t)=&\frac{\frac{d}{dt}\int_{z_\varepsilon(t)B}\tilde f(z)dz}
    {\int_{z_\varepsilon(t)B}\tilde f(z)dz}
    =\frac{\int_{z_\varepsilon(t)B}\divergence_{P^f} Z(z,\varepsilon)dz}
    {\int_{z_\varepsilon(t)B}\tilde f(z)dz},
\end {split}
\]
where $\tilde f>0$ is the $\mathcal{C} ^1$ density of $P^f$ with
respect to Riemannian volume on $h^ {-1}\mathcal{V} $, $dz$
represents integration with respect to that volume, and
$\divergence_{P^f} Z(z,\varepsilon)=\divergence_z \tilde
f(z)Z(z,\varepsilon)$. Because $z_0(t) $ preserves $P^f$,
$\divergence_{P^f} Z(z,0)\equiv 0$.  By Hadamard's Lemma, it follows
that $\divergence_{P^f}\, Z(z,\varepsilon)=\mathcal{O}(\varepsilon)$
on the compact set $h^ {-1}\mathcal{V}$.  Hence $d\gamma
(t)/dt=\mathcal{O}(\varepsilon) $, and the result follows.
\end {proof}

Returning to our proof of Anosov's theorem, it suffices to show that
\[
    P^f(\mathcal{B}_{0,\varepsilon})=\int_\mathcal{V} dh\cdot \mu_h
    \biggl\{z:\frac{1}{L(\varepsilon)}\abs{\int_0^{L(\varepsilon)}
    H(z_0(s),0)-\bar H(h_0(0))ds}\geq\frac{\delta}{2T}\biggr\}
\]
tends to $0$ with $\varepsilon$.  By our ergodicity assumption, for
almost every $h$,
\[
    \mu_h
    \biggl\{z:\frac{1}{L(\varepsilon)}\abs{\int_0^{L(\varepsilon)}
    H(z_0(s),0)-\bar H(h_0(0))ds}\geq\frac{\delta}{2T}\biggr\}
    \rightarrow 0\text{ as }\varepsilon\rightarrow 0.
\]
Finally, an application of the Bounded Convergence Theorem finishes
the proof.

\end{proof}

\section {Moral}

From the proofs of the theorems in this chapter, it should be
apparent that there are at least two key steps necessary for proving
a version of the averaging principle in the setting presented in
Section~\ref{sct:anos}.

The first step is \emph{estimating the continuity between the
$\varepsilon=0$ and the $\varepsilon>0$ solutions} of
\[
    \frac {dz} {dt}=Z(z,\varepsilon).
\]
In particular, on some relatively long timescale
$L=L(\varepsilon)\ll \varepsilon^ {-1} $, we need to show that
\[
    \sup_{0\leq t\leq L}\abs{z_{0}(t)-
    z_\varepsilon(t)}\rightarrow 0
\]
as $\varepsilon\rightarrow 0$.  As long as $L$ is sub-logarithmic in
$\varepsilon ^ {-1} $, such estimates for smooth systems can be made
using Gronwall's Inequality.

The second step is \emph{estimating the rate of ergodization} of
$H(\cdot,0) $ by $ z_0(t) $, i.e.~estimating how fast
\[
    \frac {1} {L}\int_{0}^{L} H(z_{0}(s),0)\,ds\rightarrow
     \bar H(h_{0})
\]
(generally as $L\rightarrow\infty$).  Note that the estimates in
this step compete with those in the first step in that, if $L$ is
small we obtain better continuity, but if $L$ is large we usually
obtain better ergodization.  Also, we do not need the full force of
the assumption of ergodicity of $ (z_0(t),\mu_h) $ on the fibers
$\mathcal{M}_h$. We only need $ z_0(t) $ to ergodize the specific
function $H(\cdot,0) $.  Compare the proof of
Theorem~\ref{thm:simple_averaging3}.

Note that in the setting of Anosov's theorem, uniform ergodization
leads to uniform convergence in the averaging principle. Returning
to the proof of Theorem~\ref{thm:anosov} above, suppose that
\[
    \frac {1} {L(\varepsilon)}\int_0^{L(\varepsilon)}
    H(z_0(s),0)ds\rightarrow \bar H(h_0)
\]
uniformly over all initial conditions as $L(\varepsilon)\rightarrow
\infty$. Then for all $\varepsilon $ sufficiently small and  each $k
$, $\mathcal{B}_{k,\varepsilon} =\emptyset $, and hence for all
$\varepsilon $ sufficiently small and  each $z $, $\#
(\mathcal{B}_{z,\varepsilon}) =0 $.  From Equation
\eqref{eq:hinfnorm2}, it follows that $\sup_{0\leq \tau\leq T\wedge
T_\varepsilon}\abs{\bar h(\tau)-h_\varepsilon(\tau/\varepsilon)}
 \rightarrow 0$ as $\varepsilon \rightarrow 0$, uniformly over all
initial conditions $z\in h ^ {-1}\mathcal{V} $.  However, uniform
convergence in Birkhoff's Ergodic Theorem is extremely rare and
usually comes about because of unique ergodicity, so it is
unreasonable to expect this sort of uniform convergence in most
situations where Anosov's theorem applies.

\chapter{Results for piston systems in one dimension}\label{chp:1Dpiston}

In this chapter, we present our results for piston systems in one
dimension.  These results may also be found in~\cite{Wri06}.

\section{Statement of results}\label{sct:results}

\subsection{The hard core piston problem}
\label {sct:hard_core_results}

Consider the system of $n_1+n_2+1$ point particles moving inside the
unit interval indicated in Figure \ref{fig:piston}.  One
distinguished particle, the piston, has position $Q$ and mass $M$.
To the left of the piston there are $n_1>0$ particles with positions
$q_{1,j}$ and masses $m_{1,j} $, $1\leq j\leq n_1$, and to the right
there are $n_2>0$ particles with positions $q_{2,j}$ and masses
$m_{2,j} $, $1\leq j\leq n_2$.  These gas particles do not interact
with each other, but they interact with the piston and with walls
located at the end points of the unit interval via elastic
collisions.  We denote the velocities by $dQ/dt=V$ and
$dx_{i,j}/dt=v_{i,j}$. There is a standard method for transforming
this system into a billiard system consisting of a point particle
moving inside an $(n_1+n_2+1)$-dimensional polytope~\cite{CM06}, but
we will not use this in what follows.

\begin{figure}
    \begin {center}
        \setlength{\unitlength}{1 cm}
    \begin{picture}(10,5)
        \thinlines
        \put(1,2.5){\line(1,0){8}}
        \thicklines
        \put(1,1){\line(0,1){3}}
        \put(9,1){\line(0,1){3}}
        \multiput(0.5,1)(0,0.5){6}{\line(1,1){0.5}}
        \multiput(9,1)(0,0.5){6}{\line(1,1){0.5}}
        \put(0.9,0.5){$0$}
        \put(8.9,0.5){$1$}
        \linethickness {0.15cm}
        \put(4.0,1){\line(0,1){3}}
        \put(3.83,0.5){$Q$}
        \put(3.83,4.2){$M$}
        \put(1.5,2.5){\circle*{.18}}
        \put(2,2.5){\circle*{.18}}
        \put(3,2.5){\circle*{.18}}
        \put(4.5,2.5){\circle*{.18}}
        \put(6,2.5){\circle*{.18}}
        \put(7.5,2.5){\circle*{.18}}
        \put(8,2.5){\circle*{.18}}
        \put(2.85,2){$q_{1,j}$}
        \put(2.85,2.8){$m_{1,j}$}
        \put(5.85,2){$q_{2,j}$}
        \put(5.85,2.8){$m_{2,j}$}
    \end{picture}
    \end {center}
    \caption{The piston system with $n_1=3$ and $n_2=4$.  Note that the
        gas particles do not interact with each other, but only with the
        piston and the walls.}\label{fig:piston}
\end{figure}
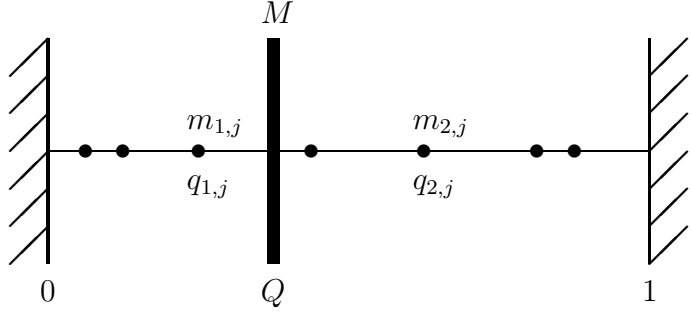

We are interested in the dynamics of this system when the numbers
and masses of the gas particles are fixed, the total energy is
bounded, and the mass of the piston tends to infinity. When
$M=\infty $, the piston remains at rest, and each gas particle
performs periodic motion.  More interesting are the motions of the
system when $M$ is very large but finite. Because the total energy
of the system is bounded, $MV^2/2\leq \C$, and so $V=\mathcal{O} (M^
{-1/2})$. Set
\[
    \varepsilon = M^ {-1/2},
\]
and let
\[
    W=\frac {V} {\varepsilon},
\]
so that
\[
    \frac {dQ} {dt}=\varepsilon W
\]
with $W =\mathcal{O} (1) $.

When $\varepsilon =0 $, the system has $n_1+n_2+2$ independent first
integrals (conserved quantities), which we take to be $Q,\: W$, and
$s_{i,j} =\abs{v_{i,j}}$, the speeds of the gas particles. We refer
to these variables as the slow variables because they should change
slowly with time when $\varepsilon$ is small, and we denote them by
\[
    h=(Q,W, s_{1,1},s_{1,2},\cdots,s_{1,n_1},
    s_{2,1},s_{2,2},\cdots,s_{2,n_2})\in\mathbb{R}^ {n_1+n_2+2}.
\]
We will often abbreviate by writing $h=(Q,W, s_{1,j},s_{2,j})$. Let
$h_\varepsilon(t,z) =h_\varepsilon(t) $ denote the dynamics of these
variables in time for a fixed value of $\varepsilon$, where $z$
represents the dependence on the initial condition in phase space.
We usually suppress the initial condition in our notation. Think of
$h_\varepsilon(\cdot) $ as a random variable which, given an initial
condition in the $2(n_1+n_2+1)$-dimensional phase space, produces a
piecewise continuous path in $\mathbb{R}^ {n_1+n_2+2}$. These paths
are the projection of the actual motions in our phase space onto a
lower dimensional space. The goal of averaging is to find a vector
field on $\mathbb{R}^ {n_1+n_2+2}$ whose orbits approximate
$h_\varepsilon(t) $.

\subsubsection {The averaged equation}
\label{sct:1dps_avg}

Sinai~\cite {Sin99} derived
\begin {equation} \label {eq:1davg}
    \frac{d}{d\tau}
    \begin {bmatrix}
    Q\\
    W\\
    s_{1,j}\\
    s_{2,j}\\
    \end {bmatrix}
    =\bar H(h) :=
    \begin {bmatrix}
    W\\
    \frac{\sum_{j=1} ^ {n_1}m_{1,j} s_{1,j}^2}{Q}-
    \frac{\sum_{j=1} ^ {n_2}m_{2,j} s_{2,j}^2}{1-Q}\\
    -\frac{s_{1,j} W}{Q}\\
    +\frac{s_{2,j} W}{1-Q}\\
    \end {bmatrix}
\end {equation}
as the averaged equation (with respect to the slow time
$\tau=\varepsilon t$) for the slow variables.  We provide a
heuristic derivation in Section \ref{sct:1dps_hc_heuristic}.  Sinai
solved this equation as follows: From
\[
    \frac {d\ln(s_{1,j})}{d\tau}=-\frac {d\ln(Q)} {d\tau},
\]
$s_{1,j}(\tau)=s_{1,j}(0)Q(0)/Q(\tau)$.  Similarly,
$s_{2,j}(\tau)=s_{2,j}(0) (1 - Q(0))/(1 - Q(\tau))$.  Hence
\[
    \frac {d^2Q} {d\tau^2}=
    \frac{\sum_{j=1} ^ {n_1}m_{1,j} s_{1,j}(0)^2 Q(0) ^2}{Q^3}
    -\frac{\sum_{j=1} ^ {n_2}
    m_{2,j} s_{2,j}(0)^2 (1-Q(0)) ^2}{(1-Q) ^3},
\]
and so $(Q, W) $ behave as if they were the coordinates of a
Hamiltonian system describing a particle undergoing periodic motion
inside a potential well. If we let
\[
    E_i=\sum_{j=1} ^ {n_i}
    \frac{m_{i,j}} {2}s_{i,j}^2
\]
be the kinetic energy of the gas particles on one side of the
piston, the effective Hamiltonian may be expressed as
\begin {equation}
\label {eq:1dpot}
    \frac {1} {2}W^2+
    \frac{E_1(0)Q(0)^2}{Q^2}+\frac{E_2(0)(1-Q(0))^2}{(1-Q)^2}.
\end {equation}
Hence, the solutions to the averaged equation are periodic for all
initial conditions under consideration.

\subsubsection{Main result in the hard core setting} \label {sct:1dres}

The solutions of the averaged equation approximate the motions of
the slow variables, $h_\varepsilon(t) $, on a time scale
$\mathcal{O} (1/\varepsilon) $ as $\varepsilon\rightarrow 0$.
Precisely, let $\bar{h} (\tau,z)=\bar{h} (\tau) $ be the solution of
\[
\frac {d\bar{h}}{d\tau} =\bar {H} (\bar {h}),\qquad \bar {h} (0)
=h_\varepsilon(0).
\]
Again, think of $\bar h(\cdot) $ as being a random variable that
takes an initial condition in our phase space and produces a path in
$\mathbb R^{n_1+n_2+2}$.

Next, fix a compact set $\mathcal{V}\subset \mathbb R^{n_1+n_2+2}$
such that $h\in \mathcal{V} \Rightarrow Q\subset\subset
(0,1),W\subset\subset \mathbb R$, and $s_{i,j} \subset\subset
(0,\infty)$ for each $i$ and $j$.\footnote { We have introduced this
notation for convenience.  For example, $h\in \mathcal{V}
\Rightarrow Q\subset\subset (0,1) $ means that there exists a
compact set $A \subset (0,1) $ such that $h\in \mathcal{V}
\Rightarrow Q\in A $, and similarly for the other variables.}  For
the remainder of this discussion we will restrict our attention to
the dynamics of the system while the slow variables remain in the
set $\mathcal{V} $. To this end, we define the stopping time
\[
    T_\varepsilon(z) =T_\varepsilon :=\inf \{\tau\geq 0: \bar {h}
    (\tau)\notin \mathcal{V} \text { or } h_\varepsilon(\tau
    /\varepsilon) \notin \mathcal{V} \}.
\]

\begin {thm}
\label{thm:1Dpiston1}

For each $T>0$,
\[
\sup_{\substack{\text{initial conditions}\\
\text{s.t. } h_\varepsilon(0)\in  \mathcal{V} }}\;
\sup_{0\leq\tau\leq T\wedge
T_\varepsilon}\abs{h_\varepsilon(\tau/\varepsilon)-\bar{h}(\tau)}
=\mathcal{O} (\varepsilon  ) \text { as } \varepsilon=M^
{-1/2}\rightarrow 0.
\]
\end{thm}
\noindent This result was independently obtained by Gorelyshev and
Neishtadt~\cite{GorNeishtadt06}.

Note that the stopping time does not unduly restrict the result.
Given any $c $ such that $h=c\Rightarrow Q\in (0,1),\: s_{i,j}\in
(0,\infty)$, then by an appropriate choice of the compact set
$\mathcal{V} $ we may ensure that, for all $\varepsilon $
sufficiently small and all initial conditions in our phase space
with $h_\varepsilon(0) =c$, $T_\varepsilon \geq T $. We do this by
choosing $\mathcal{V}\ni c $ such that the distance between
$\partial \mathcal{V} $ and the periodic orbit $\bar h (\tau) $ with
$\bar h (0) =c $ is positive. Call this distance $d $. Then
$T_\varepsilon $ can only occur before $T$ if $h_\varepsilon
(\tau/\varepsilon) $ has deviated by at least $d $ from $ \bar h
(\tau) $ for some $\tau\in [0,T) $. Since the size of the deviations
tends to zero uniformly with $\varepsilon$, this is impossible for
all small $\varepsilon $.

\subsection{The soft core piston problem}
\label {sct:soft_core_results}

In this section, we consider the same system of one piston and gas
particles inside the unit interval considered in Section
\ref{sct:hard_core_results}, but now the interactions of the gas
particles with the walls and with the piston are smooth.  Let
$\kappa\colon\mathbb {R}\rightarrow\mathbb {R} $ be a $\mathcal{C}
^2$ function satisfying
\begin {itemize}
    \item
         $\kappa (x)= 0 $ if $x\geq 1 $,
    \item
        $\kappa ' (x) < 0 $ if $x < 1 $.
\end {itemize}
Let $\delta >0 $ be a parameter of smoothing, and set
\[
    \kappa _\delta(x) =\kappa (x/\delta).
\]
Then consider the Hamiltonian system obtained by having the gas
particles interact with the piston and the walls via the potential
\[
    \sum_{j=1}^{n_1}\kappa_\delta(q_{1,j})+\kappa_\delta(Q-q_{1,j})+
    \sum_{j=1}^{n_2}\kappa_\delta(q_{2,j}-Q)+\kappa_\delta(1-q_{2,j}).
\]
As before, we set $\varepsilon =M^ {-1/2} $ and $W =V/\varepsilon$.
If we let
\begin {equation}
\label{eq:soft_energies}
\begin {split}
    E_{1,j}&=\frac{1}{2}m_{1,j}v_{1,j}^2+\kappa_\delta(q_{1,j})
        +\kappa_\delta(Q-q_{1,j}),\qquad 1\leq j\leq n_1,\\
    E_{2,j} & =\frac{1}{2}m_{2,j}v_{2,j}^2+\kappa_\delta(q_{2,j}-Q)
        +\kappa_\delta(1-q_{2,j}),\qquad 1\leq j\leq n_2,\\
\end {split}
\end {equation}
then $E_{i,j}$ may be thought of as the energy associated with a gas
particle, and $W ^2/2+\sum_{j=1}^{n_1}E_{1,j}
+\sum_{j=1}^{n_2}E_{2,j} $ is the conserved energy.

When $\varepsilon =0 $, the Hamiltonian system admits $n_1+n_2+2$
independent first integrals, which we choose this time as $h = (Q,
W,E_{1,j},E_{2,j}) $.  While discussing the soft core dynamics we
use the energies $E_{i,j} $ rather than the variables $s_{i,j}  =
\sqrt {2E_{i,j} /m_{i,j} } $, which we used for the hard core
dynamics, for convenience.

For comparison with the hard core results, we formally consider the
dynamics described by setting $\delta=0$ to be the hard core
dynamics described in Section \ref{sct:hard_core_results}. This is
reasonable because we will only consider gas particle energies below
the barrier height $\kappa(0) $.  Then for any
$\varepsilon,\delta\geq 0$, $h_\varepsilon^\delta (t) $ denotes the
actual time evolution of the slow variables. While discussing the
soft core dynamics we often use $\delta $ as a superscript to
specify the dynamics for a certain value of $\delta $. We usually
suppress the dependence on $\delta $, unless it is needed for
clarity.

\subsubsection{Main result in the soft core setting}
\label {sct:1d_sm_results}

We have already seen that when $\delta=0$, there is an appropriate
averaged vector field $\bar H^0$ whose solutions approximate the
actual motions of the slow variables, $h_\varepsilon^0 (t) $. We
will show that when $\delta >0$, there is also an appropriate
averaged vector field $\bar H^\delta$ whose solutions still
approximate the actual motions of the slow variables,
$h_\varepsilon^\delta (t) $. We delay the derivation of $\bar
H^\delta$ until Section \ref{sct:smooth_1D_derivation}.

Fix a compact set $\mathcal{V}\subset\mathbb{R}^{n_1+n_2+2}$ such
that $h\in\mathcal{V}\Rightarrow Q\subset\subset (0,1),
W\subset\subset\mathbb{R}$, and $E_{i,j}\subset\subset (0,\kappa
(0)) $ for each $i$ and $j$.  For each $\varepsilon,\delta\geq 0$ we
define the functions $\bar h ^\delta (\cdot) $ and
$T_\varepsilon^\delta $ on our phase space by letting $\bar h
^\delta (\tau) $ be the solution of
\begin {equation}
\label {eq:smooth_1D_averaged_eq}
    \frac {d\bar{h} ^\delta}{d\tau}
    =\bar {H} ^\delta (\bar {h} ^\delta),\qquad \bar {h} ^\delta (0)
    =h_\varepsilon ^\delta (0),
\end {equation}
and
\[
    T_\varepsilon ^\delta =\inf \{\tau\geq 0: \bar {h} ^\delta
    (\tau)\notin \mathcal{V} \text { or } h_\varepsilon ^\delta (\tau
    /\varepsilon) \notin \mathcal{V} \}.
\]

\begin {thm}
\label{thm:1D_smooth_uniform}

There exists $\delta_0 >0$ such that the averaged vector field $\bar
H^\delta (h) $ is $\mathcal{C} ^1$ on the domain
$\{(\delta,h):0\leq\delta \leq\delta_0,h\in\mathcal{V}\}$.
Furthermore, for each $T>0$,
\[
\sup_{0\leq\delta\leq\delta_0}\; \sup_{\substack {\text {initial conditions}\\
\text {s.t. } h_\varepsilon^\delta(0)\in  \mathcal{V} }}\;
\sup_{0\leq\tau\leq T\wedge T_\varepsilon ^\delta}\abs{h_\varepsilon
^\delta (\tau/\varepsilon)-\bar{h} ^\delta (\tau)} =\mathcal{O}
(\varepsilon  )\text { as }\varepsilon=M^ {-1/2}\rightarrow 0.
\]
\end{thm}
\noindent As in Section \ref{sct:1dres}, for any fixed $c $ there
exists a suitable choice of the compact set $\mathcal{V} $ such that
for all sufficiently small $\varepsilon$ and $\delta$,
$T_\varepsilon ^\delta\geq T $ whenever $h_\varepsilon ^\delta (0)
=c$.

As we will see, for each fixed $\delta>0 $, Anosov's
theorem~\ref{thm:anosov} applies to the soft core system and yields
a weak law of large numbers, and Theorem~\ref{thm:simple_averaging3}
applies and yields a strong law of large numbers with a uniform rate
of convergence.  However, neither of these theorems yields the
uniformity over $\delta$ in the result above.

\subsection{Applications and generalizations}
\label{sct:apps_generalizations}

\subsubsection{Relationship between the hard core and the soft core
piston} \label{sct:relationship}

It is not \emph{a priori} clear that we can compare the motions of
the slow variables on the time scale $1/\varepsilon$ for $\delta>0$
versus $\delta=0$, i.e.~compare the motions of the soft core piston
with the motions of the hard core piston on a relatively long time
scale. It is impossible to compare the motions of the fast-moving
gas particles on this time scale as $\varepsilon\rightarrow 0$.  As
we see in Section \ref{sct:sc_proof}, the frequency with which a gas
particle hits the piston changes by an amount $\mathcal{O} (\delta)
$ when we smooth the interaction. Thus, on the time scale
$1/\varepsilon$, the number of collisions is altered by roughly
$\mathcal{O} (\delta/\varepsilon) $, and this number diverges if
$\delta$ is held fixed while $\varepsilon\rightarrow 0$.

Similarly, one might expect that it is impossible to compare the
motions of the soft and hard core pistons as $\varepsilon\rightarrow
0$ without letting $\delta\rightarrow 0$ with $\varepsilon$.
However, from Gronwall's Inequality it follows that if $\bar
h^\delta (0) =\bar h^0(0) $, then
\[
    \sup_{0\leq\tau\leq T\wedge
    T_\varepsilon ^\delta\wedge T_\varepsilon ^0} \abs{\bar h^\delta
    (\tau) -\bar h^0(\tau)}=\mathcal{O} (\delta).
\]
From the triangle inequality and Theorems \ref{thm:1Dpiston1} and
\ref{thm:1D_smooth_uniform} we obtain the following corollary, which
allows us to compare the motions of the hard core and the soft core
piston.
\begin {cor}
\label {cor: 1d_smooth_comparison}

As $\varepsilon=M^ {-1/2},\delta \rightarrow 0$,
\[
\sup_{c\in \mathcal{V} }\; \sup_{\substack {\text {initial conditions}\\
\text {s.t. } h_\varepsilon^\delta(0)=c=h_\varepsilon^0(0) }}\;
 \sup_{0\leq t\leq (T\wedge T_\varepsilon
^\delta\wedge T_\varepsilon ^0)/\varepsilon}\abs{h_\varepsilon
^\delta (t)-h_\varepsilon ^0 (t)} =\mathcal{O} (\varepsilon
)+\mathcal{O} (\delta).
\]
\end {cor}

This shows that, provided the slow variables have the same initial
conditions,
\[
\sup_{0\leq t\leq 1/\varepsilon}\abs{h_\varepsilon ^\delta
(t)-h_\varepsilon ^0 (t)} =\mathcal{O} (\varepsilon )+\mathcal{O}
(\delta).
\]
Thus the motions of the slow variables converge on the time scale
$1/\varepsilon $ as $\varepsilon,\delta\rightarrow 0 $, and it is
immaterial in which order we let these parameters tend to zero.

\subsubsection{The adiabatic piston problem}

We comment on what Theorem \ref{thm:1Dpiston1} says about the
adiabatic piston problem. The initial conditions of the adiabatic
piston problem require that $W(0) =0$. Although our system is so
simple that a proper thermodynamical pressure is not defined, we can
define the pressure of a gas to be the average force received from
the gas particles by the piston when it is held fixed,
i.e.~$P_1=\sum_{j=1} ^{n_1}2m_{1,j} s_{1,j} \frac{s_{1,j}}{2Q} =2
E_1/Q $ and $P_2=2E_2/(1-Q) $. Then if $P_1(0) >P_2(0) $, the
initial condition for our averaged equation \eqref{eq:1davg} has the
motion of the piston starting at the left turning point of a
periodic orbit determined by the effective potential well.  Up to
errors not much bigger than $M^ {-1/2} $, we see the piston
oscillate periodically on the time scale $M^ {1/2}$. If $P_1(0) <
P_2(0) $, the motion of the piston starts at a right turning point.
However, if $P_1(0) = P_2(0) $, then the motion of the piston starts
at the bottom of the effective potential well. In this case of
mechanical equilibrium, $\bar h(\tau) =\bar h(0) $, and we conclude
that, up to errors not much bigger than $M^ {-1/2} $, we see no
motion of the piston on the time scale $M^ {1/2}$.  A much longer
time scale is required to see if the temperatures equilibrate.

\subsubsection{Generalizations}

A simple generalization of Theorem \ref{thm:1Dpiston1}, proved by
similar techniques, follows. The system consists of $N-1$ pistons,
that is, heavy point particles, located inside the unit interval at
positions $Q_1< Q_2<\dotsc< Q_{N-1} $. Walls are located at
$Q_0\equiv 0$ and $Q_N\equiv 1$, and the piston at position $Q_i$
has mass $M_i$. Then the pistons divide the unit interval into $N$
chambers. Inside the $i^ {th}$ chamber, there are $n_i\geq 1$ gas
particles whose locations and masses will be denoted by $x_{i,j}$
and $m_{i,j}$, respectively, where $1\leq j\leq n_i$. All of the
particles are point particles, and the gas particles interact with
the pistons and with the walls via elastic collisions. However, the
gas particles do not directly interact with each other. We scale the
piston masses as $M_i=\hat M_i/\varepsilon^2$ with $\hat M_i$
constant, define $W_i$ by $dQ_i/dt=\varepsilon W_i$, and let $E_i$
be the kinetic energy of the gas particles in the $i^ {th}$ chamber.
Then we can find an appropriate averaged equation whose solutions
have the pistons moving like an $(N-1)$-dimensional particle inside
a potential well with an effective Hamiltonian
\[
    \frac {1} {2}\sum_{i=1}^{N-1}\hat M_i W_i^2 +
    \sum_{i=1}^{N}
    \frac{E_i(0)(Q_i(0)-Q_{i-1}(0))^2}{(Q_i-Q_{i-1})^2}.
\]
If we write the slow variables as $h= (Q_i,W_i,\abs{v_{i,j}})$ and
fix a compact set $\mathcal{V}$ such that $h\in \mathcal{V}
\Rightarrow Q_{i+1} -Q_i\subset\subset (0,1),W_i\subset\subset
\mathbb R$, and $\abs{v_{i,j}} \subset\subset (0,\infty)$, then the
convergence of the actual motions of the slow variables to the
averaged solutions is exactly the same as the convergence given in
Theorem \ref{thm:1Dpiston1}.

\begin {rem}

The inverse quadratic potential between adjacent pistons in the
effective Hamiltonian above is also referred to as the
Calogero-Moser-Sutherland potential.  It has also been observed as
the effective potential created between two adjacent tagged
particles in a one-dimensional Rayleigh gas by the insertion of one
very light particle inbetween the tagged
particles~\cite{BalintTothToth2007}.

\end{rem}

\section{Heuristic derivation of the averaged equation for the
        hard core piston}
\label {sct:1dps_hc_heuristic}

We present here a heuristic derivation of Sinai's averaged equation
\eqref{eq:1davg} that is found in~\cite {Dol05}.

First, we examine interparticle collisions when $\varepsilon>0$.
When a particle on the left, say the one at position $q_{1,j} $,
collides with the piston, $s_{1,j} $ and $W$ instantaneously change
according to the laws of elastic collisions:
\begin {equation}\label {eq:v1lin}
    \begin{bmatrix}
    v_{1,j}^+\\ V^+
    \end{bmatrix}
    =
    \frac{1}{m_{1,j}+M}
    \begin{bmatrix}
    m_{1,j}-M& 2M\\
    2m_{1,j}& M-m_{1,j}\\
    \end{bmatrix}
    \begin{bmatrix}
    v_{1,j}^-\\ V^-
    \end{bmatrix}.
\end {equation}
If the speed of the left gas particle is bounded away from zero, and
$W=M^ {1/2} V $ is also bounded, it follows that for all
$\varepsilon$ sufficiently small, any collision will have $v_{1,j}^
-
>0$ and $v_{1,j}^ + <0$. In this case, when we translate Equation
\eqref{eq:v1lin} into our new coordinates, we find that
\begin {equation}\label {eq:s1lin}
    \begin{bmatrix}
    s_{1,j}^+\\ W^+
    \end{bmatrix}
    =
    \frac{1}{1+\varepsilon^2 m_{1,j}}
    \begin{bmatrix}
    1-\varepsilon^2 m_{1,j}& - 2\varepsilon\\
    2\varepsilon m_{1,j}& 1-\varepsilon^2 m_{1,j}\\
    \end{bmatrix}
    \begin{bmatrix}
    s_{1,j}^-\\ W^-
    \end{bmatrix},
\end {equation}
so that
\[
\begin {split}
    \Delta s_{1,j}&=s_{1,j}^+ -s_{1,j}^- =-2\varepsilon W^-
    +\mathcal{O}(\varepsilon^2),\\
    \Delta W &=W^+ -W^- =+2\varepsilon m_{1,j}
    s_{1,j}^- +\mathcal{O}(\varepsilon^2).\\
\end {split}
\]

The situation is analogous when particles on the right collide with
the piston. For all $\varepsilon $ sufficiently small, $s_{2,j} $
and $W$ instantaneously change by
\[
\begin {split}
    \Delta W &=W^+ -W^- = - 2\varepsilon m_{2,j}
    s_{2,j}^- +\mathcal{O}(\varepsilon^2),\\
    \Delta s_{2,j}&=s_{2,j}^+ -s_{2,j}^- = + 2\varepsilon W^-
    +\mathcal{O}(\varepsilon^2).\\
\end {split}
\]
We defer discussing the rare events in which multiple gas particles
collide with the piston simultaneously, although we will see that
they can be handled appropriately.

Let $\Delta t $ be a length of time long enough such that the piston
experiences many collisions with the gas particles, but short enough
such that the slow variables change very little, in this time
interval.  From each collision with the particle at position
$q_{1,j} $, $W$ changes by an amount $+2\varepsilon m_{1,j}
s_{1,j}+\mathcal{O}(\varepsilon^2)$, and the frequency of these
collisions is approximately $\frac {s_{1,j}} {2Q} $. Arguing
similarly for collisions with the other particles, we guess that
\[
    \frac {\Delta W} {\Delta t} =
    \varepsilon\sum_{j=1} ^ {n_1} 2m_{1,j}
    s_{1,j}\frac{s_{1,j}}{2Q}
    - \varepsilon\sum_{j=1} ^ {n_2}
    2m_{2,j} s_{2,j}\frac{s_{2,j}}{2(1-Q)}
    +\mathcal{O}(\varepsilon^2).
\]
Note that not only does the position of the piston change slowly in
time, but its velocity also changes slowly, i.e. the piston has
inertia. With $\tau=\varepsilon t$ as the slow time, a reasonable
guess for the averaged equation for $W$ is
\[
    \frac {dW}{d\tau}=\frac{\sum_{j=1} ^ {n_1}m_{1,j} s_{1,j}^2}{Q}
    -\frac{\sum_{j=1} ^ {n_2}m_{2,j} s_{2,j}^2}{1-Q}.
\]
Similar arguments for the other slow variables lead to the averaged
equation \eqref{eq:1davg}.

\section{Proof of the main result for the hard core piston}
\label {sct:hc_proof}

\subsection{Proof of Theorem \ref{thm:1Dpiston1} with only one gas
            particle on each side}

We specialize to the case when there is only one gas particle on
either side of the piston, i.e.~we assume that $ n_1= n_2=1$.  We
then denote $x_{1,1} $ by $q_1$, $m_{2,2} $ by $m_2$, etc. This
allows the proof's major ideas to be clearly expressed, without
substantially limiting their applicability. At the end of this
section, we outline the simple generalizations needed to make the
proof apply in the general case.

\subsubsection{A choice of coordinates on the phase space for a
    three particle system}
\label {sct:1dps_hcPS}

As part of our proof, we choose a set of coordinates on our
six-dimensional phase space such that, in these coordinates, the
$\varepsilon=0$ dynamics are smooth.  Complete the slow variables
$h= (Q,W,s_1,s_2)$ to a full set of coordinates by adding the
coordinates $\varphi _i\in [0,1]/ \, 0\sim 1=S^1,\:i=1,2$, defined
as follows:
\[
\begin {split}
    \varphi_1 =\varphi_1 (q_1,v_1,Q)=&\begin {cases}
    \frac{q_1}{2Q} &\text { if } v_1 >0\\
    1-\frac{q_1}{2Q} &\text { if } v_1<0\\
    \end {cases}
    \\
    \varphi_2=\varphi_2(q_2,v_2,Q)=&\begin {cases}
    \frac{1-q_2}{2(1-Q)} &\text { if } v_2<0\\
    1- \frac{1-q_2}{2(1-Q)} &\text { if } v_2>0\\
    \end {cases}\\
\end {split}.
\]
When $\varepsilon =0 $, these coordinates are simply the angle
variable portion of action-angle coordinates for an integrable
Hamiltonian system.  They are defined such that collisions occur
between the piston and the gas particles precisely when $\varphi_1 $
or $\varphi_2 =1/2$. Then $z = (h,\varphi_1,\varphi_2) $ represents
a choice of coordinates on our phase space, which is homeomorphic to
$(\text{a subset of } \mathbb{R}^4)\times \mathbb{ T} ^ 2 $.  We
abuse notation and also let $h(z) $ represent the projection onto
the first four coordinates of $z$.

Now we describe the dynamics of our system in these coordinates.
When $\varphi_1,\varphi_2\ne 1/2$,
\[
\begin {split}
    \frac{d\varphi_1}{dt} =&\begin {cases}
    \frac{s_1}{2Q}-\frac{\varepsilon W}{Q}\varphi_1&\text { if }
    0\leq \varphi_1<1/2\\
    \frac{s_1}{2Q}+\frac{\varepsilon W}{Q}(1-\varphi_1)&\text { if }
    1/2< \varphi_1\leq 1\\
    \end {cases}
    \\
    \frac{d\varphi_2}{dt} =&\begin {cases}
    \frac{s_2}{2(1-Q)}+\frac{\varepsilon W}{1-Q}\varphi_2&\text { if }
    0\leq \varphi_2<1/2\\
    \frac{s_2}{2(1-Q)}-\frac{\varepsilon W}{1-Q}(1-\varphi_2)&\text { if }
    1/2< \varphi_2\leq 1\\
    \end {cases}
\end {split}.
\]
Hence between interparticle collisions, the dynamics are smooth and
are described by
\begin {equation}\label {eq:1dode}
\begin {split}
    \frac{dQ}{dt}&=\varepsilon W,\\
    \frac{dW}{dt}&=0,\\
    \frac{ds_1}{dt}&=0,\\
    \frac{ds_2}{dt}&=0,\\
    \frac{d\varphi_1}{dt}&=\frac{s_1}{2Q}+
        \mathcal{O}(\varepsilon),\\
    \frac{d\varphi_2}{dt}&=\frac{s_2}{2 (1 - Q)}+
        \mathcal{O}(\varepsilon).\\
\end {split}
\end {equation}

When $\varphi_1$ reaches $1/2$, while $\varphi_2\neq 1/2$, the
coordinates $Q,s_2,\varphi_1$, and $\varphi_2$ are instantaneously
unchanged, while $s_1$ and $W$ instantaneously jump, as described by
Equation \eqref{eq:s1lin}. As an aside, it is curious that $s_1^+
+\varepsilon W^+ = s_1^- -\varepsilon W^-$, so that $d\varphi_1/dt $
is continuous as $\varphi_1 $ crosses $1/2 $. However, the collision
induces discontinuous jumps of size $\mathcal{O} (\varepsilon ^2) $
in $dQ/dt$ and $d\varphi_2/dt$.  Denote the linear transformation in
Equation \eqref{eq:s1lin} with $j=1$ by $A_{1,\varepsilon} $.  Then
\[
    A_{1,\varepsilon} =\begin{bmatrix}
    1& - 2\varepsilon\\
    2\varepsilon m_1& 1\\
    \end{bmatrix} +\mathcal{O} (\varepsilon^2).
\]

The situation is analogous when $\varphi_2$ reaches $1/2$, while
$\varphi_1\neq 1/2$.  Then $W$ and $s_2$ are instantaneously
transformed by a linear transformation
\[
    A_{2,\varepsilon}
    =\begin{bmatrix}
    1&  -2\varepsilon m_2\\
    2\varepsilon& 1\\
    \end{bmatrix} +\mathcal{O} (\varepsilon^2).
\]

We also account for the possibility of all three particles colliding
simultaneously.  There is no completely satisfactory way to do this,
as the dynamics have an essential singularity near $\{\varphi_1
=\varphi_2 =1/2\} $.  Furthermore, such three particle collisions
occur with probability zero with respect to the invariant measure
discussed below.  However, the two $3\times 3$ matrices
\[
    \begin{bmatrix}
    A_{1,\varepsilon} & 0\\
    0& 1\\
    \end{bmatrix},\quad
    \begin{bmatrix}
    1 & 0\\
    0& A_{2,\varepsilon}\\
    \end{bmatrix}
\]
have a commutator of size $\mathcal{O} (\varepsilon ^2) $.  We will
see that this small of an error will make no difference to us as
$\varepsilon\rightarrow 0$, and so when $\varphi_1 =\varphi_2 =1/2
$, we pretend that the left particle collides with the piston
instantaneously before the right particle does.  Precisely, we
transform the variables $s_1,\: W, $ and $s_2 $ by
\[
\begin{bmatrix}
    s_1^+\\ W^+ \\ s_2^+\\
    \end{bmatrix}=
\begin{bmatrix}
     1 & 0\\
    0& A_{2,\varepsilon}\\
    \end{bmatrix}
\begin{bmatrix}
    A_{1,\varepsilon} & 0\\
    0& 1\\
\end{bmatrix}
\begin{bmatrix}
    s_1^-\\ W^-\\ s_2^-\\
    \end{bmatrix}.
\]
We find that
\[
\begin {split}
    \Delta s_1&=s_1^+ -s_1^- = - 2\varepsilon W^-
    +\mathcal{O}(\varepsilon^2),\\
    \Delta W &=W^+ -W^- = +2\varepsilon m_1 s_1^-
    - 2\varepsilon m_2
    s_2^- +\mathcal{O}(\varepsilon^2),\\
    \Delta s_2&=s_2^+ -s_2^- = + 2\varepsilon W^-
    +\mathcal{O}(\varepsilon^2).\\
\end {split}
\]

The above rules define a flow on the phase space, which we denote by
$z_\varepsilon (t) $.  We denote its components by
$Q_\varepsilon(t),\: W_\varepsilon(t),\: s_{1,\varepsilon} (t), $
etc.  When $\varepsilon
> 0 $, the flow is not continuous, and for
definiteness we take $z_\varepsilon(t) $ to be left continuous in
$t$.

Because our system comes from a Hamiltonian system, it preserves
Liouville
 measure. In our coordinates, this measure has a density proportional
 to
$Q (1 -Q) $.  That this measure is preserved also follows from the
fact that the ordinary differential equation \eqref{eq:1dode}
preserves this measure, and the matrices
$A_{1,\varepsilon},\:A_{2,\varepsilon} $ have determinant $1 $. Also
note that the set $\{\varphi_1 =\varphi_2 =1/2 \} $ has co-dimension
two, and so $\bigcup_t z_\varepsilon(t)\{\varphi_1 =\varphi_2 =1/2
\} $ has co-dimension one, which shows that only a measure zero set
of initial conditions will give rise to three particle collisions.

\subsubsection{Argument for uniform convergence} \label
{sct:1dps_hc_unif}

\paragraph*{Step 1:  Reduction using Gronwall's Inequality.}

Define $H(z) $ by
\[
    H(z) =
    \begin{bmatrix}
    W\\
    2m_1 s_1 \delta_{\varphi_1=1/2}-2m_2
     s_2\delta_{\varphi_2=1/2}\\
    -2W\delta_{\varphi_1=1/2}\\
    2W\delta_{\varphi_2=1/2}\\
    \end{bmatrix}.
\]
Here we make use of Dirac delta functions.  All integrals involving
these delta functions may be replaced by sums.  We explicitly deal
with any ambiguities arising from collisions occurring at the limits
of integration.

\begin {lem}
For $0 \leq t\leq \frac{T\wedge T_\varepsilon}{\varepsilon}$,
\[
    h_\varepsilon(t)-h_\varepsilon(0)=
    \varepsilon\int_0^t
    H(z_\varepsilon(s))ds+\mathcal{O}(\varepsilon),
\]
where any ambiguity about changes due to collisions occurring
precisely at times $0 $ and $t $ is absorbed in the $\mathcal{O}
(\varepsilon) $ term.
\end {lem}

\begin {proof}
There are four components to verify.  The first component requires
that $Q_\varepsilon(t)-Q_\varepsilon(0)=\varepsilon\int_0^t
W_\varepsilon(s)ds+\mathcal{O}(\varepsilon)$.  This is trivially
true because $Q_\varepsilon(t)-Q_\varepsilon(0)=\varepsilon\int_0^t
W_\varepsilon(s)ds$.

The second component states that
\begin {equation}
\label {eq:delt_w}
    W_\varepsilon(t)-
    W_\varepsilon(0)
    =
    \varepsilon\int_0^t 2m_1
    s_{1,\varepsilon}(s) \delta_{\varphi_{1,\varepsilon}(s)=1/2}-2m_2
   s_{2,\varepsilon}(s)\delta_{\varphi_{2,\varepsilon}(s)=1/2}ds
     +\mathcal{O} (\varepsilon).
\end {equation}
Let $r_k $ and $q_j $ be the times in $(0,t) $ such that
$\varphi_{1,\varepsilon}(r_k)=1/2$ and
$\varphi_{2,\varepsilon}(q_j)=1/2$, respectively.  Then
\[
    W_\varepsilon(t)- W_\varepsilon(0)=
     \sum_{r_k}\Delta W_\varepsilon (r_k)+
    \sum_{q_j}\Delta W_\varepsilon (q_j)
     +\mathcal{O} (\varepsilon).
\]
Observe that there exists $\omega >0 $ such that for all
sufficiently small $\varepsilon $ and all $h\in\mathcal{V}$, $
1/\omega < \frac {d\varphi_i} {dt} <\omega $.  Thus the number of
collisions in a time interval grows no faster than linearly in the
length of that time interval.  Because $t\leq T/\varepsilon $, it
follows that
\[
   W_\varepsilon(t)- W_\varepsilon(0)=\\
    \varepsilon\sum_{r_k} 2m_1
    s_{1,\varepsilon}(r_k) -\varepsilon\sum_{q_j}2m_2
   s_{2,\varepsilon}(q_j)
     +\mathcal{O} (\varepsilon),
\]
and Equation \eqref {eq:delt_w} is verified. Note that because
$\mathcal{V} $ is compact, there is uniformity over all initial
conditions in the size of the $\mathcal{O} (\varepsilon) $ terms
above.  The third and fourth components are handled similarly.
\end {proof}

Next, $\bar {h}(\tau) $ satisfies the integral equation
\[
\bar {h}(\tau) -\bar h(0) = \int_0^{\tau}\bar H(\bar
h(\sigma))d\sigma,
\]
while $h_\varepsilon(\tau/\varepsilon)$ satisfies
\[
\begin {split}
    h_\varepsilon(\tau/\varepsilon)-h_\varepsilon(0)
    &=
    \mathcal{O}(\varepsilon) +\varepsilon\int_0^{\tau/\varepsilon}
    H(z_\varepsilon(s))ds\\
    &=\mathcal{O}(\varepsilon) +
    \varepsilon\int_0^{\tau/\varepsilon}
    H(z_\varepsilon(s))-
    \bar H(h_\varepsilon(s))ds+
    \int_0^{\tau}\bar H( h_\varepsilon(\sigma/\varepsilon))d\sigma
\end {split}
\]
for $0\leq\tau\leq T\wedge T_\varepsilon$.

Define
\[
    e_\varepsilon(\tau) =\varepsilon\int_0^{\tau/\varepsilon}
    H(z_\varepsilon(s))- \bar H(h_\varepsilon(s))ds.
\]
It follows from Gronwall's Inequality that
\begin {equation}
\label {eq:1dgronwall}
    \sup_{0\leq \tau\leq T\wedge T_\varepsilon}
    \abs{\bar h(\tau)-h_\varepsilon(\tau/\varepsilon)}\leq
    \left(\mathcal{O}(\varepsilon)+
    \sup_{0\leq \tau\leq T\wedge T_\varepsilon}
    \abs{e_\varepsilon(\tau)}\right)e^{ \Lip{\bar
    H\arrowvert _\mathcal{V}}T}.
\end {equation}
Gronwall's Inequality is usually stated for continuous paths, but
the standard proof (found in \cite{SV85}) still works for paths that
are merely integrable, and $\abs{\bar
h(\tau)-h_\varepsilon(\tau/\varepsilon)}$ is piecewise smooth.

\paragraph*{Step 2:  A splitting according to particles.}

Now
\[
    H(z)-\bar H(h) =
    \begin{bmatrix}
    0\\
    2m_1 s_1  \delta_{\varphi_1=1/2}-m_{1} s_{1}^2/Q\\
    -2W\delta_{\varphi_1=1/2}+s_1 W/Q\\
    0\\
    \end{bmatrix}
    +
    \begin{bmatrix}
    0\\
    -2m_2 s_2\delta_{\varphi_2=1/2}+m_2 s_2^2/(1-Q)\\
    0\\
    2W\delta_{\varphi_2=1/2}-s_2 W/(1-Q)\\
    \end{bmatrix}
    ,
\]
and so, in order to show that $\sup_{0\leq \tau\leq T\wedge
T_\varepsilon}\abs{e_\varepsilon(\tau)} =\mathcal{O} (\varepsilon)
$, it suffices to show that
\begin {align*}\label{eq:hc_separation}
    &
    \sup_{0\leq \tau\leq
    T\wedge T_\varepsilon}
    \abs{\int_0^{\tau/\varepsilon}
    s_{1,\varepsilon}(s)\delta_{\varphi_{1,\varepsilon}(s)=1/2}-
    \frac{s_{1,\varepsilon}(s)^2}{2Q_\varepsilon(s)}ds}
    =\mathcal{O} (1),
    \\
    &
    \sup_{0\leq \tau\leq
    T\wedge T_\varepsilon}
    \abs{\int_0^{\tau/\varepsilon}
    W_{\varepsilon}(s)\delta_{\varphi_{1,\varepsilon}(s)=1/2}-
    \frac{W_{\varepsilon}(s)s_{1,\varepsilon}(s)}{2Q_\varepsilon(s)}ds}
    =\mathcal{O} (1),
\end {align*}
as well as two analogous claims about terms involving
$\varphi_{2,\varepsilon} $.  Thus we have effectively separated the
effects of the different gas particles, so that we can deal with
each particle separately. We will only show that
\[
    \sup_{0\leq\tau\leq
    T\wedge T_\varepsilon}
    \abs{\int_0^{\tau/\varepsilon}
    s_{1,\varepsilon}(s)\delta_{\varphi_{1,\varepsilon}(s)=1/2}-
    \frac{s_{1,\varepsilon}(s)^2}{2Q_\varepsilon(s)}ds}
    =\mathcal{O} (1).
\]
The other three terms can be handled similarly.

\paragraph*{Step 3:  A sequence of times adapted for ergodization.}

Ergodization refers to the convergence along an orbit of a
function's time average to its space average.  For example, because
of the splitting according to particles above, one can easily check
that $\frac {1} {t}\int_{0}^{t} H(z_0(s))ds =\bar H (h_0)
+\mathcal{O} (1/t) $, even when $z_0 (\cdot) $ restricted to the
invariant tori $\mathcal{M}_{h_0} $ is not ergodic.  In this step,
for each initial condition $z_\varepsilon(0) $ in our phase space,
we define a sequence of times $t_{k,\varepsilon} $ inductively as
follows: $t_{0,\varepsilon}=\inf\{t\geq
0:\varphi_{1,\varepsilon}(t)=0\}$,
$t_{k+1,\varepsilon}=\inf\{t>t_{k,\varepsilon}:\varphi_{1,\varepsilon}(t)=0\}$.
This sequence is chosen because $\delta_{\varphi_{1,0}(s)=1/2}$ is
``ergodizd'' as time passes from $t_{k,0}$ to $t_{k+1,0}$.  If
$\varepsilon$ is sufficiently small and $t_{k+1,\varepsilon}\leq
(T\wedge T_\varepsilon)/\varepsilon$, then the spacings between
these times are uniformly of order $1$, i.e.~$1/\omega
<t_{k+1,\varepsilon}-t_{k,\varepsilon}< \omega$.  Thus,
\begin {equation}\label{eq:hc_splitting}
\begin {split}
    \sup_{0\leq \tau\leq
    T\wedge T_\varepsilon}
    &
    \abs{\int_0^{\tau/\varepsilon}
    s_{1,\varepsilon}(s)\delta_{\varphi_{1,\varepsilon}(s)=1/2}-
    \frac{s_{1,\varepsilon}(s)^2}{2Q_\varepsilon(s)}ds}
    \\
    &
    \leq
    \mathcal{O} (1)+
    \sum_{t_{k+1,\varepsilon}\leq
    \frac {T\wedge T_\varepsilon}{\varepsilon}}
    \abs{
    \int_{t_{k,\varepsilon}}^{t_{k+1,\varepsilon}}
    s_{1,\varepsilon}(s)\delta_{\varphi_{1,\varepsilon}(s)=1/2}-
    \frac{s_{1,\varepsilon}(s)^2}{2Q_\varepsilon(s)}ds}.
\end {split}
\end {equation}

\paragraph*{Step 4:  Control of individual terms by comparison with
        solutions along fibers.}

The sum in Equation \eqref{eq:hc_splitting} has no more than
$\mathcal{O} (1/\varepsilon)$ terms, and so it suffices to show that
each term is no larger than $\mathcal{O} (\varepsilon) $.  We can
accomplish this by comparing the motions of $z_\varepsilon(t)$ for
$t_{k,\varepsilon}\leq t\leq t_{k+1,\varepsilon}$ with the solution
of the $\varepsilon=0$ version of Equation \eqref{eq:1dode} that, at
time $t_{k,\varepsilon}$, is located at
$z_\varepsilon(t_{k,\varepsilon})$. Since each term in the sum has
the same form, without loss of generality we will only examine the
first term and suppose that $t_{0,\varepsilon} =0$, i.e.~that
$\varphi_{1,\varepsilon} (0) =0$.

\begin {lem}
\label {lem:1d_divergence}

If $t_{1,\varepsilon}\leq\frac{T\wedge T_\varepsilon}{\varepsilon}$,
then $
    \sup_{0\leq t\leq t_{1,\varepsilon}}\abs{z_{0}(t)-
    z_\varepsilon(t)}=\mathcal{O}(\varepsilon ).
$
\end {lem}
\begin {proof}
To check that $\sup_{0\leq t\leq t_{1,\varepsilon}}\abs{h_{0}(t)-
    h_\varepsilon(t)} = \mathcal{O}(\varepsilon) $,
first note that $h_0 (t) =h_0(0) =h_\varepsilon (0) $. Then
$dQ_\varepsilon/dt =\mathcal{O} (\varepsilon) $, so that $Q_0(t)
-Q_\varepsilon(t) =\mathcal{O} (\varepsilon t) $.  Furthermore, the
other slow variables change by $\mathcal{O} (\varepsilon) $ at
collisions, while the number of collisions in the time interval $[0,
t_{1,\varepsilon}] $ is $ \mathcal{O} (1) $.

It remains to show that $\sup_{0\leq t\leq t_{1,\varepsilon}}
\abs{\varphi_{i,0}(t)-\varphi_{i,\varepsilon}(t)}=
\mathcal{O}(\varepsilon )$. Using what we know about the divergence
of the slow variables,
\[
\begin {split}
    \varphi_{1,0} (t) -\varphi_{1,\varepsilon} (t)  &=
    \int_0^t\frac{s_{1,0}(s)}{2Q_0(s)}-
    \frac{s_{1,\varepsilon}(s)}{2Q_\varepsilon(s)}+
    \mathcal{O}(\varepsilon)ds
    =\int_0 ^t
    \mathcal{O} (\varepsilon)ds
    = \mathcal{O} (\varepsilon )\\
\end {split}
\]
for $0\leq t \leq t_{1,\varepsilon} $.  Showing that $\sup_{0\leq
t\leq t_{1,\varepsilon}}
\abs{\varphi_{2,0}(t)-\varphi_{2,\varepsilon}(t)}=
\mathcal{O}(\varepsilon )$ is similar.

\end {proof}

From Lemma \ref{lem:1d_divergence},
$t_{1,\varepsilon}=t_{1,0}+\mathcal{O}(\varepsilon)
=2Q_0/s_{1,0}+\mathcal{O}(\varepsilon)$.  We conclude that
\[
\begin {split}
    \int_{0}^{t_{1,\varepsilon}}
    s_{1,\varepsilon}(s)\delta_{\varphi_{1,\varepsilon}(s)=1/2}-
    \frac{s_{1,\varepsilon}(s)^2}{2Q_\varepsilon(s)}ds
    & =
    \mathcal{O} (\varepsilon) +
    \int_{0}^{t_{1,\varepsilon}}
    s_{1,0}(s)\delta_{\varphi_{1,\varepsilon}(s)=1/2}-
    \frac{s_{1,0}(s)^2}{2Q_0(s)}ds
    \\
    &=
    \mathcal{O}(\varepsilon)+s_{1,0}
    -t_{1,\varepsilon}\frac{s_{1,0}^2}{2Q_0}
    =\mathcal{O}(\varepsilon).
\end {split}
\]

It follows that $
    \sup_{0\leq \tau\leq T\wedge T_\varepsilon}
    \abs{h_\varepsilon(\tau/\varepsilon)-\bar
    h(\tau)}
    =\mathcal{O}(\varepsilon ),
$ independent of the initial condition in $h^ {-1}\mathcal{V}$.

\subsection{Extension to multiple gas particles}

When $n_1,n_2>1$, only minor modifications are necessary to
generalize the proof above. We start by extending the slow variables
$h$ to a full set of coordinates on phase space by defining the
angle variables $\varphi _{i,j}\in [0,1]/ \, 0\sim 1=S^1$ for $1\leq
i\leq 2, $ $1\leq j\leq n_i$:
\[
\begin {split}
    \varphi_{1,j} =\varphi_{1,j} (q_{1,j},v_{1,j},Q)=
    &\begin {cases}
    \frac{q_{1,j}}{2Q} &\text { if } v_{1,j} >0\\
    1-\frac{q_{1,j}}{2Q} &\text { if } v_{1,j}<0\\
    \end {cases}
    \\
    \varphi_{2,j}=\varphi_{2,j}(q_{2,j},v_{2,j},Q)=
    &\begin {cases}
    \frac{1-q_{2,j}}{2(1-Q)} &\text { if } v_{2,j}<0\\
    1- \frac{1-q_{2,j}}{2(1-Q)} &\text { if } v_{2,j}>0\\
    \end {cases}\\
\end {split}.
\]
Then $d\varphi_{1,j}/dt=s_{1,j}(2Q) ^ {-1} +\mathcal{O}
(\varepsilon)$, $d\varphi_{2,j}/dt=s_{2,j}(2(1-Q)) ^ {-1}
+\mathcal{O} (\varepsilon)$, and $z =
(h,\varphi_{1,j},\varphi_{2,j}) $ represents a choice of coordinates
on our phase space, which is homeomorphic to $(\text{a subset of }
\mathbb{R}^{n_1+n_2+2})\times \mathbb{ T} ^ {n_1+n_2} $.  In these
coordinates, the dynamical system yields a discontinuous flow
$z_\varepsilon(t) $ on phase space.  The flow preserves Liouville
measure, which in our coordinates has a density proportional to $Q^
{n_1}(1-Q)^{n_2}$. As is Section \ref{sct:1dps_hcPS}, one can show
that the measure of initial conditions leading to multiple particle
collisions is zero.

Next, define $H(z) $ by
\[
    H(z) =
    \begin{bmatrix}
    W\\
    \sum_{j=1}^{n_1}2m_{1,j} s_{1,j} \delta_{\varphi_{1,j}=1/2}
    -\sum_{j=1}^{n_2}2m_2
     s_{2,j}\delta_{\varphi_{2,j}=1/2}\\
    -2W\delta_{\varphi_{1,j}=1/2}\\
    2W\delta_{\varphi_{2,j}=1/2}\\
    \end{bmatrix}.
\]
For $0 \leq t\leq \frac{T\wedge T_\varepsilon}{\varepsilon}$, $
    h_\varepsilon(t)-h_\varepsilon(0)=
    \varepsilon\int_0^t
    H(z_\varepsilon(s))ds+\mathcal{O}(\varepsilon).
$ From here, the rest of the proof follows the same arguments made
in Section \ref{sct:1dps_hc_unif}.

\section{Proof of the main result for the soft core piston}
\label {sct:sc_proof}

For the remainder of this chapter, we consider the family of
Hamiltonian systems introduced in Section
\ref{sct:soft_core_results}, which are parameterized by
$\varepsilon,\delta\geq 0$. For simplicity, we specialize to
$n_1=n_2=1$. As in Section \ref{sct:hc_proof}, the generalization to
$n_1, n_2>1$ is not difficult. The Hamiltonian dynamics are given by
the following ordinary differential equation:
\begin {equation}
\label {eq:smooth1dode1}
\begin {split}
    \frac{dQ}{dt}&=\varepsilon W,\\
    \frac{dW}{dt}&=\varepsilon\left(
    -\kappa_\delta'(Q-x_{1})
    +\kappa_\delta'(x_{2}-Q)\right),\\
    \frac{dx_{1}}{dt}&=v_{1},\\
    \frac{dv_{1}}{dt}&=\frac{1}{m_{1}}\bigl( -\kappa_\delta'(x_{1})
    +\kappa_\delta'(Q-x_{1})\bigr),\\
    \frac{dx_{2}}{dt}&=v_{2},\\
    \frac{dv_{2}}{dt}&=\frac{1}{m_{2}}\bigl( -\kappa_\delta'(x_{2}-Q)
    +\kappa_\delta'(1-x_{2})\bigr).\\
\end {split}
\end {equation}
Recalling the particle energies defined by Equation
\eqref{eq:soft_energies}, we find that
\begin {equation*}
\begin {split}
    \frac {d E_{1}} {dt} =\varepsilon W\kappa_\delta' (Q
    -x_{1}),\qquad
    \frac {d E_{2}} {dt} = -\varepsilon W\kappa_\delta' (x_{2} -Q).\\
\end {split}
\end {equation*}

\comment {
\begin {equation}
\label {eq:smooth1dode1}
\begin {split}
    \frac{dQ}{dt}&=\varepsilon W,\\
    \frac{dW}{dt}&=\varepsilon\left(\sum_{j=1}^{n_1}
    -\kappa_\delta'(Q-q_{1,j})
    +\sum_{j=1}^{n_2}\kappa_\delta'(q_{2,j}-Q)\right),\\
    \frac{dq_{1,j}}{dt}&=v_{1,j},\\
    \frac{dv_{1,j}}{dt}&=\frac{1}{m_{1,j}}\bigl( -\kappa_\delta'(q_{1,j})
    +\kappa_\delta'(Q-q_{1,j})\bigr),\\
    \frac{dq_{2,j}}{dt}&=v_{2,j},\\
    \frac{dv_{2,j}}{dt}&=\frac{1}{m_{2,j}}\bigl( -\kappa_\delta'(q_{2,j}-Q)
    +\kappa_\delta'(1-q_{2,j})\bigr).\\
\end {split}
\end {equation}
Recalling the particle energies $E_{i,j} $ defined by Equation
\eqref{eq:soft_energies}, we find that
\begin {equation*}
\begin {split}
    \frac {d E_{1,j}} {dt} =\varepsilon W\kappa_\delta' (Q
    -q_{1,j}),\qquad
    \frac {d E_{2,j}} {dt} = -\varepsilon W\kappa_\delta' (q_{2,j} -Q).\\
\end {split}
\end {equation*}

}

For the compact set $\mathcal{V} $ introduced in Section
\ref{sct:1d_sm_results}, fix a small positive number $\mathcal{E} $
and an open set $\mathcal{U}\subset \mathbb{R}^4 $ such that $
\mathcal{V} \subset\mathcal{U} $ and $h\in\mathcal{U}\Rightarrow
Q\in (\mathcal{E},1 -\mathcal{E})$, $W\subset\subset \mathbb {R} $,
and $\mathcal{E} <E_1,E_2 <\kappa (0) -\mathcal{E} $. We only
consider the dynamics for $0<\delta<\mathcal{E}/2$ and $h
\in\mathcal{U} $.

Define
\begin {equation*}
\begin {split}
    U_1(q_1)&=U_1(q_1,Q,\delta):=\kappa_\delta(q_1)
    +\kappa_\delta(Q-q_1),\\
    U_2(q_2)&=U_2(q_2,Q,\delta)  :=\kappa_\delta(q_2-Q)
    +\kappa_\delta(1-q_2).\\
\end {split}
\end {equation*}
Then the energies $E_i$ satisfy $E_i=m_i v_i^2/2+U_i(x_i)$.

Let $T_1=T_1(Q,E_1,\delta) $ and $T_2=T_2(Q,E_2,\delta)$ denote the
periods of the motions of the left and right gas particles,
respectively, when $\varepsilon =0 $.

\begin {lem}
\label {lem:1d_smooth_periods} For $i=1,2 $,
\[
    T_i\in\mathcal{C} ^1\{(Q,E_i,\delta):
    Q\in (\mathcal{E},1-\mathcal{E})
    ,E_i\in (\mathcal{E} ,\kappa (0) -\mathcal{E}),0\leq
    \delta <\mathcal{E}/2\} .
\]
Furthermore,
\[
\begin {split}
    T_1(Q,E_1,\delta)&=\sqrt{\frac  {2m_1} {E_1}}Q
    +\mathcal{O} (\delta),\\
    T_2(Q,E_2,\delta)&=\sqrt{\frac  {2m_2} {E_2}} (1 - Q)
    +\mathcal{O} (\delta).\\
\end {split}
\]
\end {lem}

The proof of this lemma is mostly computational, and so we delay it
until Section \ref{sct:smooth1DTech}.  Note especially that the
periods can be suitably defined such that their regularity extends
to $\delta=0$.

In this section, and in Section \ref{sct:smooth1DTech} below, we
adopt the following convention on the use of the $\mathcal{O} $
notation.  \emph{All use of the $\mathcal{O} $ notation will
explicitly contain the dependence on $\varepsilon$ and $\delta$} as
$\varepsilon, \delta\rightarrow 0$.  For example, if a function
$f(h,\varepsilon,\delta) =\mathcal{O} (\varepsilon) $, then there
exists $\delta',\varepsilon' >0$ such that $\sup_{0<\varepsilon\leq
\varepsilon',\,0<\delta\leq\delta',\,h\in\mathcal{V}}
\abs{f(h,\varepsilon,\delta)/\varepsilon} < \infty $.

When $\varepsilon =0 $,
\[
    \frac { dx_i} {dt} =\pm \sqrt {\frac {2}
    {m_i} (E_i -U_i (x_i))} .
\]
Define $a = a (E_i,\delta) $ by
\[
    \kappa_\delta (a) =\kappa
    (a/\delta) =E_i,
\]
so that $a(E_1,\delta) $ is a turning point for the left gas
particle. Then $a=\delta\kappa ^ {-1} (E_i) $, where $\kappa ^ {-1}
$ is defined as follows: $\kappa: [0,1]\rightarrow [0,\kappa (0)] $
takes $0 $ to $\kappa (0) $ and $1 $ to $0 $. Furthermore,
$\kappa\in\mathcal{C} ^2 ([0,1]) $, $\kappa '\leq 0 $, and $\kappa '
(x) <0 $ if $x < 1 $. By monotonicity, $\kappa ^ {-1}\colon
[0,\kappa (0)]\rightarrow [0,1] $ exists and takes $0 $ to $1 $ and
$\kappa (0) $ to $0 $. Also, by the Implicit Function Theorem,
$\kappa ^ {-1}\in\mathcal{C} ^2 ((0,\kappa (0)]) $, $(\kappa ^
{-1})'(y) < 0 $ for $y>0$, and $(\kappa ^ {-1})'(y)\rightarrow
-\infty $ as $y\rightarrow 0^ +$. Because we only consider energies
$E_i\in (\mathcal{E},\kappa(0)-\mathcal{E}) $, it follows that $a
(E_i,\delta)$ is a $\mathcal{C} ^2$ function for the domains of
interest.

\subsection{Derivation of the averaged equation}

\label {sct:smooth_1D_derivation}

As we previously pointed out, for each fixed $\delta>0$, Anosov's
theorem~\ref{thm:anosov} and Theorem~\ref{thm:simple_averaging3}
apply directly to the family of ordinary differential equations in
Equation \eqref{eq:smooth1dode1}, provided that $\delta$ is
sufficiently small.  The invariant fibers $\mathcal{M}_h$ of the
$\varepsilon=0$ flow are tori described by a fixed value of the four
slow variables and $\{(Q,W,q_1,v_1,q_2,v_2): E_1=m_1
v_1^2/2+U_1(q_1,Q,\delta), E_2=m_2 v_2^2/2+U_2(q_2,Q,\delta)\} $. If
we use $(q_1,q_2) $ as local coordinates on $\mathcal{M}_h$, which
is valid except when $v_1\text { or }v_2=0$, the invariant measure
$\mu_h$ of the unperturbed flow has the density
\[
    \frac {dq_1 dq_2}
    {
    T_1\sqrt {\frac {2} {m_1} (E_1-U_1(q_1))}\:
    T_2\sqrt {\frac {2} {m_2} (E_2-U_2(q_2))}
    }.
\]
The restricted flow is ergodic for almost every $h $. See Corollary
\ref{cor:irrat_periods} in Section \ref{sct:smooth1DTech}.

Now
\[
    \frac {dh_\varepsilon ^\delta} {dt}
    =\varepsilon
    \begin {bmatrix}
             W \\
         -\kappa_\delta'(Q -q_1 )
                +\kappa_\delta'(q_2-Q )\\
             W\kappa_\delta' (Q
                -q_1)\\
            - W\kappa_\delta' (
                q_2-Q )\\
    \end {bmatrix},
\]
and
\[
\begin {split}
    \int_{\mathcal{M}_h}\kappa_\delta'(Q -q_1 )d\mu_h
    &
    =\frac {2} {T_1}\int_a^{Q-a}dq_1\frac{\kappa_\delta'(Q-q_1)}
    {\sqrt {\frac {2} {m_1} (E_1-U_1(q_1))}}
    \\
    &
    =\frac {\sqrt{2m_1}} {T_1}\int_{Q-\delta}^{Q-a}dq_1\frac{\kappa_\delta'(Q-q_1)}
    {\sqrt { E_1-\kappa_\delta(Q-q_1)}}
    \\
    &
    =-\frac {\sqrt{2m_1}} {T_1}\int_{0}^{E_1}\frac{du}
    {\sqrt { E_1-u}}
    \\
    &
    = -\frac {\sqrt{8m_1E_1}} {T_1}.
    \\
\end {split}
\]
Similarly,
\[
\begin {split}
    \int_{\mathcal{M}_h}\kappa_\delta'(q_2-Q)d\mu_h
    = -\frac {\sqrt{8m_2E_2}} {T_2}.\\
\end {split}
\]
It follows that the averaged vector field is
\[
    \bar H^\delta (h) =
    \begin {bmatrix}
             W \\
         \frac{\sqrt{8m_1E_1}}{T_1}-
     \frac{\sqrt{8m_2E_2}}{T_2}\\
             -W\frac{\sqrt{8m_1E_1}}{T_1}\\
            +W\frac{\sqrt{8m_2E_2}}{T_2}\\
        \end {bmatrix},
\]
where from Lemma \ref{lem:1d_smooth_periods} we see that $\bar H^
{\cdot} (\cdot)\in\mathcal{C} ^1 (\{(\delta,h):0\leq\delta
<\mathcal{E}/2,h\in\mathcal{V}\})$.  $\bar H^0(h) $ agrees with the
averaged vector field for the hard core system from Equation
\eqref{eq:1davg}, once we account for the change of coordinates
$E_i=m_i s_i^2/2$.

\begin {rem}

An argument due to Neishtadt and Sinai~\cite{NS04} shows that the
solutions to the averaged equation \eqref{eq:smooth_1D_averaged_eq}
are periodic.  This argument also shows that, as in the case $\delta
=0$, the limiting dynamics of $(Q,W) $ are effectively Hamiltonian,
with the shape of the Hamiltonian depending on $\delta$, $Q(0) $,
and the initial energies of the gas particles. The argument depends
heavily on the observation that the phase integrals
\[
    I_i(Q, E_i,\delta)=
    \int_{\frac{1}{2}m_i v^2+U_i(x,Q,\delta)\leq E_i}
    dxdv
\]
are adiabatic invariants, i.e.~they are integrals of the solutions
to the averaged equation.  Thus the four-dimensional phase space of
the averaged equation is foliated by invariant two-dimensional
submanifolds, and one can think of the effective Hamiltonians for
the piston as living on these submanifolds.

\end {rem}

\subsection {Proof of Theorem \ref{thm:1D_smooth_uniform}}

The following arguments are motivated by our proof in Section
\ref{sct:hc_proof}, although the details are more involved as we
show that the rate of convergence is independent of all small
$\delta$.

\subsubsection{A choice of coordinates on phase space}

We wish to describe the dynamics in a coordinate system inspired by
the one used in Section \ref{sct:1dps_hcPS}.  For each fixed
$\delta\in (0,\delta_0]$, this change of coordinates will be
$\mathcal{C} ^1 $ in all variables on the domain of interest.
However, it is an exercise in analysis to show this, and so we delay
the proofs of the following two lemmas until Section
\ref{sct:smooth1DTech}.

We introduce the angular coordinates $\varphi _i\in [0,1]/ \, 0\sim
1=S^1$ defined by
\begin {equation}
\label {eq:anglevar_defn}
\begin {split}
    \varphi_1 =\varphi_1 (q_1,v_1,Q)=&\begin {cases}
    0&\text { if } q_1 =a\\
    \frac{1}{T_1} \int_a^{q_1}\sqrt{\frac{m_1/2}
        {E_1-U_1(s)}}ds&\text { if } v_1 >0\\
    1/2&\text { if } q_1 =Q -a\\
    1-\frac{1}{T_1} \int_a^{q_1}\sqrt{\frac{m_1/2}
        {E_1-U_1(s)}}ds &\text { if } v_1<0\\
    \end {cases}
    \\
    \varphi_2=\varphi_2(q_2,v_2,Q)=&\begin {cases}
    0&\text { if } q_2 =1-a\\
    \frac{1}{T_2} \int_{q_2} ^{1 - a}\sqrt{\frac{m_2/2}
        {E_2-U_2(s)}}ds&\text { if } v_2 < 0\\
    1/2&\text { if } q_2 =Q+a\\
    1-\frac{1}{T_2} \int_{q_2} ^{1 -a}\sqrt{\frac{m_2/2}
        {E_2-U_2(s)}}ds &\text { if } v_2 > 0\\
    \end {cases}\\
\end {split}.
\end{equation}
Then $z =(h,\varphi_1,\varphi_2) $ is a choice of coordinates on $h^
{-1}\mathcal{U} $. As before, we will abuse notation and let $h (z)
$ denote the projection onto the first four coordinates of $z $.

There is a fixed value of $\delta_0$ in the statement of Theorem
\ref{thm:1D_smooth_uniform}.  However, for the purposes of our
proof, it will be convenient to progressively choose $\delta_0$
smaller when needed.  At the end of the proof, we will have only
shrunk $\delta_0$ a finite number of times, and this final value
will satisfies the requirements of the theorem.  Our first
requirement on $\delta_0$ is that it is smaller than
$\mathcal{E}/2$.

\begin {lem}
\label {lem:1d_smooth_ode}

If $\delta_0 >0$ is sufficiently small, then for each $\delta\in
(0,\delta_0]$ the ordinary differential equation
\eqref{eq:smooth1dode1} in the coordinates $z $ takes the form
\begin {equation}
\label {eq:smooth1dode_2}
    \frac {dz } {dt} =Z ^\delta (z ,\varepsilon),
\end {equation}
where $Z ^\delta \in\mathcal{C} ^1 ( h  ^ {-1} \mathcal{U}\times
[0,\infty)) $. When $z \in h^{-1}\mathcal{U} $,

\begin {equation}
\label {eq:smooth1dode_3}
    Z ^\delta (z,\varepsilon) =
    \begin {bmatrix}
            \varepsilon W \\
            \varepsilon\bigl
                ( -\kappa_\delta'(Q -q_1 (z))
                +\kappa_\delta'(q_2(z)-Q )\bigr)\\
            \varepsilon W\kappa_\delta' (Q
                -q_1(z))\\
            -\varepsilon W\kappa_\delta' (
                q_2(z)-Q )\\
            \frac {1} {T_1}+\mathcal{O}
                (\varepsilon)\\
            \frac {1} {T_2}+\mathcal{O}
                (\varepsilon)\\
    \end {bmatrix}.
\end {equation}

\end {lem}

Recall that, by our conventions, the $\mathcal{O} (\varepsilon) $
terms in Equation \eqref{eq:smooth1dode_3} have a size that can be
bounded independent of all $\delta$ sufficiently small.  Denote the
flow determined by $Z ^\delta (\cdot,\varepsilon) $ by
$z_\varepsilon ^\delta (t) $, and its components by
$Q_\varepsilon^\delta (t) $, $W_\varepsilon^\delta (t) $,
$E_{1,\varepsilon}^\delta (t) $, etc. Also, set
$h_\varepsilon^\delta (t) =h(z_\varepsilon ^\delta (t) )$.  From
Equation \eqref{eq:smooth1dode_3},
\begin{equation}\label {eq:sm_H}
    H ^\delta (z,\varepsilon) :=
    \frac{1}{\varepsilon}
    \frac{dh_\varepsilon^\delta}{dt}=
    \begin {bmatrix}
            W \\
             -\kappa_\delta'(Q -q_1 (z))
                +\kappa_\delta'(q_2(z)-Q )\\
            W\kappa_\delta' (Q
                -q_1(z))\\
            -W\kappa_\delta' (
                q_2(z)-Q )\\
    \end {bmatrix}.
\end {equation}
In particular, $H ^\delta (z,\varepsilon) =H ^\delta (z,0)$.

Before proceeding, we need one final technical lemma.

\begin {lem}
\label {lem:1d_smooth_delta}

If $\delta_0>0$ is chosen sufficiently small, there exists a
constant $K$ such that for all $\delta\in (0,\delta_0]$,
$\kappa_\delta' (\abs{Q -x_i(z)}) =0 $ unless $\varphi_i\in [1/2
-K\delta,1/2 +K\delta] $.
\end {lem}

\subsubsection{Argument for uniform convergence}

We start by proving the following lemma, which essentially says that
an orbit $z_\varepsilon ^\delta (t) $ only spends a fraction
$\mathcal{O} (\delta) $ of its time in a region of phase space where
$\abs{H^\delta (z_\varepsilon ^\delta (t),\varepsilon)}
=\abs{H^\delta (z_\varepsilon ^\delta (t),0)}$ is of size
$\mathcal{O} (\delta ^ {-1}) $

\begin {lem}
\label {lem:1d_smooth_intbound}

For $0\leq\mathcal{T}'\leq\mathcal{T}\leq\frac {T\wedge
T_\varepsilon ^\delta} {\varepsilon}$,
\[
    \int_{\mathcal{T}'}^{\mathcal{T}} \abs{H ^\delta (z_\varepsilon ^\delta
    (s),0)}ds
    =\mathcal{O} (1\vee(\mathcal{T}-\mathcal{T}')).
\]
\end {lem}

\begin {proof}

Without loss of generality, $\mathcal{T}' =0$.  From Lemmas
\ref{lem:1d_smooth_periods} and \ref{lem:1d_smooth_ode} it follows
that if we choose $\delta_0$ sufficiently small, then there exists
$\omega
>0 $ such that for all sufficiently small $\varepsilon $ and all
$\delta\in (0,\delta_0]$, $h\in\mathcal{V}\Rightarrow 1/\omega
<\frac {d\varphi_{i,\varepsilon} ^\delta} {dt} <\omega $.   Define
the set $B = [1/2 -K\delta, 1/2+K\delta] $, where $K $ comes from
Lemma \ref{lem:1d_smooth_delta}.  Then we find a crude bound on
$\int_0^{\mathcal{T}} \abs {\kappa_\delta' \bigl(Q_\varepsilon
^\delta (s) -q_1 (z_\varepsilon ^ \delta(s))\bigr)}ds $ using that
\[
    \frac {d\varphi_{1,\varepsilon} ^\delta} {dt}\text { is }
    \begin {cases}
    \geq 1/\omega &\text { if } \varphi_{1,\varepsilon} ^\delta\in B\\
    \leq \omega &\text { if } \varphi_{1,\varepsilon} ^\delta\in B^c.\\
    \end {cases}
\]
This yields
\[
\begin{split}
    \int_0^{\mathcal{T}}  \abs {\kappa_\delta'
    \bigl( Q_\varepsilon ^\delta(s)
    -q_1
    (z_\varepsilon ^ \delta(s))  \bigr)}ds
    &
    \leq\frac { \C} {\delta}\int_0^{\mathcal{T}}
    1_{\varphi_{1,\varepsilon} ^\delta (s)\in B}ds\\
    & \leq \frac { \C} {\delta} \left(\frac {2K\omega\delta}
    {2K\omega\delta+\frac{1-2K\delta}{\omega}}\mathcal{T}
    +2K\omega\delta\right)
    \\
    &
    =\mathcal{O}(1\vee\mathcal{T}).\\
\end{split}
\]
Similarly, $\int_0^{\mathcal{T}} \abs {\kappa_\delta' (
q_2(z_\varepsilon ^ \delta(s))-Q_\varepsilon ^\delta(s))}ds
    =\mathcal{O}(1\vee\mathcal{T})$, and so $\int_0^{\mathcal{T}}
    \abs{H ^\delta (z_\varepsilon ^\delta
    (s),0)}ds= \mathcal{O}(1\vee\mathcal{T})$.
\end {proof}

We now follow steps one through four from Section
\ref{sct:1dps_hc_unif}, making modifications where necessary.

\paragraph*{Step 1:  Reduction using Gronwall's Inequality.}

Now $h_\varepsilon^\delta(\tau/\varepsilon)$ satisfies
\[
    h_\varepsilon^\delta(\tau/\varepsilon)-h_\varepsilon^\delta(0)    =
    \varepsilon\int_0^{\tau/\varepsilon}
    H^\delta(z_\varepsilon^\delta(s),0)ds.
\]
Define
\[
    e_\varepsilon^\delta(\tau)
    =\varepsilon\int_0^{\tau/\varepsilon}
    H^\delta(z_\varepsilon^\delta(s),0)- \bar
    H^\delta(h_\varepsilon^\delta(s))ds.
\]
It follows from Gronwall's Inequality and the fact that $\bar H^
{\cdot} (\cdot)\in\mathcal{C} ^1 (\{(\delta,h):0\leq\delta
\leq\delta_0,h\in\mathcal{V}\})$ that
\begin {equation}
\label{eq:1d_smooth_gronwall}
\begin {split}
    \sup_{0\leq \tau\leq T\wedge T_\varepsilon ^\delta}
    \abs{h_\varepsilon ^\delta (\tau/\varepsilon)-\bar
    h ^\delta (\tau)}
    &\leq
    \left(\sup_{0\leq \tau\leq T\wedge T_\varepsilon ^\delta}
    \abs{e_\varepsilon^\delta (\tau)}\right)
    e^{ \Lip{\bar
    H ^\delta \arrowvert _\mathcal{V}} T}
    \\
    &=\mathcal{O}\left(\sup_{0\leq \tau\leq T\wedge T_\varepsilon ^\delta}
    \abs{e_\varepsilon^\delta (\tau)}\right).
\end {split}
\end {equation}

\paragraph*{Step 2:  A splitting according to particles.}

Next,
\[
\begin {split}
    H^\delta(z,0)
    &-\bar H^\delta(h)
    \\
    &=
    \begin{bmatrix}
    0\\
    -\kappa_\delta'(Q-q_1(z))-\frac {\sqrt{8m_1 E_1}} {T_1}\\
    W\kappa_\delta'(Q-q_1(z))+W\frac {\sqrt{8m_1 E_1}} {T_1}\\
    0\\
    \end{bmatrix}
    +
    \begin{bmatrix}
    0\\
    \kappa_\delta'(q_2(z)-Q)+\frac {\sqrt{8m_2 E_2}} {T_2}\\
    0\\
    -W\kappa_\delta'(q_2(z)-Q)-W\frac {\sqrt{8m_2 E_2}} {T_2}\\
    \end{bmatrix}
    ,
\end {split}
\]
and so, in order to show that $\sup_{0\leq \tau\leq T\wedge
T_\varepsilon^\delta}\abs{e_\varepsilon^\delta(\tau)} =\mathcal{O}
(\varepsilon) $, it suffices to show that for $i=1,2$,
\begin {equation*}
\begin {split}
    \sup_{0\leq \tau\leq
    T\wedge T_\varepsilon^\delta}
    &
    \abs{\int_0^{\tau/\varepsilon}
    \kappa_\delta'\bigl(\abs{Q_\varepsilon^\delta(s)-
    x_i(z_\varepsilon^\delta(s))}\bigr)+
    \frac {\sqrt{8m_i E_{i,\varepsilon} ^\delta(s)}} {T_i(Q_{\varepsilon}^\delta(s),
     E_{i,\varepsilon}^\delta(s),
    \delta)}ds}
    =
    \mathcal{O} (1),
    \\
    \sup_{0\leq \tau\leq
    T\wedge T_\varepsilon^\delta}
    &
    \abs{\int_0^{\tau/\varepsilon}W_\varepsilon(s)
    \kappa_\delta'\bigl(\abs{Q_\varepsilon^\delta(s)-
    x_i(z_\varepsilon^\delta(s))}\bigr)+W_\varepsilon(s)
    \frac {\sqrt{8m_i E_{i,\varepsilon} ^\delta(s)}} {T_i(Q_{\varepsilon}^\delta(s),
     E_{i,\varepsilon}^\delta(s),
    \delta)}ds}
    \\&=
    \mathcal{O} (1).\\
\end {split}
\end {equation*}
We only demonstrate that
\[
    \sup_{0\leq \tau\leq
    T\wedge T_\varepsilon^\delta}
    \abs{\int_0^{\tau/\varepsilon} \kappa_\delta'\bigl(Q_\varepsilon^\delta(s)-
    q_1(z_\varepsilon^\delta(s))\bigr)+
    \frac {\sqrt{8m_1 E_{1,\varepsilon} ^\delta(s)}} {T_1(Q_{\varepsilon}^\delta(s),
     E_{1,\varepsilon}^\delta(s),
    \delta)}ds}
    =
    \mathcal{O} (1).
\]
The other three terms are handled similarly.

\paragraph*{Step 3:  A sequence of times adapted for ergodization.}

Define the sequence of times $t_{k,\varepsilon}^\delta $ inductively
by $t_{0,\varepsilon}^\delta=\inf\{t\geq
0:\varphi_{1,\varepsilon}^\delta(t)=0\}$,
$t_{k+1,\varepsilon}^\delta=\inf\{t>t_{k,\varepsilon}^\delta:
\varphi_{1,\varepsilon}^\delta(t)=0\}$. If $\varepsilon$ and
$\delta$ are sufficiently small and $t_{k+1,\varepsilon}^\delta\leq
(T\wedge T_\varepsilon^\delta)/\varepsilon$, then it follows from
Lemma \ref{lem:1d_smooth_ode} and the discussion in the proof of
Lemma~\ref{lem:1d_smooth_intbound} that $1/\omega
<t_{k+1,\varepsilon}^\delta-t_{k,\varepsilon}^\delta< \omega$. From
Lemmas \ref{lem:1d_smooth_ode} and \ref{lem:1d_smooth_intbound} it
follows that
\begin {equation}\label{eq:sc_splitting}
\begin {split}
    &\sup_{0\leq \tau\leq
    T\wedge T_\varepsilon^\delta}
    \abs{\int_0^{\tau/\varepsilon} \kappa_\delta'\bigl(Q_\varepsilon^\delta(s)-
    q_1(z_\varepsilon^\delta(s))\bigr)+
    \frac {\sqrt{8m_1 E_{1,\varepsilon} ^\delta(s)}} {T_1(Q_{\varepsilon}^\delta(s),
     E_{1,\varepsilon}^\delta(s),
    \delta)}ds}
    \\
    &\leq
    \mathcal{O} (1)+
    \sum_{t_{k+1,\varepsilon}^\delta\leq
    \frac {T\wedge T_\varepsilon^\delta}{\varepsilon}}
    \abs{\int_{t_{k,\varepsilon}^\delta}^{t_{k+1,\varepsilon}^\delta}
     \kappa_\delta'\bigl(Q_\varepsilon^\delta(s)-
    q_1(z_\varepsilon^\delta(s))\bigr)+
    \frac {\sqrt{8m_1 E_{1,\varepsilon} ^\delta(s)}} {T_1(Q_{\varepsilon}^\delta(s),
    E_{1,\varepsilon}^\delta(s),
    \delta)}ds}.
\end {split}
\end {equation}

\paragraph*{Step 4:  Control of individual terms by comparison with
        solutions along fibers.}

As before, it suffices to show that each term in the sum in Equation
\eqref{eq:sc_splitting} is no larger than $\mathcal{O} (\varepsilon)
$. Without loss of generality we will only examine the first term
and suppose that $t_{0,\varepsilon}^\delta =0$, i.e.~that
$\varphi_{1,\varepsilon}^\delta (0) =0$.

\begin {lem}
\label {lem:1d_smooth_divergence}

If $t_{1,\varepsilon}^\delta\leq\frac{T\wedge
T_\varepsilon^\delta}{\varepsilon}$, then $
    \sup_{0\leq t\leq t_{1,\varepsilon}^\delta}\abs{z_{0}^\delta(t)-
    z_\varepsilon^\delta(t)}=\mathcal{O}(\varepsilon ).
$
\end {lem}
\begin {proof}
By Lemma \ref{lem:1d_smooth_intbound}, $h_{0} ^\delta
(t)-h_\varepsilon ^\delta (t) =h_{\varepsilon} ^\delta
(0)-h_\varepsilon ^\delta (t) = -\varepsilon\int_0 ^t H^\delta
(z_{\varepsilon} ^\delta (s),0)ds=\mathcal{O} (\varepsilon (1\vee
t))$ for $t\geq 0$.

Using what we know about the divergence of the slow variables, we
find that
\[
\begin {split}
    \varphi_{1,0} ^\delta (t) -\varphi_{1,\varepsilon} ^\delta (t)
    &=
    \int_0^t\frac{1}{T_1 (Q_0 ^\delta (s),E_0 ^\delta (s),\delta)}-
    \frac{1}{T_1 (Q_\varepsilon ^\delta (s),E_\varepsilon ^\delta
    (s),\delta)}+
    \mathcal{O}(\varepsilon)ds\\
    &=\int_0 ^t
    \mathcal{O} (\varepsilon)ds\\
    &= \mathcal{O} (\varepsilon )\\
\end {split}
\]
for $0\leq t \leq t_{1,\varepsilon}^\delta$. Lemmas
\ref{lem:1d_smooth_periods} and \ref{lem:1d_smooth_ode} ensure the
desired uniformity in the sizes of the orders of magnitudes.
 Showing that $\sup_{0\leq t\leq t_{1,\varepsilon}^\delta} \abs{\varphi_{2,0}
^\delta (t)-\varphi_{2,\varepsilon} ^\delta (t)}=
\mathcal{O}(\varepsilon )$ is similar.

\end {proof}

From Lemma \ref{lem:1d_smooth_divergence} we find that
$t_{1,\varepsilon}=t_{1,0}+\mathcal{O}(\varepsilon) =T_1 (Q_0
^\delta,E_0 ^\delta,\delta)+\mathcal{O}(\varepsilon)$.  Hence
\[
\begin {split}
    \int_{0}^{t_{1,\varepsilon}^\delta}
    \frac {\sqrt{8m_1 E_{1,\varepsilon} ^\delta(s)}} {T_1(Q_{\varepsilon}^\delta(s),
    E_{1,\varepsilon}^\delta(s),
    \delta)}ds
    &=
    \mathcal{O} (\varepsilon) +
    \int_{0}^{t_{1,0}^\delta}
    \frac {\sqrt{8m_1 E_{1,0} ^\delta}} {T_1(Q_{0}^\delta,
    E_{1,0}^\delta,
    \delta)}ds
    \\
    &=\mathcal{O} (\varepsilon) +\sqrt{8m_1 E_{1,0} ^\delta}.
\end {split}
\]
But when $q_1(z_\varepsilon^\delta)<Q_\varepsilon^\delta-a$,
\[
\begin {split}
    \frac {d} {ds}
    &
    \sqrt{E_{1,\varepsilon}^\delta(s)
    -\kappa_\delta\bigl(Q_\varepsilon^\delta(s)
    -q_1(z_\varepsilon^\delta(s))\bigr)}
    =
    \frac {\text {sign} \bigl(v_1(z_\varepsilon^\delta(s))\bigr)
    \kappa_\delta ' \bigl(Q_\varepsilon^\delta(s)
        -q_1(z_\varepsilon^\delta(s))\bigr)} {\sqrt{2m_1}},
    \\
\end {split}
\]
and so
\[
\begin {split}
    \int_{0}^{t_{1,\varepsilon}^\delta}\kappa_\delta'\bigl(Q_\varepsilon^\delta(s)-
    q_1(z_\varepsilon^\delta(s))\bigr) ds
    &
    =-\sqrt{2m_1 E_{1,\varepsilon} ^\delta(0)}
    -\sqrt{2m_1 E_{1,\varepsilon} ^\delta(t_{1,\varepsilon}^\delta)}
    \\
    &=\mathcal{O} (\varepsilon) -\sqrt{8m_1 E_{1,0} ^\delta}.
\end{split}
\]
Hence,
\[
    \int_{0}^{t_{1,\varepsilon}^\delta}\kappa_\delta'\bigl(Q_\varepsilon^\delta(s)-
    q_1(z_\varepsilon^\delta(s))\bigr)+
    \frac {\sqrt{8m_1 E_{1,\varepsilon} ^\delta(s)}} {T_1(Q_{\varepsilon}^\delta(s),
    E_{1,\varepsilon}^\delta(s),
    \delta)}ds=\mathcal{O} (\varepsilon),
\]
as desired.


\section{Appendix to Section \ref{sct:sc_proof}}
\label {sct:smooth1DTech}

\paragraph* {Proof of Lemma \ref{lem:1d_smooth_periods}:}
\begin{proof}

For $0<\delta <\mathcal{E}/2 $,
\[
\begin {split}
    T_1=T_1(Q,E_1,\delta)=2\int_a^{Q-a}\sqrt{\frac{m_1/2}
        {E_1-U_1(s)}}ds,\\
    T_2=T_2(Q,E_2,\delta)=2\int_{Q+a}^{1-a}\sqrt{\frac{m_2/2}
        {E_2-U_2(s)}}ds.\\
\end {split}
\]
We only consider the claims about $T_1 $, and for convenience we
take $m_1 =2 $.  Then
\[
\begin {split}
    T_1(Q,E_1,\delta)&=
    2\int_a^{Q-a}\frac{ds}{\sqrt{E_1-U_1(s)}}=
    4\int_a^{Q/2}\frac{ds}{\sqrt{E_1-\kappa_\delta(s)}}\\
    &=
    4\left(\frac{Q/2-\delta}{\sqrt{E_1}}+
    \int_a^{\delta}\frac{ds}{\sqrt{E_1-\kappa_\delta(s)}}\right)
    \\
    &
    =
    \frac{2Q-4\delta}{\sqrt{E_1}}+
    4\delta\int_{\kappa^{-1}(E_1)}^1
    \frac{ds}{\sqrt{E_1-\kappa(s)}}.\\
\end {split}
\]

Define
\[
    F (E) :=\int_{\kappa^{-1}(E)}^1
    \frac{ds}{\sqrt{E-\kappa(s)}}
    =
    \int_{0}^E
    \frac{- (\kappa ^ {-1})' (u)}{\sqrt{E-u}}du.
\]
Notice that $ (\kappa ^ {-1})' (u)$ diverges as $u\rightarrow 0 ^ +
$, while $(E-u) ^ {-1/2} $ diverges as $u\rightarrow E ^ - $, but
both functions are still integrable on $[0,E] $.  It follows that $F
(E) $ is well defined. Then it suffices to show that $F:
[\mathcal{E},\kappa (0) -\mathcal{E}]\rightarrow \mathbb{R} $ is
$\mathcal{C} ^1 $.

Write
\[
\begin {split}
    F (E)
    &=
    \int_0 ^ {\mathcal{E}/2}
    \frac{- (\kappa ^ {-1})' (u)}{\sqrt{E-u}}du +
    \int_{\mathcal{E}/2} ^E
    \frac{- (\kappa ^ {-1})' (u)}{\sqrt{E-u}}du\\
    &:=
    F_1 (E) +F_2 (E).\\
\end {split}
\]
A standard application of the Dominated Convergence Theorem allows
us to differentiate inside the integral and conclude that
$F_1\in\mathcal{C} ^ {\infty} ([\mathcal{E} ,\kappa (0)
-\mathcal{E}]) $, with
\[
    F_1' (E) =\int_0 ^ {\mathcal{E}/2}
    \frac{ (\kappa ^ {-1})' (u)}{2 (E-u) ^ {3/2}}du.
\]

To examine $F_2 $, we make the substitution $v =E-u $ to find that
\[
    F_2 (E) =\int_0 ^ {E-\mathcal{E}/2}
    \frac{- (\kappa ^ {-1})' (E-v)}{\sqrt{v}}dv.
\]
Using the fact that $(\kappa ^ {-1})'\in\mathcal{C} ^1 ([\mathcal{E}
/2,\kappa (0)]) $ and the Dominated Convergence Theorem, we find
that $F_2 $ is differentiable, with
\[
    F_2' (E) =\frac{- (\kappa ^ {-1})' (\mathcal{E}/2)}
    {\sqrt{E-\mathcal{E}/2}} +
    \int_0 ^ {E-\mathcal{E}/2}
    \frac{- (\kappa ^ {-1})'' (E-v)}{\sqrt{v}}dv.
\]
Another application of the Dominated Convergence Theorem shows that
$F_2' $ is continuous, and so $F_2\in\mathcal{C} ^ {1} ([\mathcal{E}
,\kappa (0) -\mathcal{E}]) $.

Thus
\[
    T_1(Q,E_1,\delta)=
    \frac{2Q}{\sqrt{E_1}}+
    4\delta\left[-E_1 ^ {-1/2} +F_1 (E_1) +F_2 (E_1)\right]
\]
has the desired regularity.  For future reference, we note that
\begin {equation} \label {eq:partial_periods}
\begin {split}
    \frac {\partial T_1} {\partial Q} =
    \frac {2} {\sqrt E_1},
    \quad
    \frac {\partial T_1} {\partial E_1} =
    \frac {-Q} {E_1 ^ {3/2}} +\mathcal{O} (\delta).
\end {split}
\end {equation}

\end {proof}

\begin {cor}
\label {cor:irrat_periods} For all $\delta$  sufficiently small, the
flow $z_0 ^\delta (t) $ restricted to the invariant tori
$\mathcal{M}_c =\{h =c\} $ is ergodic (with respect to the invariant
Lebesgue measure) for almost every $c\in \mathcal{U}$.

\end {cor}

\begin {proof}
The flow is ergodic whenever the periods $T_1 $ and $T_2 $ are
irrationally related.  Fix $\delta $ sufficiently small such that $
\frac {\partial T_1} {\partial E_1} =
    -Q/E_1 ^ {3/2} +\mathcal{O} (\delta) < 0 $.  Next, consider
$Q $, $W$, and $E_2 $ fixed, so that $T_2 $ is constant.  Because
$T_1\in\mathcal{C}^1 $, it follows that, as we let $E_1 $ vary,
$\frac {T_1} {T_2}\notin\mathbb {Q} $ for almost every $E_1 $.  The
result follows from Fubini's Theorem.
\end {proof}

\paragraph* {Proof of Lemma \ref{lem:1d_smooth_ode}:}
\begin{proof}
For the duration of this proof, we consider the dynamics for a
small, fixed value of $\delta >0 $, which we generally suppress in
our notation.  For convenience, we take $m_1 =2 $.

Let $\psi $ denote the map taking $(Q,W,q_1,v_1,q_2,v_2) $ to
$(Q,W,E_1,E_2,\varphi_1,\varphi_2) $.  We claim that $\psi$ is a
$\mathcal{C} ^1 $ change of coordinates on the domain of interest.
Since $E_1 =v_1 ^2+\kappa_\delta(q_1)+\kappa_\delta(Q-q_1)$, $E_1$
is a $\mathcal{C} ^2 $ function of $q_1, v_1, $ and $Q$.  A similar
statement holds for $E_2 $.

The angular coordinates $\varphi_i (x_i,v_i,Q) $ are defined by
Equation \eqref{eq:anglevar_defn}.  We only consider $\varphi_1 $,
as the statements for $\varphi_2 $ are similar.  Then $\varphi_1
(q_1,v_1,Q) $ is clearly $\mathcal{C} ^1 $ whenever $q_1\neq a,Q -a
$.  The apparent difficulties in regularity at the turning points
are only a result of how the definition of $\varphi_1 $ is presented
in Equation \eqref{eq:anglevar_defn}.  Recall that the angle
variables are actually defined by integrating the elapsed time along
orbits, and our previous definition expressed $\varphi_1 $ in a
manner which emphasized the dependence on $q_1 $.  In fact, whenever
$\abs{v_1} <\sqrt E_1 $,
\begin {equation}
\label {eq:anglevar_defn2}
\begin {split}
    \varphi_1 (q_1,v_1,Q)=\begin {cases}
    -\frac{2}{T_1} \int_0^{v_1}(\kappa_\delta^{-1})'
        (E_1-v^2)dv&\text { if } q_1 <\delta\\
    \frac{1}{2}+\frac{2}{T_1} \int_0^{v_1}(\kappa_\delta^{-1})'
        (E_1-v^2)dv&\text { if } q_1 >Q -\delta.\\
    \end {cases}
    \\
\end {split}
\end {equation}
Here $E_1$ is implicitly considered to be a function of $q_1, v_1,$
and $Q$.  One can verify that $D\psi $ is non-degenerate on the
domain of interest, and so $\psi$ is indeed a $\mathcal{C} ^ 1 $
change of coordinates.

Next observe that $d\varphi_{1,0}/dt=1/T_1$, so Hadamard's Lemma
implies that
\[
    \frac {d\varphi_{1,\varepsilon}} {dt}=\frac {1} {T_1} +\mathcal{O}
    (\varepsilon f(\delta)).
\]
It remains to show that, in fact, we may take $f(\delta) =1$. It is
easy to verify this whenever $q_1\leq Q -\delta $ because $dE_1/dt=
0 $ there.  We only perform the more difficult verification when
$q_1 > Q -\delta $.

When $q_1 > Q -\delta $, $\abs {v_1} <\sqrt E_1 $ and $E_1 =v_1 ^2
+\kappa_\delta (Q-q_1) $.  From Equation \eqref{eq:anglevar_defn2}
we find that
\begin {equation}\label {eq:anglevar_defn3}
    \varphi_1 =    \frac{1}{2}+\frac{2\delta}{T_1(Q,E_1,\delta)}
    \int_0^{v_1}(\kappa^{-1})'
        (E_1-v^2)dv.
\end {equation}
To find $d\varphi_1/dt $, we consider $\varphi_1 $ as a function of
$v_1, Q, $ and $E_1 $, so that
\[
    \frac {d\varphi_1}{dt}=
    \frac {\partial \varphi_1} {\partial v_1}\frac {d v_1} {dt} +
    \frac {\partial \varphi_1} {\partial Q}\frac {d Q} {dt} +
    \frac {\partial \varphi_1} {\partial E_1}\frac {d E_1} {dt}.
\]
Then, using Equations \eqref{eq:partial_periods} and
\eqref{eq:anglevar_defn3}, we compute
\[
\begin {split}
    \frac {\partial \varphi_1} {\partial v_1}
    \frac {dv_1} {dt}
    & =
    \frac {2} {T_1}(\kappa_\delta^{-1})'(E_1-v_1^2)
    \frac {\kappa_\delta' (Q -q_1)} {2}
    =
    \frac {1} {T_1},
    \\
    \frac {\partial \varphi_1} {\partial Q}
    \frac {dQ} {dt}
    & =
    \frac {1/2 -\varphi_1} {T_1}
    \frac {\partial T_1} {\partial Q}
    (\varepsilon W)
    =
    \varepsilon W\frac {1/2 -\varphi_1} {T_1}
    \frac {2} {\sqrt E_1},
    \\
    \frac {\partial \varphi_1} {\partial E_1}
    \frac {dE_1} {dt}
    & =
    \left(
    \frac {1/2 -\varphi_1} {T_1}
    \frac {\partial T_1} {\partial E_1}
    +
    \frac{2\delta}{T_1}\int_0^{v_1}(\kappa^{-1})''(E_1-v^2)dv
    \right)
    (\varepsilon W \kappa_\delta' (Q -q_1)).
\end {split}
\]
Using that $\kappa_\delta'(Q-q_1) =\kappa' (\kappa^ {-1}
(E_1-v_1^2))/\delta=(\delta(\kappa^ {-1})' (E_1-v_1^2))^ {-1} $, we
find that
\[
    \frac {\partial \varphi_1} {\partial E_1}
    \frac {dE_1} {dt}
    =
    \varepsilon\mathcal{O}\left (\frac {1/2-\varphi_1}
    {\delta}\right)
    +
    \varepsilon\mathcal{O}
    \left(
    \frac {1}
    {(\kappa^ {-1})' (E_1-v_1^2)}
    \int_0^{v_1}(\kappa^{-1})''(E_1-v^2)dv
    \right).
\]
But here $1/2-\varphi_1$ is $\mathcal{O}(\delta)$.  See the proof of
Lemma \ref{lem:1d_smooth_delta} below.  Thus the claims about
$d\varphi_1/dt$ will be proven, provided we can uniformly bound
\[
    \frac{1}{(\kappa^ {-1})'(E_1-v_1^2)}
    \int_0^{v_1}(\kappa^{-1})''(E_1-v^2)dv.
\]
Note that the apparent divergence of the integral as
$\abs{v_1}\rightarrow\sqrt {E_1}  $ is entirely due to the fact that
our expression for $\varphi_1$ from Equation
\eqref{eq:anglevar_defn3} requires $\abs {v_1} < \sqrt E_1 $. If we
make the substitution $u=E_1-v^2$ and let $e=E_1-v_1^2$, then it
suffices to show that
\[
    \sup_{\mathcal{E}\leq E_1\leq\kappa(0)-\mathcal{E}}
    \;
    \sup_{0<e\leq E_1}
    \abs{\frac{1}{(\kappa^ {-1})' (e)}
    \int_e^{E_1}\frac{(\kappa^{-1})''(u)}{\sqrt{E_1-u}}du}
    <+\infty.
\]
The only difficulties occur when $e$ is close to $0$.  Thus it
suffices to show that
\[
    \sup_{\mathcal{E}\leq E_1\leq\kappa(0)-\mathcal{E}}
    \;
    \sup_{0<e\leq \mathcal{E}/2}
    \abs{\frac{1}{(\kappa^ {-1})' (e)}
    \int_e^{\mathcal{E}/2}\frac{(\kappa^{-1})''(u)}{\sqrt{E_1-u}}du}
\]
is finite.  But this is bounded by
\[
\begin {split}
    \sup_{0<e\leq \mathcal{E}/2}
    &
    \abs{\frac{1}{(\kappa^ {-1})' (e)}
    \int_e^{\mathcal{E}/2}\frac{(\kappa^{-1})''(u)}{\sqrt{\mathcal{E}/2}}du}
    \\
    &=
    \sup_{0<e\leq \mathcal{E}/2}
    \abs{\frac{\sqrt{2/\mathcal{E}}}{(\kappa^ {-1})' (e)}
    \bigl((\kappa^ {-1})' (\mathcal{E}/2) -(\kappa^ {-1})' (e)\bigr) }
    ,
\end{split}
\]
which is finite because $(\kappa^ {-1})' (e)\rightarrow -\infty $ as
$e\rightarrow 0^ + $.  The claims about $d\varphi_2/dt$ can be
proven similarly.

\end {proof}

\paragraph* {Proof of Lemma \ref{lem:1d_smooth_delta}:}

\begin{proof}
We continue in the notation of the proofs of Lemmas
\ref{lem:1d_smooth_periods} and \ref{lem:1d_smooth_ode} above, and
we set $m_1 =2 $.  Then from Equation \eqref{eq:anglevar_defn3}, we
see that $\kappa_\delta' (Q-q_1) =0 $ unless $\abs{\varphi_1
-1/2}\leq \abs{\frac{2\delta}{T_1} \int_0^{\sqrt
E_1}(\kappa^{-1})'(E_1-v^2)dv} = \delta F (E_1)/T_1 =\mathcal{O}
(\delta)  $.  Dealing with $\varphi_2 $ is similar.

\end {proof}

\chapter{The periodic oscillation of an adiabatic piston in
two or three dimensions}\label{chp:dDpiston}

In this chapter, we present our results for the piston system in two
or three dimensions. These results may also be found
in~\cite{Wri06b}.

\section{Statement of the main result}\label{sct:main_result}

\subsection{Description of the model}\label{sct:model}

Consider a massive, insulating piston of mass $M$ that separates a
gas container $\mathcal{D} $ in $\mathbb{R}^d$, $ d= 2\text { or
}3$. See Figure~\ref{fig:domain1}. Denote the location of the piston
by $Q$, its velocity by $dQ/dt=V$, and its cross-sectional length
(when $ d=2$, or area, when $ d=3$) by $\ell$. If $Q$ is fixed, then
the piston divides $\mathcal{D} $ into two subdomains,
$\mathcal{D}_1(Q) =\mathcal{D}_1 $ on the left and $\mathcal{D}_2(Q)
=\mathcal{D}_2  $ on the right. By $E_i$ we denote the total energy
of the gas inside $\mathcal{D}_i$, and by $\abs{\mathcal{D}_i} $ we
denote the area (when $ d=2$, or volume, when $ d=3$) of
$\mathcal{D}_i$.

We are interested in the dynamics of the piston when the system's
total energy is bounded and $M\rightarrow \infty $.  When
$M=\infty$, the piston remains fixed in place, and each energy $E_i$
remains constant. When $M$ is large but finite, $MV^2/2$ is bounded,
and so $V=\mathcal{O} (M^ {-1/2}) $.  It is natural to define
\[
\begin {split}
    \varepsilon&=M^ {-1/2},\\ W& =\frac {V} {\varepsilon},
\end {split}
\]
so that $W$ is of order $1$ as $\varepsilon\rightarrow 0$.  This is
equivalent to scaling time by $\varepsilon$.

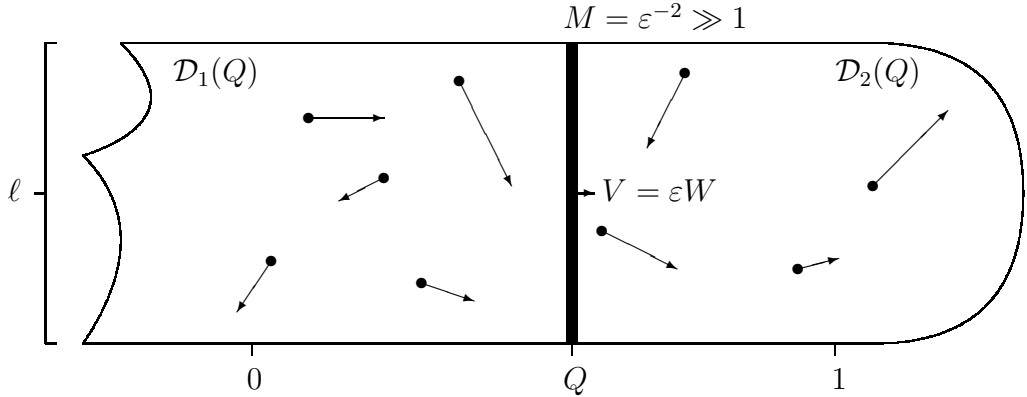
\begin{figure}
    \begin {center}
    \setlength{\unitlength}{1 cm}
    \begin{picture}(15,6)
        \put(2.5,1){\line(1,0){10}}
        \put(2.5,5){\line(1,0){10}}
        \qbezier(2,3.5)(3.5,4)(2.5,5)
        \qbezier(2,3.5)(3,2.5)(2,1)
        \put(2,1){\line(1,0){1}}
        \qbezier(12.5,1)(14.5,1)(14.5,3)
        \qbezier(14.5,3)(14.5,5)(12.5,5)
        \put(3.2,4.5){$\mathcal{D}_1(Q)$}
        \put(12,4.5){$\mathcal{D}_2(Q)$}
        \put(1.5,1){\line(0,1){4}}
        \put(1.5,1){\line(1,0){0.15}}
        \put(1.5,5){\line(1,0){0.15}}
        \put(1.5,3){\line(-1,0){0.15}}
        \put(1,2.9){$\ell$}
        \linethickness {0.15cm}
        \put(8.5,1){\line(0,1){4}}
        \thinlines
        \put(8.5,3){\vector(1,0){.3}}
        \put(8.9,2.9){$V=\varepsilon W$}
        \put(8.5,1){\line(0,-1){0.15}}
        \put(8.38,0.4){$Q$}
        \put(8.38,5.2){$M=\varepsilon ^ {-2}\gg 1$}
        \put(4.25,1){\line(0,-1){0.15}}
        \put(4.18,0.4){$0$}
        \put(12,1){\line(0,-1){0.15}}
        \put(11.95,0.4){$1$}
        \put(5,4){\circle*{.15}}
        \put(5,4){\vector(1,0){1}}
        \put(7,4.5){\circle*{.15}}
        \put(7,4.5){\vector(1,-2){0.7}}
        \put(6,3.2){\circle*{.15}}
        \put(6,3.2){\vector(-2,-1){0.6}}
        \put(6.5,1.8){\circle*{.15}}
        \put(6.5,1.8){\vector(3,-1){0.7}}
        \put(4.5,2.1){\circle*{.15}}
        \put(4.5,2.1){\vector(-2,-3){0.45}}
        \put(8.9,2.5){\circle*{.15}}
        \put(8.9,2.5){\vector(2,-1){1}}
        \put(11.5,2.0){\circle*{.15}}
        \put(11.5,2.0){\vector(4,1){0.55}}
        \put(10,4.6){\circle*{.15}}
        \put(10,4.6){\vector(-1,-2){0.5}}
        \put(12.5,3.1){\circle*{.15}}
        \put(12.5,3.1){\vector(1,1){1}}
    \end{picture}
    \end {center}
    \caption{A gas container $\mathcal{D}\subset \mathbb{R}^2 $ separated by a piston.}
    \label{fig:domain1}
\end{figure}

Next we precisely describe the gas container.  It is a compact,
connected billiard domain $\mathcal{D} \subset\mathbb{R}^d$ with a
piecewise $\mathcal{C} ^3$ boundary, i.e.~$\partial\mathcal{D} $
consists of a finite number of $\mathcal{C} ^3$ embedded
hypersurfaces, possibly with boundary and a finite number of corner
points.  The container consists of a ``tube,'' whose perpendicular
cross-section $\mathcal{P} $ is the shape of the piston, connecting
two disjoint regions.  $\mathcal{P} \subset \mathbb{R} ^ {d-1} $ is
a compact, connected domain whose boundary is piecewise $\mathcal{C}
^3$.  Then the ``tube'' is the region $[0,1]\times
\mathcal{P}\subset\mathcal{D} $ swept out by the piston for $ 0\leq
Q\leq 1$, and $[0,1]\times
\partial\mathcal{P}\subset\partial\mathcal{D} $.  If $d=2$, $\mathcal{P} $
is just a closed line segment, and the ``tube'' is a rectangle.  If
$ d=3$, $\mathcal{P} $ could be a circle, a square, a pentagon, etc.

Our fundamental assumption is as follows:
\begin{main_assu}
For almost every $Q\in [0,1]$ the billiard flow of a single particle
on an energy surface in either of the two subdomains
$\mathcal{D}_i(Q)$ is ergodic (with respect to the invariant
Liouville measure).
\end {main_assu}
\noindent If $ d=2$, the domain could be the Bunimovich
stadium~\cite{Bun79}. Another possible domain is indicated in Figure
\ref{fig:domain1}.  The ergodicity of billiards in such domains,
which produce hyperbolic flows, goes back to the pioneering work of
Sinai~\cite{Sin70}, although a number of individuals have
contributed to the theory.  A full accounting of this history can be
found in~\cite{CM06}.  Polygonal domains satisfying our assumptions
can also be constructed~\cite{Vorobets_1997}. Suitable domains in
$d=3$ dimensions can be constructed using a rectangular box with
shallow spherical caps adjoined~\cite{BunimovichRehacek1998}.  Note
that we make no assumptions regarding the hyperbolicity of the
billiard flow in the domain.

The Hamiltonian system we consider consists of the massive piston of
mass $M$ located at position $Q$, as well as $ n_1+ n_2 $ gas
particles, $ n_1$ in $\mathcal{D}_1$ and $n_2$ in $\mathcal{D}_2$.
Here $n_1$ and $n_2$ are fixed positive integers. For convenience,
the gas particles all have unit mass, though all that is important
is that each gas particle has a fixed mass. We denote the positions
of the gas particles in $\mathcal{D}_i$ by $q_{ i,j}$, $1\leq j\leq
n_i$.  The gas particles are ideal point particles that interact
with $\partial\mathcal{D} $ and the piston by hard core, elastic
collisions. Although it has no effect on the dynamics we consider,
for convenience we complete our description of the Hamiltonian
dynamics by specifying that the piston makes elastic collisions with
walls located at $Q=0,\: 1$ that are only visible to the piston.  We
denote velocities by $dQ/dt=V=\varepsilon W$ and $dq_{ i,j}/dt=v_{
i,j}$, and we set
\[
    E_{ i,j}=v_{ i,j}^2/2,\qquad E_i=\sum_{ j=1} ^
    {n_i} E_{ i,j}.
\]
Our system has $d(n_1+n_2)+1$ degrees of freedom, and so its phase
space is $(2d(n_1+n_2)+2)$-dimensional.

We let
\[
    h(z)=h=(Q,W,E_{1,1},E_{1,2},\cdots,E_{1,n_1},E_{2,1},E_{2,2},\cdots,E_{2,n_2}),
\]
so that $h$ is a function from our phase space to
$\mathbb{R}^{n_1+n_2+2}$.  We often abbreviate
$h=(Q,W,E_{1,j},E_{2,j})$, and we refer to $h$ as consisting of the
slow variables because these quantities are conserved when
$\varepsilon=0$.  We let $h_\varepsilon(t,z)=h_\varepsilon(t) $
denote the actual motions of these variables in time for a fixed
value of $\varepsilon$. Here $z$ represents the initial condition in
phase space, which we usually suppress in our notation.  One should
think of $h_\varepsilon(\cdot) $ as being a random variable that
takes initial conditions in phase space to paths (depending on the
parameter t) in $\mathbb{R}^{n_1+n_2+2}$.

\subsection {The averaged equation}

From the work of Neishtadt and Sinai~\cite{NS04}, one can derive
\begin {equation}\label{eq:d_davg}
    \frac{d}{d\tau}
    \begin {bmatrix}
    Q\\
    W\\
    E_{1,j}\\
    E_{2,j}\\
    \end {bmatrix}
    =\bar H(h):=
    \begin {bmatrix}
    W\\
    \displaystyle\frac{2E_1\ell}{d\abs{\mathcal{D}_1(Q)}}
    -\frac{2E_2\ell}{d\abs{\mathcal{D}_2(Q)}}\\
    \displaystyle-\frac{2WE_{1,j}\ell}{d\abs{\mathcal{D}_1(Q)}}\\
    \displaystyle+\frac{2WE_{2,j}\ell}{d\abs{\mathcal{D}_2(Q)}}\\
    \end {bmatrix}
\end {equation}
as the averaged equation (with respect to the slow time
$\tau=\varepsilon t$) for the slow variables. Later, in Section
\ref{sct:heuristic2}, we will give another heuristic derivation of
the averaged equation that is more suggestive of our proof.

Neishtadt and Sinai~\cite{Sin99, NS04} pointed out that the
solutions of Equation~\eqref{eq:heuristic_avg_eq} have $(Q, W) $
behaving as if they were the coordinates of a Hamiltonian system
describing a particle undergoing motion inside a potential well.  As
in Section~\ref{sct:heuristic}, the effective Hamiltonian is given
by
\begin{equation*}
    \frac {1} {2}W^2+
    \frac{E_1(0)\abs{\mathcal{D}_1(Q(0))}^{2/d}}
    {\abs{\mathcal{D}_1(Q)}^{2/d}}+
    \frac{E_2(0)\abs{\mathcal{D}_2(Q(0))}^{2/d}}
    {\abs{\mathcal{D}_2(Q)}^{2/d}}.
\end{equation*}
This can be seen as follows.  Since
\[
    \frac {\partial\abs{\mathcal{D}_1(Q)}} {\partial Q}=\ell
    =-\frac {\partial\abs{\mathcal{D}_2(Q)}} {\partial Q},
\]
$d\ln(E_{i,j})/d\tau=-(2/d)d\ln(\abs{\mathcal{D}_i(Q)})/d\tau$, and
so
\[
    E_{i,j}(\tau)=E_{i,j}(0)\left (\frac{\abs{\mathcal{D}_i(Q(0))}}
        {\abs{\mathcal{D}_i(Q(\tau))}}\right) ^ {2/d}.
\]
By summing over $j$, we find that
\[
    E_{i}(\tau)=E_{i}(0)\left (\frac{\abs{\mathcal{D}_i(Q(0))}}
        {\abs{\mathcal{D}_i(Q(\tau))}}\right) ^ {2/d}
\]
and so
\[
    \frac {d^2Q(\tau)} {d\tau^2}=
    \frac {2\ell} {d}
    \frac{E_1(0)\abs{\mathcal{D}_1(Q(0))}^{2/d}}{\abs{\mathcal{D}_1(Q(\tau))}^{1+2/d}}
    -\frac {2\ell} {d}
    \frac{E_2(0)
    \abs{\mathcal{D}_2(Q(0))}^{2/d}}{\abs{\mathcal{D}_2(Q(\tau))}^{1+2/d}}.
\]

Let $\bar{h} (\tau,z)=\bar{h} (\tau) $ be the solution of
\[
\frac {d\bar{h}}{d\tau} =\bar {H} (\bar {h}),\qquad \bar {h} (0)
=h_\varepsilon(0).
\]
Again, think of $\bar h(\cdot) $ as being a random variable.

\subsection{The main result}

The solutions of the averaged equation approximate the motions of
the slow variables, $h_\varepsilon(t) $, on a time scale
$\mathcal{O} (1/\varepsilon) $ as $\varepsilon\rightarrow 0$.
Precisely, fix a compact set $\mathcal{V}\subset \mathbb
R^{n_1+n_2+2}$ such that $h\in \mathcal{V} \Rightarrow
Q\subset\subset (0,1),W\subset\subset \mathbb R$, and $E_{i,j}
\subset\subset (0,\infty)$ for each $i$ and $j$.\footnote { We have
introduced this notation for convenience.  For example, $h\in
\mathcal{V} \Rightarrow Q\subset\subset (0,1) $ means that there
exists a compact set $A \subset (0,1) $ such that $h\in \mathcal{V}
\Rightarrow Q\in A $, and similarly for the other variables.}  We
will be mostly concerned with the dynamics when $h\in\mathcal{V} $.
Define
\[
\begin {split}
    Q_{min}&=\inf_{h\in\mathcal{V}}Q,\qquad
    Q_{max}=\sup_{h\in\mathcal{V}}Q,
    \\
    E_{min}&=\inf_{h\in\mathcal{V}}\frac{1}{2}W^2+E_1+E_2,\qquad
    E_{max}=\sup_{h\in\mathcal{V}}\frac{1}{2}W^2+E_1+E_2.
\end {split}
\]
For a fixed value of $\varepsilon >0$, we only consider the dynamics
on the invariant subset of phase space defined by
\[
\begin {split}
    \mathcal{M}_\varepsilon =
    \{(Q,V,q_{i,j},v_{i,j})\in\mathbb{R}^ {2d(n_1+n_2)+2}:
    Q\in [0,1],\;q_{i,j}\in\mathcal{D}_{i}(Q),
    &
    \\
    E_{min}\leq \frac{M}{2}V^2+E_1+E_2\leq E_{ max}\}
    &.
\end {split}
\]
Let $P_\varepsilon$ denote the probability measure obtained by
restricting the invariant Liouville measure to
$\mathcal{M}_\varepsilon $.  Define the stopping time
\[
    T_\varepsilon(z) =T_\varepsilon =\inf \{\tau\geq 0: \bar {h}
    (\tau)\notin \mathcal{V} \text { or } h_\varepsilon(\tau
    /\varepsilon) \notin \mathcal{V} \}.
\]

\begin {thm}\label{thm:dDpiston}

If $\mathcal{D} $ is a gas container in $d=2$ or $3$ dimensions
satisfying the assumptions in Subsection~\ref{sct:model} above, then
for each $T>0$,
\[
    \sup_{0\leq\tau\leq T\wedge
    T_\varepsilon}\abs{h_\varepsilon(\tau/\varepsilon)-\bar{h}(\tau)}
    \rightarrow 0 \text { in probability as } \varepsilon=M^
    {-1/2}\rightarrow 0,
\]
i.e.~for each fixed $\delta>0$,
\[
    P_\varepsilon\left(\sup_{0\leq\tau\leq T\wedge
    T_\varepsilon}\abs{h_\varepsilon(\tau/\varepsilon)-\bar{h}(\tau)}\geq\delta
    \right)
    \rightarrow 0\text { as } \varepsilon=M^
    {-1/2}\rightarrow 0.
\]
\end{thm}

\begin {rem}
It should be noted that the stopping time in the above result is not
unduly restrictive.  If the initial pressures of the two gasses are
not too mismatched, then the solution to the averaged equation is a
periodic orbit, with the effective potential well keeping the piston
away from the walls.  Thus, if the actual motions follow the
averaged solution closely for $0\leq\tau\leq T\wedge T_\varepsilon$,
and the averaged solution stays in $\mathcal{V} $, it follows that
$T_\varepsilon
>T$.
\end{rem}

\begin {rem}
The techniques of this work should immediately generalize to prove
the analogue of Theorem~\ref{thm:dDpiston} above in the nonphysical
dimensions $d>3$, although we do not pursue this here.
\end{rem}

\begin {rem}

As in Subsection~\ref{sct:apps_generalizations},
Theorem~\ref{thm:dDpiston} can be easily generalized to cover a
system of $N-1$ pistons that divide $N$ gas containers, so long as,
for almost every fixed location of the pistons,  the billiard flow
of a single gas particle on an energy surface in any of the $N$
subcontainers is ergodic (with respect to the invariant Liouville
measure). The effective Hamiltonian for the pistons has them moving
like an $(N-1)$-dimensional particle inside a potential well.

\end{rem}

\section{Preparatory material concerning a\\ two-dimensional gas
    container with only one gas particle on each side}
\label{sct:2dprep}

Our results and techniques of proof are essentially independent of
the dimension and the fixed number of gas particles on either side
of the piston. Thus, we focus on the case when $d=2$ and there is
only one gas particle on either side.  Later, in Section
\ref{sct:generalization}, we will indicate the simple modifications
that generalize our proof to the general situation.  For clarity, in
this section and next, we denote $ q_{ 1,1} $ by $ q_1$, $v_{2, 1} $
by $ v_2$, etc. We decompose the gas particle coordinates according
to whether they are perpendicular to or parallel to the piston's
face, for example $q_1= (q_1^\perp,q_1^\parallel)$. See Figure
\ref{fig:domain}.

\begin{figure}
    \begin {center}
    \setlength{\unitlength}{0.7 cm}
    \begin{picture}(15,6)
        \put(2.5,1){\line(1,0){10}}
        \put(2.5,5){\line(1,0){10}}
        \qbezier(2,3.5)(3.5,4)(2.5,5)
        \qbezier(2,3.5)(3,2.5)(2,1)
        \put(2,1){\line(1,0){1}}
        \qbezier(12.5,1)(14.5,1)(14.5,3)
        \qbezier(14.5,3)(14.5,5)(12.5,5)
        \put(3.2,4.5){$\mathcal{D}_1$}
        \put(12,4.5){$\mathcal{D}_2$}
        \put(13.3,5){\vector(0,1){0.5}}
        \put(13.23,5.7){$\parallel$}
        \put(13.3,5){\vector(1,0){0.5}}
        \put(13.95,4.9){$\perp$}
        \put(1.5,1){\line(0,1){4}}
        \put(1.5,1){\line(1,0){0.15}}
        \put(1.5,5){\line(1,0){0.15}}
        \put(1.5,3){\line(-1,0){0.15}}
        \put(1,2.9){$\ell$}
        \linethickness {0.1cm}
        \put(8.5,1){\line(0,1){4}}
        \thinlines
        \put(8.5,3){\vector(1,0){.4}}
        \put(9.0,2.9){$V=\varepsilon W$}
        \put(8.5,1){\line(0,-1){0.15}}
        \put(8.38,0.3){$Q$}
        \put(8.38,5.2){$M=\varepsilon ^ {-2}\gg 1$}
        \put(4.25,1){\line(0,-1){0.15}}
        \put(4.15,0.3){$0$}
        \put(12,1){\line(0,-1){0.15}}
        \put(11.95,0.3){$1$}
        \put(5,4){\circle*{.15}}
        \put(5,4){\vector(1,2){.4}}
        \put(5.05,3.65){$q_1$}
        \put(5.45,4.5){$v_1$}
        \put(11.5,2.5){\circle*{.15}}
        \put(11.5,2.5){\vector(-1,-3){.3}}
        \put(11.65,2.3){$q_2$}
        \put(11.35,1.5){$v_2$}
    \end{picture}
    \end {center}
    \caption{A choice of coordinates on phase space.}
    \label{fig:domain}
\end{figure}
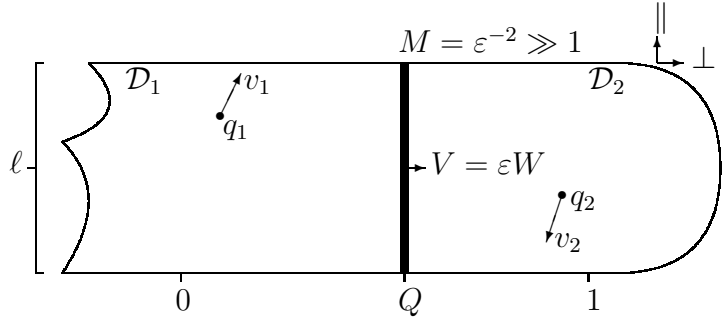

The Hamiltonian dynamics define a flow on our phase space. We denote
this flow by $z_\varepsilon(t,z) =z_\varepsilon(t)$, where
$z=z_\varepsilon(0,z) $. One should think of $z_\varepsilon(\cdot) $
as being a random variable that takes initial conditions in phase
space to paths in phase space.  Then $h_\varepsilon(t)
=h(z_\varepsilon(t)) $.  By the change of coordinates
$W=V/\varepsilon$, we may identify all of the
$\mathcal{M}_\varepsilon$ defined in Section~\ref{sct:main_result}
with the space
\[
\begin {split}
    \mathcal{M} =
    \{(Q,W,q_1,v_1,q_2,v_2)\in\mathbb{R}^ {10}:
    Q\in [0,1],\;q_1\in\mathcal{D}_1(Q),\;
    q_2\in\mathcal{D}_2(Q),\;
    &
    \\
    E_{min}\leq \frac{1}{2}W^2+E_1+E_2\leq E_{ max}\}
    &.
\end {split}
\]
and all of the $P_\varepsilon$ with the probability measure $P$ on
$\mathcal{M}$, which has the density
\[
    dP=\C \,dQdWdq_1^\perp dq_1^\parallel dv_1^\perp dv_1^\parallel
    dq_2^\perp dq_2^\parallel dv_2^\perp dv_2^\parallel.
\]
(Throughout this work we will use $\C$ to represent generic
constants that are independent of $\varepsilon$.) We will assume
that these identifications have been made, so that we may consider
$z_\varepsilon(\cdot) $ as a family of measure preserving flows on
the same space that all preserve the same probability measure.  We
denote the components of $z_\varepsilon(t) $ by $Q_\varepsilon(t) $,
$q_{1,\varepsilon}^\perp(t) $, etc.

The set $\{z\in\mathcal{M}:q_1=Q=q_2\}$ has co-dimension two, and so
$\bigcup_t z_\varepsilon(t)\{q_1=Q=q_2\}$ has co-dimension one,
which shows that only a measure zero set of initial conditions will
give rise to three particle collisions.  We ignore this and other
measures zero events, such as gas particles hitting singularities of
the billiard flow, in what follows.

Now we present some background material, as well as some lemmas that
will assist us in our proof of Theorem \ref{thm:dDpiston}.  We begin
by studying the billiard flow of a gas particle when the piston is
infinitely massive.  Next we examine collisions between the gas
particles and the piston when the piston has a large, but finite,
mass.  Then we present a heuristic derivation of the averaged
equation that is suggestive of our proof.  Finally we prove a lemma
that allows us to disregard the possibility that a gas particle will
move nearly parallel to the piston's face -- a situation that is
clearly bad for having the motions of the piston follow the
solutions of the averaged equation.

\subsection{Billiard flows and maps in two dimensions}\label{sct:billiard}

In this section, we study the billiard flows of the gas particles
when $M=\infty $ and the slow variables are held fixed at a specific
value $h\in\mathcal{V} $.  We will only study the motions of the
left gas particle, as similar definitions and results hold for the
motions of the right gas particle.  Thus we wish to study the
billiard flow of a point particle moving inside the domain
$\mathcal{D}_1$ at a constant speed $\sqrt{2E_1} $.  The results of
this section that are stated without proof can be found
in~\cite{CM06}.

Let $\mathcal{TD}_1$ denote the tangent bundle to $\mathcal{D}_1$.
The billiard flow takes place in the three-dimensional space
$\mathcal{M}_h^1=\mathcal{M}^1=\{(q_1,v_1)\in\mathcal{TD}_1:q_1\in\mathcal{
D}_1,\; \abs{v_1}=\sqrt{2E_1}\}/\sim$. Here the quotient means that
when $q_1\in\partial\mathcal{ D}_1$, we identify velocity vectors
pointing outside of $\mathcal{D}_1$ with those pointing inside
$\mathcal{D}_1$ by reflecting through the tangent line to
$\partial\mathcal{D}_1$ at $ q_1$, so that the angle of incidence
with the unit normal vector to $\partial\mathcal{ D}_1$ equals the
angle of reflection. Note that most of the quantities defined in
this subsection depend on the fixed value of $h$.  We will usually
suppress this dependence, although, when necessary, we will indicate
it by a subscript $h$.  We denote the resulting flow by
$y(t,y)=y(t)$, where $y(0,y)=y$. As the billiard flow comes from a
Hamiltonian system, it preserves Liouville measure restricted to the
energy surface.  We denote the resulting probability measure by $\mu
$. This measure has the density
$d\mu=dq_1dv_1/(2\pi\sqrt{2E_1}\abs{\mathcal{D}_1} ) $. Here $dq_1$
represents area on $\mathbb{R}^2$, and $ dv_1$ represents length on
$S^1_{\sqrt{2E_1}}=\set{v_1\in\mathbb{R}^2:\abs{v_1}=\sqrt{2E_1}}$.

There is a standard cross-section to the billiard flow, the
collision cross-section
$\Omega=\{(q_1,v_1)\in\mathcal{TD}_1:q_1\in\partial\mathcal{ D}_1,\;
\abs{v_1}=\sqrt{2E_1}\}/\sim$.  It is customary to parameterize
$\Omega$ by $\{x=(r,\varphi):r\in\partial\mathcal{D}_1,\:\varphi\in
[-\pi /2,+\pi /2]\}$, where $r$ is arc length and $\varphi$
represents the angle between the outgoing velocity vector and the
inward pointing normal vector to $\partial\mathcal{D}_1$. It follows
that $\Omega$ may be realized as the disjoint union of a finite
number of rectangles and cylinders.  The cylinders correspond to
fixed scatterers with smooth boundary placed inside the gas
container.
 If $F:\Omega\circlearrowleft$ is the collision map, i.e.~the
return map to the collision cross-section, then $F$ preserves the
projected probability measure $\nu $, which has the density
$d\nu=\cos\varphi\, d\varphi \, dr/(2\abs{\partial\mathcal{D}_1}) $.
Here $\abs{\partial\mathcal{D}_1}$ is the length of
$\partial\mathcal{D}_1$.

We suppose that the flow is ergodic, and so $F$ is an invertible,
ergodic measure preserving transformation. Because
$\partial\mathcal{D}_1$ is piecewise $\mathcal{C} ^3$, $F$ is
piecewise $\mathcal{C} ^2$, although it does have discontinuities
and unbounded derivatives near discontinuities corresponding to
grazing collisions.  Because of our assumptions on $\mathcal{D}_1$,
the free flight times and the curvature of $\partial\mathcal{D}_1$
are uniformly bounded.  It follows that if $x\notin \partial\Omega
\cup F^{-1} (\partial\Omega) $, then $F$ is differentiable at $x$,
and
\begin {equation}\label{eq:2d_derivative_bound}
    \norm {DF(x)}\leq\frac {\C} {\cos \varphi(Fx)},
\end{equation}
where $\varphi(Fx) $ is the value of the $\varphi$ coordinate at the
image of $x$.

Following the ideas in Section \ref{sct:inducing}, we induce $F$ on
the subspace $\hat\Omega$ of $\Omega$ corresponding to collisions
with the (immobile) piston.  We denote the induced map by $\hat F$
and the induced measure by $\hat \nu$. We parameterize $\hat\Omega$
by $\{(r,\varphi):0\leq r\leq \ell,\:\varphi\in [-\pi/2,+\pi/2]\}$.
As $\nu \hat\Omega=\ell/\abs{\partial\mathcal{D}_1} $, it follows
that $\hat\nu $ has the density $d\hat \nu=\cos\varphi\, d\varphi \,
dr/(2\ell) $.

For $x\in\Omega $, define $\zeta x $ to be the free flight time,
i.e.~the time it takes the billiard particle traveling at speed
$\sqrt{2E_1} $ to travel from $x$ to $Fx$. If $x\notin
\partial\Omega \cup F^{-1} (\partial\Omega) $,
\begin {equation}
\label{eq:2d_time_derivative}
    \norm {D\zeta (x)}\leq\frac {\C} {\cos \varphi(Fx)}.
\end {equation}
Santal\'{o}'s formula~\cite{San76,Chernov1997} tells us that
\begin {equation}
\label{eq:2d_Santalo}
    E_\nu \zeta=\frac {\pi
    \abs{\mathcal{D}_1}} {\abs{v_1}\abs{\partial\mathcal{D}_1}}.
\end {equation}
If $\hat\zeta:\hat\Omega\rightarrow\mathbb{R} $ is the free flight
time between collisions with the piston, then it follows from
Proposition \ref{prop:inducing} that
\begin {equation}
\label{eq:2d_flight}
    E_{\hat\nu} \hat\zeta=\frac {\pi
    \abs{\mathcal{D}_1}}{\abs{v_1}\ell}.
\end{equation}

The expected value of $ \abs{v_1^\perp }$ when the left gas particle
collides with the (immobile) piston is given by
\begin {equation}
\label{eq:2d_momentum}
    E_{\hat\nu} \abs{v_1^\perp }=E_{\hat\nu} \sqrt{2E_1}\cos\varphi=
    \frac{\sqrt{2E_1}}{2}\int_{-\pi/2}^{+\pi/2} \cos^2\varphi\,d\varphi=
    \sqrt{2E_1}\frac{\pi}{4}.
\end {equation}

We wish to compute $\lim_{t\rightarrow\infty} t^{-1} \int_0^t
\abs{2v_1^\perp (s)}\delta_{q_1^\perp (s)=Q}ds$, the time average of
the change in momentum of the left gas particle when it collides
with the piston.  If this limit exists and is equal for almost every
initial condition of the left gas particle, then it makes sense to
define the pressure inside $\mathcal{D}_1$ to be this quantity
divided by $\ell$.  Because the collisions are hard-core, we cannot
directly apply Birkhoff's Ergodic Theorem to compute this limit.
However, we can compute this limit by using the map $\hat F$.

\begin {lem}
\label{lem:ae_convergence}

If the billiard flow $y(t) $ is ergodic, then for $\mu-a.e.$ $y\in
\mathcal{M}^1$,
\[
    \lim_{t\rightarrow\infty} \frac{1}{t}
    \int_0^t \abs{v_1^\perp (s)}\delta_{q_1^\perp (s)
    =Q}ds=
    \frac{E_1\ell}{2\abs{\mathcal{D}_1(Q)}}.
\]

\end {lem}

\begin {proof}
Because the billiard flow may be viewed as a suspension flow over
the collision cross-section with $\zeta$ as the height function, it
suffices to show that the convergence takes place for $\hat\nu-a.e.$
$x\in\hat\Omega$. For an initial condition $x\in\hat\Omega$, define
$\hat{N}_t(x)=\hat{N}_t=\#\set{s\in (0,t]:y(s,x) \in\hat\Omega}$. By
the Poincar\'e  Recurrence Theorem, $\hat{N}_t\rightarrow\infty$ as
$t\rightarrow\infty$, $\hat\nu-a.e.$

But
\[
\begin {split}
    \frac{\hat{N}_t}{\sum_{n=0}^{\hat{N}_t}\hat\zeta(\hat F^n x)}
    \frac{1}{\hat{N}_t}
    \sum_{n=1}^{\hat{N}_t}\abs{v_1^\perp }(\hat F^n x)
    &\leq
    \frac{1}{t}\int_0^t \abs{v_1^\perp (s)}\delta_{q_1^\perp (s)=Q}ds
    \\
    &\leq
    \frac{\hat{N}_t}{\sum_{n=0}^{\hat{N}_t-1}\hat\zeta(\hat F^n x)}
    \frac{1}{\hat{N}_t}
    \sum_{n=0}^{\hat{N}_t}\abs{v_1^\perp }(\hat F^n x),
\end {split}
\]
and so the result follows from Birkhoff's Ergodic Theorem and
Equations \eqref{eq:2d_flight} and \eqref{eq:2d_momentum}.
\end {proof}

\begin {cor}
\label{cor:ae_convergence}

If the billiard flow $y(t) $ is ergodic, then for each $\delta>0$,
\[
    \mu
    \set{y\in \mathcal{M}^1:\abs{\frac{1}{t}
    \int_0^t \abs{v_1^\perp (s)}\delta_{q_1^\perp (s)
    =Q}ds-
    \frac{E_1\ell}{2\abs{\mathcal{D}_1(Q)}}}\geq \delta}
    \rightarrow 0\text{ as }t\rightarrow \infty.
\]

\end {cor}

\subsection{Analysis of collisions}\label{sct:collisions}

In this section, we return to studying our piston system when
$\varepsilon>0$.  We will examine what happens when a particle
collides with the piston.  For convenience, we will only examine in
detail collisions between the piston and the left gas particle.
Collisions with the right gas particle can be handled similarly.

When the left gas particle collides with the piston, $v_1^\perp $
and $V$ instantaneously change according to the laws of elastic
collisions:
\begin {equation*}
    \begin{bmatrix}
    v_1^{\perp +}\\ V^+
    \end{bmatrix}
    =
    \frac{1}{1+M}
    \begin{bmatrix}
    1-M& 2M\\
    2& M-1\\
    \end{bmatrix}
    \begin{bmatrix}
    v_1^{\perp -}\\ V^-
    \end{bmatrix}.
\end {equation*}
In our coordinates, this becomes
\begin {equation}\label{eq:collision_change}
    \begin{bmatrix}
    v_1^{\perp +}\\ W^+
    \end{bmatrix}
    =
    \frac{1}{1+\varepsilon^2 }
    \begin{bmatrix}
    \varepsilon^2 -1 & 2\varepsilon\\
    2\varepsilon & 1-\varepsilon^2 \\
    \end{bmatrix}
    \begin{bmatrix}
    v_1^{\perp -}\\ W^-
    \end{bmatrix}.
\end {equation}
Recalling that $ v_1, W=\mathcal{O} (1) $, we find that to first
order in $\varepsilon$,
\begin{equation}
\label{eq:v_1Wchange}
    v_1^{\perp +}=-v_1^{\perp -}+\mathcal{O}(\varepsilon),\qquad
    W^ +=W^ -+\mathcal{O}(\varepsilon).
\end{equation}
Observe that a collision can only take place if $v_1^{\perp
-}>\varepsilon W^ - $.  In particular, $v_1^{\perp -}> -
\varepsilon\sqrt{2E_{max}}$.  Thus, either $v_1^{\perp -}> 0$ or
$v_1^{\perp -}= \mathcal{O} (\varepsilon) $.  By expanding
Equation~\eqref{eq:collision_change} to second order in
$\varepsilon$, it follows that
\begin{equation}
\label{eq:E_1Wchange}
\begin {split}
    E_1^+ -E_1^- &=-2\varepsilon W \abs{v_1^{\perp }}
    +\mathcal{O}(\varepsilon^2),\\
    W^+ -W^- &=+2\varepsilon \abs{v_1^{\perp }}
    +\mathcal{O}(\varepsilon^2).
\end {split}
\end{equation}
Note that it is immaterial whether we use the pre-collision or
post-collision values of $W$ and $\abs{v_1^{\perp }}$ on the right
hand side of Equation~\eqref{eq:E_1Wchange}, because any ambiguity
can be absorbed into the $\mathcal{O} (\varepsilon^2) $ term.

It is convenient for us to define a ``clean collision'' between the
piston and the left gas particle:
\begin {defn}
    The left gas particle experiences a \emph{clean collision} with the
    piston if and only if $v_1^{\perp -}>0$ and $v_1^{\perp +}<-\varepsilon
    \sqrt{2E_{max}}$.
\end{defn}
\noindent In particular, after a clean collision, the left gas
particle will escape from the piston, i.e.~the left gas particle
will have to move into the region $q_1^{\perp }\leq 0 $ before it
can experience another collision with the piston.  It follows that
there exists a constant $C_1>0$, which depends on the set
$\mathcal{V} $, such that for all $\varepsilon$ sufficiently small,
so long as $ Q\geq Q_{min} $ and $\abs{v_1^{\perp }}
>\varepsilon C_1$ when $q_1^{\perp }\in [Q_{ min},Q]$, then the left gas particle
will experience only clean collisions with the piston, and the time
between these collisions will be greater than $2Q_{min}/(\sqrt
{2E_{max}})$.  (Note that when we write expressions such as
$q_1^{\perp }\in [Q_{ min},Q]$, we implicitly mean that $q_1$ is
positioned inside the ``tube'' discussed at the beginning of
Section~\ref{sct:main_result}.) One can verify that $C_1=5\sqrt
{2E_{max}}$ would work.

Similarly, we can define clean collisions between the right gas
particle and the piston.  We assume that $C_1$ was chosen
sufficiently large such that for all $\varepsilon$ sufficiently
small, so long as $ Q\leq Q_{max} $ and $\abs{v_2^{\perp }}
>\varepsilon C_1$ when $q_2^{\perp }\in [Q,Q_{max}] $, then the right gas particle
will experience only clean collisions with the piston.

Now we define three more stopping times, which are functions of the
initial conditions in phase space.
\[
\begin {split}
    T_\varepsilon' =&\inf \{\tau\geq 0: Q_{min}\leq q_{1,\varepsilon}^{\perp
    }(\tau/\varepsilon)\leq Q_\varepsilon(\tau/\varepsilon)\leq Q_{max}
    \text { and}\abs{v_{1,\varepsilon}^{\perp }(\tau/\varepsilon)}\leq C_1\varepsilon \}
    ,
    \\
    T_\varepsilon'' =&
    \inf \{\tau\geq 0:Q_{min}\leq Q_\varepsilon(\tau/\varepsilon)\leq q_{2,\varepsilon}^{\perp
    }(\tau/\varepsilon)\leq Q_{max}
    \text { and}\abs{v_{2,\varepsilon}^{\perp }(\tau/\varepsilon)}\leq C_1\varepsilon
    \},
    \\
    \tilde{T}_\varepsilon =&
    T\wedge T_\varepsilon\wedge T_\varepsilon'\wedge T_\varepsilon''
\end {split}
\]

Define $H(z) $ by
\[
    H(z) =
    \begin{bmatrix}
    W\\
    +2\abs{v_1^{\perp }} \delta_{q_1^{\perp }=Q}
    -2\abs{v_2^{\perp }} \delta_{q_2^{\perp }=Q}\\
    -2W\abs{v_1^{\perp }} \delta_{q_1^{\perp }=Q}\\
    +2W\abs{v_2^{\perp }} \delta_{q_2^{\perp }=Q}\\
    \end{bmatrix}.
\]
Here we make use of Dirac delta functions.  All integrals involving
these delta functions may be replaced by sums.

The following lemma is an immediate consequence of Equation
\eqref{eq:E_1Wchange} and the above discussion:

\begin{lem}
\label{lem:h_int}

If $0\leq t_1\leq t_2\leq \tilde{T}_\varepsilon/\varepsilon $, the
piston experiences $\mathcal{O} ((t_2-t_1)\vee 1) $ collisions with
gas particles in the time interval $[t_1, t_2]$, all of which are
clean collisions. Furthermore,
\begin {equation*}
    h_\varepsilon(t_2)-h_\varepsilon(t_1)=
    \mathcal{O}(\varepsilon)+\varepsilon\int_{t_1}^{t_2}
    H(z_\varepsilon(s))ds.
\end {equation*}
Here any ambiguities arising from collisions occurring at the limits
of integration can be absorbed into the $\mathcal{O} (\varepsilon) $
term.

\end {lem}

\subsection{Another heuristic derivation of the averaged
equation}\label{sct:heuristic2}

The following heuristic derivation of Equation \eqref{eq:d_davg}
when $ d=2$ was suggested in~\cite{Dol05}. Let $\Delta t $ be a
length of time long enough such that the piston experiences many
collisions with the gas particles, but short enough such that the
slow variables change very little, in this time interval.  From each
collision with the left gas particle, Equation~\eqref{eq:E_1Wchange}
states that $W$ changes by an amount $+2\varepsilon
\abs{v_1^{\perp}} +\mathcal{O}(\varepsilon^2)$, and from
Equation~\eqref{eq:2d_momentum} the average change in $W$ at these
collisions should be approximately $\varepsilon\pi
\sqrt{2E_1}/2+\mathcal{O}(\varepsilon^2)$. From
Equation~\eqref{eq:2d_flight} the frequency of these collisions is
approximately $\sqrt{2E_1}\,\ell /(\pi
    \abs{\mathcal{D}_1})$. Arguing
similarly for collisions with the other particle, we guess that
\[
    \frac {\Delta W} {\Delta t} =
    \varepsilon\frac{E_1\ell}{\abs{\mathcal{D}_1(Q)}}
    -\varepsilon\frac{E_2\ell}{\abs{\mathcal{D}_2(Q)}}
    +\mathcal{O}(\varepsilon^2).
\]
With $\tau=\varepsilon t$ as the slow time, a reasonable guess for
the averaged equation for $W$ is
\[
    \frac {dW}{d\tau}=\frac{E_1\ell}{\abs{\mathcal{D}_1(Q)}}
    -\frac{E_2\ell}{\abs{\mathcal{D}_2(Q)}}.
\]
Similar arguments for the other slow variables lead to the averaged
equation \eqref{eq:d_davg}, and this explains why we used $P_i=
E_i/\abs{\mathcal{D}_i}$ for the pressure of a $2$-dimensional gas
in Section~\ref{sct:heuristic}.

There is a similar heuristic derivation of the averaged equation in
$ d>2$ dimensions.  Compare the analogues of
Equations~\eqref{eq:2d_flight} and \eqref{eq:2d_momentum} in
Subsection~\ref{sct:higher_d}.

\subsection{\textit{A priori} estimate on the size
        of a set of bad initial conditions}

In this section, we give an \textit{a priori} estimate on the size
of a set of initial conditions that should not give rise to orbits
for which $\sup_{0\leq\tau\leq T\wedge
T_\varepsilon}\abs{h_\varepsilon(\tau/\varepsilon)-\bar{h}(\tau)}$
is small.  In particular, when proving Theorem \ref{thm:dDpiston},
it is convenient to focus on orbits that only contain clean
collisions with the piston.  Thus, we show that $P\{\tilde
{T}_\varepsilon<T\wedge T_\varepsilon \} $ vanishes as
$\varepsilon\rightarrow 0$. At first, this result may seem
surprising, since $P\{T_\varepsilon'\wedge T_\varepsilon''=0\}
=\mathcal{O}(\varepsilon)$, and one would expect $\cup_{t=0} ^
{T/\varepsilon} z_\varepsilon(-t)\{T_\varepsilon'\wedge
T_\varepsilon''=0\}$ to have a size of order $1$. However, the rate
at which orbits escape from $\{T_\varepsilon'\wedge
T_\varepsilon''=0\}$ is very small, and so we can prove the
following:

\begin {lem}
\label{lem:no_vertical}
    \[
        P\{\tilde {T}_\varepsilon<T\wedge T_\varepsilon \} =\mathcal{O}
        (\varepsilon).
    \]
\end {lem}

In some sense, this lemma states that the probability of having a
gas particle move nearly parallel to the piston's face within the
time interval $[0,T/\varepsilon ] $, when one would expect the other
gas particle to force the piston to move on a macroscopic scale,
vanishes as $\varepsilon\rightarrow 0$.  Thus, one can hope to
control the occurrence of the ``nondiffusive fluctuations'' of the
piston described in~\cite{CD06} on a time scale $\mathcal{O}
(\varepsilon^ {-1}) $.

\begin {proof}

As the left and the right gas particles can be handled similarly, it
suffices to show that $P\{T_\varepsilon'<T \} =\mathcal{O}
(\varepsilon)$.  Define
\[
    \mathfrak{B}_\varepsilon=\{z\in\mathcal{M}:
    Q_{min}\leq q_1^{\perp
    }\leq Q\leq Q_{max}
    \text { and}\abs{v_1^{\perp }}\leq C_1\varepsilon
    \}.
\]
Then $\{T_\varepsilon'<T \}\subset \cup_{t=0} ^ {T/\varepsilon}
z_\varepsilon(-t)\mathfrak{B}_\varepsilon$, and if $\gamma=
Q_{min}/\sqrt{8 E_{max}}$,
\[
\begin {split}
    P\left (\bigcup_{t=0} ^{T/\varepsilon}
    z_\varepsilon(-t)\mathfrak{B}_\varepsilon\right)
    &
    =
    P\left (\bigcup_{t=0} ^{T/\varepsilon}
    z_\varepsilon(t)\mathfrak{B}_\varepsilon\right)
    =
    P\left(\mathfrak{B}_\varepsilon\cup\bigcup_{t=0} ^{T/\varepsilon}
    ((z_\varepsilon(t)\mathfrak{B}_\varepsilon) \backslash \mathfrak{B}_\varepsilon)
    \right)
    \\
    &
    \leq
    P\mathfrak{B}_\varepsilon+P\left( \bigcup_{k=0}^{T/(\varepsilon\gamma
    )} z_\varepsilon(k\gamma )
    \Bigl[ \bigcup_{t=0}^\gamma
    (z_\varepsilon(t)\mathfrak{B}_\varepsilon)\backslash
    \mathfrak{B}_\varepsilon \Bigr]
    \right)
    \\
    &
    \leq
    P\mathfrak{B}_\varepsilon+
    \left(\frac{T}{\varepsilon\gamma }+1\right)
    P\left(
    \bigcup_{t=0}^\gamma
    (z_\varepsilon(t)\mathfrak{B}_\varepsilon)\backslash
    \mathfrak{B}_\varepsilon
    \right).
\end {split}
\]
Now $P\mathfrak{B}_\varepsilon=\mathcal{O} (\varepsilon) $, so if we
can show that
$P\left(\bigcup_{t=0}^\gamma(z_\varepsilon(t)\mathfrak{B}_\varepsilon)\backslash
\mathfrak{B}_\varepsilon\right)=\mathcal{O} (\varepsilon^2)$, then
it will follow that $P\{T_\varepsilon'<T \} =\mathcal{O}
(\varepsilon)$.

If
$z\in\bigcup_{t=0}^\gamma(z_\varepsilon(t)\mathfrak{B}_\varepsilon)\backslash
\mathfrak{B}_\varepsilon$, it is still true that
$\abs{v_1^\perp}=\mathcal{O}(\varepsilon)$.  This is because
$\abs{v_1^\perp}$ changes by at most $\mathcal{O} (\varepsilon) $ at
the collisions, and if a collision forces
$\abs{v_1^\perp}>C_1\varepsilon$, then the gas particle must escape
to the region $q_1^\perp\leq 0$ before $ v_1^\perp $ can change
again, and this will take time greater than $\gamma $. Furthermore,
if
$z\in\bigcup_{t=0}^\gamma(z_\varepsilon(t)\mathfrak{B}_\varepsilon)\backslash
\mathfrak{B}_\varepsilon$, then at least one of the following four
possibilities must hold:
\begin {itemize}
\item
     $\abs{q_1^\perp-Q_{min}}\leq\mathcal{O}(\varepsilon)$,
\item
     $\abs{Q-Q_{min}}\leq\mathcal{O}(\varepsilon)$,
\item
     $\abs{Q-Q_{max}}\leq\mathcal{O}(\varepsilon)$,
\item
     $\abs{Q-q_1^\perp}\leq\mathcal{O}(\varepsilon)$.
\end {itemize}
It follows that
$P\left(\bigcup_{t=0}^\gamma(z_\varepsilon(t)\mathfrak{B}_\varepsilon)\backslash
\mathfrak{B}_\varepsilon\right)=\mathcal{O} (\varepsilon^2)$.  For
example,
\[
\begin {split}
    \int_{\mathcal{M}}
    &
    1_{\{\abs{v_1^\perp}\leq
    \mathcal{O}(\varepsilon),\:
    \abs{q_1^\perp-Q_{min}}\leq\mathcal{O}(\varepsilon)\}}dP
    \\
    &
    =
    \C
    \int_{\set {E_{min}\leq W^2/2+v_1^2/2+v_2^2/2\leq E_{max}}}
    1_{\{\abs{v_1^\perp}\leq
    \mathcal{O}(\varepsilon)\}}
    dW dv_1^{\perp}dv_1^{\parallel} dv_2^{\perp}dv_2^{\parallel}
    \\
    &
    \qquad\times
    \int_{\set {Q\in [0,1],\, q_1\in\mathcal{D}_1,\, q_2\in\mathcal{D}_2}}
    1_{\{\abs{q_1^\perp-Q_{min}}\leq\mathcal{O}(\varepsilon)\}}
    dQ dq_1^{\perp}dq_1^{\parallel} dq_2^{\perp}dq_2^{\parallel}
    \\
    &
    =\mathcal{O}(\varepsilon^2).
\end {split}
\]

\end {proof}

\section{Proof of the main result for two-dimensional gas
    containers with only one gas particle on each side}
\label{sct:2dproof}

As in Section~\ref{sct:2dprep}, we continue with the case when $d=2$
and there is only one gas particle on either side of the piston.

\subsection{Main steps in the proof of convergence in probability}\label{sct:main_steps}

By Lemma \ref{lem:no_vertical}, it suffices to show that $
\sup_{0\leq\tau\leq \tilde{T}_\varepsilon}
\abs{h_\varepsilon(\tau/\varepsilon)-\bar{h}(\tau)}\rightarrow 0$ in
probability as $\varepsilon=M^{-1/2}\rightarrow 0$. Several of the
ideas in the steps below were inspired by a recent proof of Anosov's
averaging theorem for smooth systems that is due to
Dolgopyat~\cite{Dol05}.

\paragraph*{Step 1:  Reduction using Gronwall's Inequality.}

Observe that $\bar {h}(\tau) $ satisfies the integral equation
\[
\bar {h}(\tau) -\bar h(0) = \int_0^{\tau}\bar H(\bar
h(\sigma))d\sigma,
\]
while from Lemma \ref{lem:h_int},
\[
\begin {split}
    h_\varepsilon(\tau/\varepsilon)-h_\varepsilon(0)
    &
    =\mathcal{O}(\varepsilon) +\varepsilon\int_0^{\tau/\varepsilon}
    H(z_\varepsilon(s))ds\\
    &=\mathcal{O}(\varepsilon) +
    \varepsilon\int_0^{\tau/\varepsilon}
    H(z_\varepsilon(s))-
    \bar H(h_\varepsilon(s))ds+
    \int_0^{\tau}\bar H( h_\varepsilon(\sigma/\varepsilon))d\sigma
\end {split}
\]
for $0\leq\tau\leq \tilde{T}_\varepsilon$.  Define
\[
    e_\varepsilon(\tau) =\varepsilon\int_0^{\tau/\varepsilon}
    H(z_\varepsilon(s))- \bar H(h_\varepsilon(s))ds.
\]
It follows from Gronwall's Inequality that
\begin {equation}\label {eq:2d_Gronwall}
    \sup_{0\leq \tau\leq \tilde{T}_\varepsilon}
    \abs{h_\varepsilon(\tau/\varepsilon)-\bar h(\tau)}\leq
    \left(\mathcal{O}(\varepsilon)+
    \sup_{0\leq \tau\leq \tilde{T}_\varepsilon}
    \abs{e_\varepsilon(\tau)}\right)e^{ \Lip{\bar
    H\arrowvert_\mathcal{V}}T}.
\end {equation}
\noindent Gronwall's Inequality is usually stated for continuous
paths, but the standard proof (found in \cite{SV85}) still works for
paths that are merely integrable, and
$\abs{h_\varepsilon(\tau/\varepsilon)-\bar h(\tau)}$ is piecewise
smooth.

\paragraph*{Step 2:  Introduction of a time scale for ergodization.}

Let $L(\varepsilon) $ be a real valued function such that
$L(\varepsilon)\rightarrow\infty$, but $L(\varepsilon)\ll \log
\varepsilon^ {-1} $, as $\varepsilon\rightarrow 0$. In
Section~\ref{sct:Gronwall} we will place precise restrictions on the
growth rate of $L(\varepsilon) $.  Think of $L(\varepsilon) $ as
being a time scale that grows as $\varepsilon\rightarrow 0$ so that
\emph{ergodization}, i.e.~the convergence along an orbit of a
function's time average to a space average, can take place. However,
$L(\varepsilon) $ doesn't grow too fast, so that on this time scale
$z_\varepsilon(t) $ essentially stays on the submanifold
$\set{h=h_\varepsilon(0)}$, where we have our ergodicity assumption.
Set $t_{k,\varepsilon} =kL(\varepsilon) $, so that
\begin {equation}
\label{eq:2d_einfnorm}
    \sup_{0\leq \tau\leq \tilde{T}_\varepsilon}\abs{e_\varepsilon(\tau)}
    \leq \mathcal{O}(\varepsilon L(\varepsilon))+
    \varepsilon\sum_{k=0}^{\frac{\tilde{T}_\varepsilon}{\varepsilon
    L(\varepsilon)}-1}\abs{\int_{t_{k,\varepsilon}}^{t_{k+1,\varepsilon}}H(z_\varepsilon(s))-\bar
    H(h_\varepsilon(s))ds}.
\end {equation}

\paragraph*{Step 3:  A splitting according to particles.}

Now $H(z) -\bar H(h(z)) $ divides into two pieces, each of which
depends on only one gas particle when the piston is held fixed:
\[
    H(z) -\bar H(h(z))=
    \begin{bmatrix}
    0\\
    2\abs{v_1^{\perp }} \delta_{q_1^{\perp }=Q}
    -\frac{E_1\ell}{\abs{\mathcal{D}_1(Q)}}\\
    -2W\abs{v_1^{\perp }} \delta_{q_1^{\perp }=Q}
    +\frac{WE_1\ell}{\abs{\mathcal{D}_1(Q)}}\\
    0\\
    \end{bmatrix}
    +
    \begin{bmatrix}
    0\\
    \frac{E_2\ell}{\abs{\mathcal{D}_2(Q)}}
    -2\abs{v_2^{\perp }} \delta_{q_2^{\perp }=Q}\\
    0\\
    -\frac{WE_2\ell}{\abs{\mathcal{D}_2(Q)}}+
    2W\abs{v_2^{\perp }} \delta_{q_2^{\perp }=Q}\\
    \end{bmatrix}.
\]
We will only deal with the piece depending on the left gas particle,
as the right particle can be handled similarly.  Define
\begin{equation}
\label{eq:G_definition}
    G(z)=\abs{v_1^{\perp }} \delta_{q_1^{\perp }=Q},
    \qquad
    \bar G(h)=
    \frac{E_1\ell}{2\abs{\mathcal{D}_1(Q)}}.
\end {equation}
Returning to Equation \eqref{eq:2d_einfnorm}, we see that in order
to prove Theorem \ref{thm:dDpiston}, it suffices to show that both
\[
\begin {split}
    &\varepsilon\sum_{k=0}^{\frac{\tilde{T}_\varepsilon}{\varepsilon
    L(\varepsilon)}-1}\abs{\int_{t_{k,\varepsilon}}^{t_{k+1,\varepsilon}}G(z_\varepsilon(s))-\bar
    G(h_\varepsilon(s))ds} \qquad\text { and}
    \\
    &\varepsilon\sum_{k=0}^{\frac{\tilde{T}_\varepsilon}{\varepsilon
    L(\varepsilon)}-1}\abs{\int_{t_{k,\varepsilon}}^{t_{k+1,\varepsilon}}W_\varepsilon(s)
    \bigl(G(z_\varepsilon(s))-\bar
    G(h_\varepsilon(s))\bigr)ds}
\end {split}
\]
converge to $0$ in probability as $\varepsilon\rightarrow 0$.

\paragraph*{Step 4:  A splitting for using the triangle inequality.}

Now we let $z_{k,\varepsilon} (s) $ be the orbit of the
$\varepsilon=0$ Hamiltonian vector field satisfying
$z_{k,\varepsilon}(t_{k,\varepsilon})=z_{\varepsilon}(t_{k,\varepsilon})$.
Set $h_{k,\varepsilon} (t) =h (z_{k,\varepsilon} (t)) $. Observe
that $h_{k,\varepsilon} (t) $ is independent of $t $.

We emphasize that so long as $0\leq
t\leq\tilde{T}_\varepsilon/\varepsilon$, the times between
collisions of a specific gas particle and piston are uniformly
bounded greater than $0$, as explained before Lemma \ref{lem:h_int}.
It follows that, so long as
$t_{k+1,\varepsilon}\leq\tilde{T}_\varepsilon/\varepsilon$,
\begin{equation}
\label{eq:h_div}
    \sup_{t_{k,\varepsilon}\leq t\leq t_{k+1,\varepsilon}}
    \abs{h_{k,\varepsilon} (t) -h_\varepsilon (t)}
    =\mathcal{O}(\varepsilon L(\varepsilon)).
\end{equation}
This is because the slow variables change by at most $\mathcal{O}
(\varepsilon) $ at collisions, and
$dQ_\varepsilon/dt=\mathcal{O}(\varepsilon)$.

Also,
\[
\begin {split}
    \int_{t_{k,\varepsilon}}^{t_{k+1,\varepsilon}}
    &
    W_\varepsilon(s) \bigl(G(z_\varepsilon(s))-\bar
    G(h_\varepsilon(s))\bigr)ds
    \\
    &= \mathcal{O} (\varepsilon
    L(\varepsilon)^2) +W_{k,\varepsilon}(t_{k,\varepsilon})\int_{t_{k,\varepsilon}}^{t_{k+1,\varepsilon}}
    G(z_\varepsilon(s))-\bar G(h_\varepsilon(s))ds,
\end {split}
\]
and so
\[
\begin {split}
    \varepsilon\sum_{k=0}^{\frac{\tilde{T}_\varepsilon}{\varepsilon
    L(\varepsilon)}-1}
    &
    \abs{\int_{t_{k,\varepsilon}}^{t_{k+1,\varepsilon}}W_\varepsilon(s)
    \bigl(G(z_\varepsilon(s))-\bar
    G(h_\varepsilon(s))\bigr)ds}
    \\
    &\leq
    \mathcal{O}(\varepsilon L(\varepsilon))+
    \varepsilon\,\C\sum_{k=0}^{\frac{\tilde{T}_\varepsilon}{\varepsilon
    L(\varepsilon)}-1}\abs{\int_{t_{k,\varepsilon}}^{t_{k+1,\varepsilon}}
    G(z_\varepsilon(s))-\bar
    G(h_\varepsilon(s))ds}.
\end {split}
\]
Thus, in order to prove Theorem \ref{thm:dDpiston}, it suffices to
show that
\[
\begin {split}
    \varepsilon\sum_{k=0}^{\frac{\tilde{T}_\varepsilon}{\varepsilon
    L(\varepsilon)}-1}\abs{\int_{t_{k,\varepsilon}}^{t_{k+1,\varepsilon}}G(z_\varepsilon(s))-\bar
    G(h_\varepsilon(s))ds}
    \leq
    \varepsilon\sum_{k=0}^{\frac{\tilde{T}_\varepsilon}{\varepsilon
    L(\varepsilon)}-1}
    \abs{I_{k,\varepsilon}}+\abs{II_{k,\varepsilon}}+\abs{III_{k,\varepsilon}}
\end {split}
\]
converges to $0$ in probability as $\varepsilon\rightarrow 0$, where
\[
\begin {split}
    I_{k,\varepsilon}
    &=\int_{t_{k,\varepsilon}}^{t_{k+1,\varepsilon}}G(z_\varepsilon(s)) -
    G(z_{k,\varepsilon}(s))ds,
    \\
    II_{k,\varepsilon}
    &=\int_{t_{k,\varepsilon}}^{t_{k+1,\varepsilon}} G(z_{k,\varepsilon}(s))-
    \bar G(h_{k,\varepsilon}(s))ds,
    \\
    III_{k,\varepsilon}
    &=\int_{t_{k,\varepsilon}}^{t_{k+1,\varepsilon}} \bar G(h_{k,\varepsilon}(s))-
    \bar G(h_{\varepsilon}(s))ds.
\end {split}
\]

The term $II_{k,\varepsilon}$ represents an ``ergodicity term'' that
can be controlled by our assumptions on the ergodicity of the flow
$z_0(t) $, while the terms $I_{k,\varepsilon}$ and
$III_{k,\varepsilon}$ represent ``continuity terms'' that can be
controlled by controlling the drift of $z_{\varepsilon} (t) $ from
$z_{k,\varepsilon} (t) $ for $t_{k,\varepsilon}\leq t\leq
t_{k+1,\varepsilon}$.

\paragraph*{Step 5:  Control of drift from the $\varepsilon=0$ orbits.}

Now $\bar G$ is uniformly Lipschitz on the compact set $\mathcal{V}
$, and so it follows from Equation \eqref{eq:h_div} that
$III_{k,\varepsilon}=\mathcal{O}(\varepsilon L(\varepsilon)^2)$.
Thus,
$\varepsilon\sum_{k=0}^{\frac{\tilde{T}_\varepsilon}{\varepsilon
L(\varepsilon)}-1}\abs{III_{k,\varepsilon}} =\mathcal{O}(\varepsilon
L(\varepsilon))\rightarrow 0$ as $\varepsilon\rightarrow 0$.

Next, we show that for fixed $\delta > 0$,
$P\left(\varepsilon\sum_{k=0}^{\frac{\tilde{T}_\varepsilon}{\varepsilon
    L(\varepsilon)}-1}\abs{I_{k,\varepsilon}}\geq
\delta\right)\rightarrow 0$ as $\varepsilon\rightarrow 0$.

For initial conditions $z\in \mathcal{M}$ and for integers $k\in
[0,T/(\varepsilon L(\varepsilon))-1]$ define
\[
\begin {split}
    \mathcal{A}_{k,\varepsilon} & =\set{z:\frac{1}{L(\varepsilon)}
    \abs{I_{k,\varepsilon}}
    >\frac{\delta}{2T} \text { and } k\leq\frac{\tilde{T}_\varepsilon}{\varepsilon
    L(\varepsilon)}-1 }
    ,\\
    \mathcal{A}_{z,\varepsilon} & =\set{k:z\in
    \mathcal{A}_{k,\varepsilon}}.
\end {split}
\]
Think of these sets as describing ``poor continuity'' between
solutions of the $\varepsilon=0$ and the $\varepsilon>0 $
Hamiltonian vector fields. For example, roughly speaking,
$z\in\mathcal{A}_{k,\varepsilon}$ if the orbit $z_{\varepsilon}(t)$
starting at $z$ does not closely follow $z_{k,\varepsilon}(t)$ for
$t_{k,\varepsilon}\leq t\leq t_{k+1,\varepsilon} $.

One can easily check that $\abs{I_{k,\varepsilon}}\leq \mathcal{O}
(L(\varepsilon))$ for $k\leq\ \tilde{T}_\varepsilon/(\varepsilon
L(\varepsilon))-1$, and so it follows that
\[
    \varepsilon
    \sum_{k=0}^{\frac{\tilde{T}_\varepsilon}{\varepsilon
    L(\varepsilon)}-1}\abs{I_{k,\varepsilon}}\leq
    \frac{\delta}{2}+\mathcal{O}(\varepsilon L(\varepsilon)
    \# (\mathcal{A}_{z,\varepsilon})).
\]
Therefore it suffices to show that $P\left(\#
(\mathcal{A}_{z,\varepsilon})\geq\delta(\C\, \varepsilon
L(\varepsilon)) ^ {-1}\right)\rightarrow 0$ as
$\varepsilon\rightarrow 0$. By Chebyshev's Inequality, we need only
show that
\[
    E_P (\varepsilon L(\varepsilon)\# (\mathcal{A}_{z,\varepsilon})) =
    \varepsilon L(\varepsilon)\sum_{k=0}^{\frac{T}{\varepsilon
    L(\varepsilon)}-1}P(\mathcal{A}_{k,\varepsilon})
\]
tends to $0$ with $\varepsilon$.

Observe that $z_\varepsilon
(t_{k,\varepsilon})\mathcal{A}_{k,\varepsilon}\subset\mathcal{A}_{0,\varepsilon}
$. In words, the initial conditions giving rise to orbits that are
``bad'' on the time interval
$[t_{k,\varepsilon},t_{k+1,\varepsilon}] $, moved forward by time
$t_{k,\varepsilon}$, are initial conditions giving rise to orbits
which are ``bad'' on the time interval
$[t_{0,\varepsilon},t_{1,\varepsilon}] $. Because the flow
$z_\varepsilon(\cdot) $ preserves the measure, we find that
\[
    \varepsilon L(\varepsilon)\sum_{k=0}^{\frac{T}{\varepsilon
    L(\varepsilon)}-1}P(\mathcal{A}_{k,\varepsilon})
    \leq \C\, P(\mathcal{A}_{0,\varepsilon}).
\]

To estimate $P(\mathcal{A}_{0,\varepsilon})$, it is convenient to
use a different probability measure, which is uniformly equivalent
to $P$ on the set $\set{z\in\mathcal{M}:h(z)\in\mathcal{V}}
\supset\{\tilde{T}_\varepsilon\geq \varepsilon L(\varepsilon)\} $.
We denote this new probability measure by $P^f$, where the $f$
stands for ``factor.''  If we choose coordinates on $\mathcal{M} $
by using $h$ and the billiard coordinates on the two gas particles,
then $P^f$ is defined on $\mathcal{M}$ by $dP^f
=dh\,d\mu^1_h\,d\mu^2_h$, where $dh$ represents the uniform measure
on $\mathcal{V}\subset\mathbb{R}^4$, and the factor measure
$d\mu^i_h$ represents the invariant billiard measure of the $i^ {th}
$ gas particle coordinates for a fixed value of the slow variables.
One can verify that $1_{\set{h(z)\in\mathcal{V}}}dP \leq \C\, dP^f$,
but that $P^f$ is not invariant under the flow $z_\varepsilon(\cdot)
$ when $\varepsilon>0$.

We abuse notation, and consider $\mu^1_h$ to be a measure on the
left particle's initial billiard coordinates once $h$ and the
initial coordinates of the right gas particle are fixed.  In this
context, $\mu^1_h$ is simply the measure $\mu$ from
Subsection~\ref{sct:billiard}.  Then
\[
\begin {split}
    &
    P^f(\mathcal{A}_{0,\varepsilon})
    \\
    &\leq \int
    dh\,d\mu^2_h \cdot\mu_h^1
    \set{z:\abs{\frac{1}{L(\varepsilon)}\int_0^{L(\varepsilon)}
    G(z_\varepsilon(s))-G(z_0(s))ds}\geq\frac{\delta}{2T}
    \text { and } \varepsilon
    L(\varepsilon)\leq\tilde{T}_\varepsilon },
\end {split}
\]
and we must show that the last term tends to $0$ with $\varepsilon$.
By the Bounded Convergence Theorem, it suffices to show that for
almost every $h\in\mathcal{V}$ and initial condition for the right
gas particle,
\begin {equation}\label{eq:Gronwall_probability}
    \mu_h^1
    \set{z:\abs{\frac{1}{L(\varepsilon)}\int_0^{L(\varepsilon)}
    G(z_\varepsilon(s))-G(z_0(s))ds}\geq\frac{\delta}{2T}
    \text { and } \varepsilon
    L(\varepsilon)\leq\tilde{T}_\varepsilon }
    \rightarrow 0\text{ as }\varepsilon\rightarrow 0.
\end {equation}

Note that if $G$ were a smooth function and $z_\varepsilon(\cdot) $
were the flow of a smooth family of vector fields $Z(z,\varepsilon)
$ that depended smoothly on $\varepsilon$, then from Gronwall's
Inequality, it would follow that $\sup_{0\leq t\leq
L(\varepsilon)}\abs{z_{\varepsilon}(t)-
    z_0(t)}\leq \mathcal{O}(\varepsilon L(\varepsilon)
    e^{ \Lip{Z} L(\varepsilon)}).$
If this were the case, then $\abs{L(\varepsilon)^
{-1}\int_0^{L(\varepsilon)}
G(z_\varepsilon(s))-G(z_0(s))ds}=\mathcal{O}(\varepsilon
L(\varepsilon) e^{ \Lip{Z} L(\varepsilon)})$, which would tend to
$0$ with $\varepsilon$.  Thus, we need a Gronwall-type inequality
for billiard flows.  We obtain the appropriate estimates in Section
\ref{sct:Gronwall}.

\paragraph*{Step 6:  Use of ergodicity along fibers to
control $II_{k,\varepsilon} $.}

All that remains to be shown is that for fixed $\delta > 0$,
$P\left(\varepsilon\sum_{k=0}^{\frac{\tilde{T}_\varepsilon}{\varepsilon
    L(\varepsilon)}-1}\abs{II_{k,\varepsilon}}\geq
\delta\right)\rightarrow 0$ as $\varepsilon\rightarrow 0$.

For initial conditions $z\in \mathcal{M}$ and for integers $k\in
[0,T/(\varepsilon L(\varepsilon))-1]$ define
\[
\begin {split}
    \mathcal{B}_{k,\varepsilon} & =\set{z:\frac{1}{L(\varepsilon)}
    \abs{II_{k,\varepsilon}}
    >\frac{\delta}{2T} \text { and } k\leq\frac{\tilde{T}_\varepsilon}{\varepsilon
    L(\varepsilon)}-1 }
    ,\\
    \mathcal{B}_{z,\varepsilon} & =\set{k:z\in
    \mathcal{B}_{k,\varepsilon}}.
\end {split}
\]
Think of these sets as describing ``bad ergodization.''  For
example, roughly speaking, $z\in\mathcal{B}_{k,\varepsilon}$ if the
orbit $z_{\varepsilon}(t)$ starting at $z$ spends the time between
$t_{k,\varepsilon}$ and $t_{k+1,\varepsilon} $ in a region of phase
space where the function $G(\cdot) $ is ``poorly ergodized'' on the
time scale $L(\varepsilon) $ by the flow $z_0(t) $ (as measured by
the parameter $\delta/2T$).  Note that $G(z)=\abs{v_1^{\perp }}
\delta_{q_1^{\perp }=Q}$ is not really a function, but that we may
still speak of the convergence of $t^ {-1}\int_0^t G(z_0(s))ds$ as
$t\rightarrow\infty$.  As we showed in
Lemma~\ref{lem:ae_convergence}, the limit is $\bar G(h_0) $ for
almost every initial condition.

Proceeding as in Step 5 above, we find that it suffices to show that
for almost every $h\in\mathcal{V}$,
\[
    \mu_h^1
    \set{z:\abs{\frac{1}{t}\int_0^{t}
    G(z_0(s))ds-\bar G(h_0(0))}\geq\frac{\delta}{2T}}
    \rightarrow 0\text{ as }t\rightarrow \infty.
\]
But this is simply a question of examining billiard flows, and it
follows immediately from Corollary \ref{cor:ae_convergence} and our
Main Assumption.

\subsection{A Gronwall-type inequality for billiards}\label{sct:Gronwall}

We begin by presenting a general version of Gronwall's Inequality
for billiard maps.  Then we will show how these results imply the
convergence required in Equation~\eqref{eq:Gronwall_probability}.

\subsubsection{Some inequalities for the collision map}\label{sct:Gronwall_map}

In this section, we consider the value of the slow variables to be
fixed at $h_0\in\mathcal{V} $.  We will use the notation and results
presented in Section~\ref{sct:billiard}, but because the value of
the slow variables is fixed, we will omit it in our notation.

Let $\rho $, $\gamma $, and $\lambda$ satisfy $0<\rho\ll\gamma \ll
1\ll \lambda<\infty $. Eventually, these quantities will be chosen
to depend explicitly on $\varepsilon$, but for now they are fixed.

Recall that the phase space $\Omega$ for the collision map $F$ is a
finite union of disjoint rectangles and cylinders.  Let
$d(\cdot,\cdot)$ be the Euclidean metric on connected components of
$\Omega$. If $x$ and $x'$ belong to different components, then we
set $d(x,x') =\infty $. The invariant measure $\nu$ satisfies
$\nu<\C\cdot (\text {Lebesgue measure}) $.  For $A\subset\Omega$ and
$a >0$, let $\mathcal{N}_a (A) =\set {x\in\Omega:d(x,A)< a} $ be the
$a$-neighborhood of $A$.

For $x\in\Omega$ let $x_k(x) = x_k =F^k x$, $ k\geq 0$, be its
forward orbit.  Suppose $x\notin \mathcal{C}_{\gamma ,\lambda}$,
where
\[
    \mathcal{C}_{\gamma,\lambda}=
    \bigl(\cup_{k=0} ^\lambda F^
    {-k}\mathcal{N}_\gamma (\partial\Omega)\bigr)\bigcup\bigl(\cup_{k=0}
    ^\lambda F^ {-k}\mathcal{N}_\gamma  (F^ {-1}\mathcal{N}_\gamma
    (\partial\Omega))\bigr).
\]
Thus for $0\leq k\leq \lambda $, $x_k$ is well defined, and from
Equation~\eqref{eq:2d_derivative_bound} it satisfies
\begin {equation}\label{eq:gron0}
    d(x',x_k)\leq \gamma \;\Rightarrow\; d(Fx',x_{k+1})\leq\frac{\C}{\gamma
    } d(x',x_k).
\end {equation}

Next, we consider any $\rho$-pseudo-orbit $x'_k$ obtained from $x$
by adding on an error of size $\leq\rho$ at each application of the
map, i.e.~$d(x'_0,x_0)\leq \rho$, and for $k\geq 1$, $d(x'_k,Fx'_{
k-1})\leq \rho$.  Provided $d(x_j,x'_j)<\gamma $ for each $j<k$, it
follows that
\begin {equation}\label{eq:gron1}
    d(x_k,x'_k)\leq
    \rho\sum_{j=0}^{k}\left(\frac{\C}{\gamma}\right)^j\leq
    \C\,\rho \left (\frac {\C} {\gamma}\right) ^k .
\end {equation}
In particular, if $\rho $, $\gamma $, and $\lambda$ were chosen such
that
\begin {equation}\label{eq:gron2}
    \C\,\rho \left (\frac {\C} {\gamma}\right) ^\lambda<\gamma,
\end {equation}
then Equation~\eqref{eq:gron1} will hold for each $k\leq\lambda $.
We assume that Equation~\eqref{eq:gron2} is true.  Then we can also
control the differences in elapsed flight times using
Equation~\eqref{eq:2d_time_derivative}:
\begin {equation}\label{eq:gron3}
    \abs{\zeta x_k-\zeta x'_k}
    \leq
    \frac {\C\,\rho} {\gamma } \left (\frac {\C} {\gamma}\right) ^k.
\end {equation}

It remains to estimate the size $\nu \mathcal{C}_{\gamma,\lambda}$
of the set of $x$ for which the above estimates do not hold.  Using
Lemma~\ref{lem:gron1} below,
\begin {equation}\label{eq:gron4}
    \nu\mathcal{C}_{\gamma,\lambda}
    \leq
    (\lambda +1)\bigl(\nu \mathcal{N}_\gamma (\partial\Omega)+\nu \mathcal{N}_\gamma
    (F^ {-1}\mathcal{N}_\gamma
    (\partial\Omega))  \bigr)
    \leq
    \mathcal{O} (\lambda (\gamma +\gamma  ^ {1/3})) =\mathcal{O} (\lambda \gamma  ^
    {1/3}).
\end {equation}

\begin{lem}\label{lem:gron1}
    As $\gamma \rightarrow 0$,
    \[
        \nu \mathcal{N}_\gamma  (F^ {-1}\mathcal{N}_\gamma
        (\partial\Omega))=\mathcal{O}(\gamma  ^ {1/3}).
    \]
\end {lem}

This estimate is not necessarily the best possible.  For example,
for dispersing billiard tables, where the curvature of the boundary
is positive, one can show that $\nu \mathcal{N}_\gamma  (F^
{-1}\mathcal{N}_\gamma (\partial\Omega))=\mathcal{O}(\gamma  )$.
However, the estimate in Lemma~\ref{lem:gron1} is general and
sufficient for our needs.

\begin {proof}
First, we note that it is equivalent to estimate $\nu
\mathcal{N}_\gamma (F\mathcal{N}_\gamma (\partial\Omega))$, as $F$
has the measure-preserving involution $\mathcal{I} (r,\varphi) = (r,
-\varphi) $, i.e.~$F^ {-1} =\mathcal{I}\circ F\circ\mathcal{I}
$~\cite{CherMark06}.

Fix $\alpha\in (0,1/2) $, and cover $\mathcal{N}_\gamma
(\partial\Omega) $ with $\mathcal{O} (\gamma ^ {-1}) $ starlike
sets, each of diameter no greater than $\mathcal{O} (\gamma) $.  For
example, these sets could be squares of side length $\gamma $.
Enumerate the sets as $\set {A_i} $.  Set $\mathcal{G}=\set
{i:FA_i\cap \mathcal{N}_{\gamma ^\alpha} (\partial\Omega)
=\varnothing}$.

If $i\in \mathcal{G}$, $F\arrowvert _{A_i} $ is a diffeomorphism
satisfying $\norm{DF\arrowvert_{A_i}}\leq\mathcal{O} (\gamma ^
{-\alpha}) $.  See Equation~\eqref{eq:2d_derivative_bound}.  Thus
$\dia{FA_i}\leq\mathcal{O} (\gamma ^ {1-\alpha}) $, and so
\[
    \dia{\mathcal{N}_\gamma (FA_i)}\leq\mathcal{O} (\gamma ^
    {1-\alpha}).
\]
Hence $\nu\mathcal{N}_\gamma (FA_i)\leq\mathcal{O} (\gamma ^
{2(1-\alpha)}) $, and $\nu\mathcal{N}_\gamma
(\cup_{i\in\mathcal{G}}FA_i)\leq\mathcal{O} (\gamma ^ {1-2\alpha})
$.

If $i\notin \mathcal{G}$,  $A_i\cap F^ {-1} (\mathcal{N}_{\gamma
^\alpha} (\partial\Omega)) \neq\varnothing $.  Thus $A_i$ might be
cut into many pieces by $F^ {-1} (\partial\Omega) $, but each of
these pieces must be mapped near $\partial\Omega $.  In fact,
$FA_i\subset\mathcal{N}_{\mathcal{O} (\gamma ^\alpha)}
(\partial\Omega) $.  This is because outside $F^ {-1}
(\mathcal{N}_{\gamma ^\alpha} (\partial\Omega))$,
$\norm{DF}\leq\mathcal{O} (\gamma ^ {-\alpha}) $, and so points in
$FA_i$ are no more than a distance $\mathcal{O} (\gamma /\gamma ^
{\alpha}) $ away from $\mathcal{N}_{\gamma^\alpha} (\partial\Omega)
$, and $\gamma <\gamma  ^ {1-\alpha} <\gamma  ^\alpha $.  It follows
that $\mathcal{N}_\gamma (FA_i) \subset\mathcal{N}_{\mathcal{O}
(\gamma ^\alpha)} (\partial\Omega)$, and
$\nu\mathcal{N}_{\mathcal{O} (\gamma ^\alpha)}
(\partial\Omega)=\mathcal{O} (\gamma ^\alpha).  $

Thus $\nu \mathcal{N}_\gamma  (F^ {-1}\mathcal{N}_\gamma
(\partial\Omega))=\mathcal{O}(\gamma  ^ {1-2\alpha}+\gamma  ^
{\alpha})$, and we obtain the lemma by taking $\alpha=1/3$.

\end {proof}

\subsubsection{Application to a perturbed billiard flow}
\label{sct:gron_ap}

Returning to the end of Step 5 in Section~\ref{sct:main_steps}, let
the initial conditions of the slow variables be fixed at
$h_0=(Q_0,W_0,E_{1,0},E_{2,0})\in\mathcal{V} $ throughout the
remainder of this section.  We can assume that the billiard dynamics
of the left gas particle in $\mathcal{D}_1(Q_0) $ are ergodic. Also,
fix a particular value of the initial conditions for the right gas
particle for the remainder of this section.  Then $z_\varepsilon(t)
$ and $\tilde T_\varepsilon$ may be thought of as random variables
depending on the left gas particle's initial conditions
$y\in\mathcal{M} ^1$. Now if $h_\varepsilon (t)=
(Q_\varepsilon(t),W_\varepsilon(t),E_{1,\varepsilon}(t),E_{2,\varepsilon}(t))$
denotes the actual motions of the slow variables when
$\varepsilon>0$, it follows from Equation~\eqref{eq:h_div} that,
provided $\varepsilon L(\varepsilon)\leq \tilde {T}_\varepsilon $,
\begin {equation}\label{eq:h_div2}
    \sup_{0\leq t\leq
    L(\varepsilon)}\abs{h_0-h_\varepsilon(t)}=\mathcal{O}(\varepsilon L(\varepsilon)).
\end {equation}
Furthermore, we only need to show that
\begin {equation}\label{eq:gron5}
    \mu
    \set{y\in\mathcal{M} ^1:\abs{\frac{1}{L(\varepsilon)}\int_0^{L(\varepsilon)}
    G(z_\varepsilon(s))-G(z_0(s))ds}\geq\frac{\delta}{2T}
    \text { and } \varepsilon
    L(\varepsilon)\leq\tilde{T}_\varepsilon }
    \rightarrow 0
\end {equation}
as $\varepsilon\rightarrow 0$, where $G$ is defined in
Equation~\eqref{eq:G_definition}.

For definiteness, we take the following quantities from
Subsection~\ref{sct:Gronwall_map} to depend on $\varepsilon$ as
follows:
\begin {equation}\label{eq:gron6}
\begin {split}
    L(\varepsilon) &= L=\log \log\frac{1}{\varepsilon},
    \\
    \gamma (\varepsilon) &= \gamma =e^{-L},
    \\
    \lambda(\varepsilon)&=\lambda=
    \frac{2}{E_{\nu}\zeta}L,
    \\
    \rho(\varepsilon) &=\rho=\C\frac {\varepsilon L} {\gamma }.
\end {split}
\end {equation}
The constant in the choice of $\rho$ and $\rho$'s dependence on
$\varepsilon$ will be explained in the proof of
Lemma~\ref{lem:gron3}, which is at the end of this subsection. The
other choices may be explained as follows. We wish to use continuity
estimates for the billiard map to produce continuity estimates for
the flow on the time scale $L$. As the divergence of orbits should
be exponentially fast, we choose $L$ to grow sublogarithmically in
$\varepsilon^ {-1} $.  Since from Equation~\eqref{eq:2d_Santalo} the
expected flight time between collisions with
$\partial\mathcal{D}_1(Q_0)$ when $\varepsilon=0$ is
$E_{\nu}\zeta=\pi\abs{\mathcal{D}_1(Q_0)}/(\sqrt{2E_{1,0}}\abs{\partial\mathcal{D}_1(Q_0)})$,
we expect to see roughly $\lambda/2$ collisions on this time scale.
Considering $\lambda$ collisions gives us some margin for error.
Furthermore, we will want orbits to keep a certain distance, $\gamma
$, away from the billiard discontinuities. $\gamma \rightarrow 0$ as
$\varepsilon\rightarrow 0$, but $\gamma  $ is very large compared to
the possible drift $\mathcal{O} (\varepsilon L) $ of the slow
variables on the time scale $L$.  In fact, for each $C,m,n>0$,
\begin {equation}\label{eq:gron7}
    \frac {\varepsilon L^m } {\gamma ^n}
    \left (\frac {C} {\gamma}\right)^\lambda=\mathcal{O}
    (\varepsilon\, e^{\C\,L^2})
    \rightarrow 0\text { as }\varepsilon\rightarrow 0.
\end {equation}

Let $X:\mathcal{M} ^1\rightarrow\Omega$ be the map taking
$y\in\mathcal{M} ^1$ to $x=X(y)\in\Omega $, the location of the
billiard orbit of $y$ in the collision cross-section that
corresponds to the most recent time in the past that the orbit was
in the collision cross-section.  We consider the set of initial
conditions
\[
    \mathcal{E}_\varepsilon=
    X^{-1}(\Omega\backslash\mathcal{C}_{\gamma ,\lambda})\bigcap
    X^{-1}
    \set {x\in\Omega:  \sum_{k=0}^\lambda \zeta (F^k x)>
    L}.
\]
Now from Equations~\eqref{eq:gron4} and~\eqref{eq:gron6},
$\nu\mathcal{C}_{\gamma ,\lambda}\rightarrow 0$ as
$\varepsilon\rightarrow 0$.  Furthermore, by the ergodicity of $F$,
\[
    \nu\set {x\in\Omega:\sum_{k=0}^\lambda \zeta (F^k x)\leq L}=\nu\set
    {x\in\Omega:\lambda^ {-1}\sum_{k=0}^\lambda \zeta (F^k x)\leq E_\nu
    \zeta/2}\rightarrow 0
\]
as $\varepsilon\rightarrow 0$. But because the free flight time is
bounded above, $\mu X^ {-1}\leq \C\cdot \nu $, and so
$\mu\mathcal{E}_\varepsilon\rightarrow 1$ as $\varepsilon\rightarrow
0$.  Hence, the convergence in Equation~\eqref{eq:gron5} and the
conclusion of the proof in Section~\ref{sct:main_steps} follow from
the lemma below and Equation~\eqref{eq:gron7}.

\begin{lem}[Analysis of deviations along good orbits]\label{lem:gron2}
As $\varepsilon\rightarrow 0$,
\[
    \sup_{y\in\mathcal{E}_\varepsilon \cap \set{\varepsilon L\leq \tilde {T}_\varepsilon
    }}
    \abs{\frac{1}{L}\int_0^{L}
    G(z_\varepsilon(s))-G(z_0(s))ds}=
    \mathcal{O} \left(\rho
    \left(\frac {\C} {\gamma} \right)^ {\lambda}\right)
    +\mathcal{O}(L^ {- 1})\rightarrow 0.
\]

\end {lem}

\begin {proof}

Fix a particular value of $y\in\mathcal{E}_\varepsilon \cap
\set{\varepsilon L\leq \tilde {T}_\varepsilon}$. For convenience,
suppose that $y=X(y) =x\in\Omega $.  Let $y_0(t) $ denote the time
evolution of the billiard coordinates for the left gas particle when
$\varepsilon=0$.  Then there is some $N\leq \lambda$ such that the
orbit $ x_k=F^k x= ( r_k,\varphi_k)$ for $0\leq k\leq N $
corresponds to all of the instances (in order) when $y_0(t) $ enters
the collision cross-section $\Omega=\Omega_{ h_0} $ corresponding to
collisions with $\partial\mathcal{D}_1(Q_0) $ for $0\leq t\leq L$.
We write $\Omega_{ h_0}$ to emphasize that in this subsection we are
only considering the collision cross-section corresponding to the
billiard dynamics in the domain $\mathcal{D}_1(Q_0) $ at the energy
level $ E_{1,0} $.  In particular, $F$ will always refer to the
return map on $\Omega_{ h_0}$.

Also, define an increasing sequence of times $ t_k$ corresponding to
the actual times $ y_0(t) $ enters the collision cross-section, i.e.
\[
\begin {split}
    t_0 &=0,\\
    t_k & = t_{k-1} +\zeta x_{k-1}\text { for } k>0.
\end {split}
\]
Then $ x_k= y_0 ( t_k) $.  Furthermore, define inductively
\[
\begin {split}
    N_1&=\inf\set{k>0: t_k\text{ corresponds to a collision with the
    piston}},\\
    N_j&=\inf\set{k>N_{j-1}: t_k\text{ corresponds to a collision with the
    piston}}.\\
\end {split}
\]

Next, let $y_\varepsilon(t) $ denote the time evolution of the
billiard coordinates for the left gas particle when $\varepsilon>0$.
We will construct a pseudo-orbit $x_{k,\varepsilon}' =
(r_{k,\varepsilon}',\varphi_{k,\varepsilon}')$ of points in
$\Omega_{ h_0}$ that essentially track the collisions (in order) of
the left gas particle with the boundary under the dynamics of
$y_\varepsilon(t) $ for $0\leq t\leq L$.

First, define an increasing sequence of times $ t_{k,\varepsilon}'$
corresponding to the actual times $ y_\varepsilon(t) $ experiences a
collision with the boundary of the gas container or the moving
piston.  Define
\[
\begin {split}
    N_{\varepsilon}'&=\sup\set{k\geq 0: t_{k,\varepsilon}'
    \leq L},\\
    N_{1,\varepsilon}'&=\inf\set{k>0: t_{k,\varepsilon}'
    \text{ corresponds to a collision with the
    piston}},\\
    N_{j,\varepsilon}'&=\inf\set{k>N_{j-1,\varepsilon}': t_{k,\varepsilon}'
    \text{ corresponds to a collision with the
    piston}}.\\
\end {split}
\]
Because $L\leq \tilde {T}_\varepsilon(y)/\varepsilon$, we know that
as long as $N_{j+1,\varepsilon}'\leq N_{\varepsilon}'$, then
$N_{j+1,\varepsilon}'-N_{j,\varepsilon}'\geq 2$.  See the discussion
in Subsection~\ref{sct:collisions}.  Then we define
$x_{k,\varepsilon}'\in\Omega_{ h_0}$ by
\[
    x_{k,\varepsilon}' =
    \begin {cases}
    y_\varepsilon (t_{k,\varepsilon}')
    \text { if }k\notin\set {N_{j,\varepsilon}'},\\
    F^ {-1}x_{k+1,\varepsilon}'
    \text { if }k\in\set {N_{j,\varepsilon}'}.
    \end {cases}
\]

\begin{lem}\label{lem:gron3}
Provided $\varepsilon$ is sufficiently small, the following hold for
each $k\in [0,N\wedge N_\varepsilon')$.  Furthermore, the requisite
smallness of $\varepsilon$ and the sizes of the constants in these
estimates may be chosen independent of the initial condition
$y\in\mathcal{E}_\varepsilon \cap \set{\varepsilon L\leq \tilde
{T}_\varepsilon}$ and of $k$:
\begin {itemize}
    \item[\emph{(a)}]
        $x_{k,\varepsilon}' $ is well defined.  In particular, if
        $k\notin\set{N_{j,\varepsilon}'} $,
        $y_\varepsilon(t_{k,\varepsilon}') $ corresponds to a
        collision point on $\partial\mathcal{D}_1( Q_0)$, and not to
        a collision point on a piece of $\partial\mathcal{D} $ to
        the right of $ Q_0$.
    \item[\emph{(b)}]
        If $ k>0$ and $k\notin\set{N_{j,\varepsilon}'} $, then
        $x_{k,\varepsilon}' =Fx_{k-1,\varepsilon}' $.
    \item[\emph{(c)}]
        If $ k>0$ and $k\in\set{N_{j,\varepsilon}'} $, then
        $d(x_{k,\varepsilon}',Fx_{k-1,\varepsilon}')\leq\rho$ and
        the $\varphi$ coordinate of $y_\varepsilon(t_{k,\varepsilon}') $
        satisfies
        $\varphi(y_\varepsilon(t_{k,\varepsilon}')) =\varphi_{k,\varepsilon}' +
        \mathcal{O} (\varepsilon).$
    \item[\emph{(d)}]
        $d(x_k,x'_{k,\varepsilon})\leq\C\,\rho (\C/\gamma) ^k$ .
    \item[\emph{(e)}]
        $k=N_{j,\varepsilon}'$ if and only if $k=N_j$.
    \item[\emph{(f)}]
        If $ k>0$, $t_{k,\varepsilon}'-t_{k-1,\varepsilon}'
        =
        t_k - t_{ k-1} +
        \mathcal{O}(\rho \left (\C/\gamma\right) ^k).$
\end {itemize}

\end{lem}

We defer the proof of Lemma~\ref{lem:gron3} until the end of this
subsection.  Assuming that $\varepsilon$ is sufficiently small for
the conclusions of Lemma~\ref{lem:gron3} to be valid, we continue
with the proof of Lemma~\ref{lem:gron2}.

Set $M=N\wedge N_\varepsilon'-1$. Note that $M\leq\lambda\sim L $.
From (f) in Lemma~\ref{lem:gron3} and Equations~\eqref{eq:gron6} and
\eqref{eq:gron7}, we see that
\[
\begin {split}
    \abs{t_M-t_{M,\varepsilon}'}
    &
    \leq \sum_{k=1}^M
    \abs{t_{k,\varepsilon}'-t_{k-1,\varepsilon}'- (t_k - t_{ k-1})}
    =
    \mathcal{O}\left(\rho \frac{\C^\lambda}{\gamma^{\lambda}}\right)
    \rightarrow 0\text{ as }\varepsilon\rightarrow 0.
\end {split}
\]
Because the flight times $t_{k,\varepsilon}'-t_{k-1,\varepsilon}'$
and $t_k - t_{ k-1}$ are uniformly bounded above, it follows from
the definitions of $N$ and $N_\varepsilon' $ that $t_M,\,
t_{M,\varepsilon}'\geq L-\C$.  But from
Subsection~\ref{sct:collisions}, the time between the collisions of
the left gas particle with the piston are uniformly bounded away
from zero. Using (c) and Equation~\eqref{eq:h_div2}, it follows that
\[
\begin {split}
    &\abs{\frac{1}{L}
    \int_0^{L}G(z_\varepsilon(s))-G(z_0(s))ds}
    \\
    &\qquad
    =\mathcal{O} (L^ {-1}) +
    \sum_{k\in \set { N_j:N_j\leq M}} \abs{\sqrt{2E_{1,0}}\,\cos \varphi_k
    -\sqrt{2E_{1,\varepsilon}(t_{k,\varepsilon}')}\,\cos
    (\varphi_{k,\varepsilon}'+\mathcal{O} (\varepsilon))}
    \\
    &\qquad
    =\mathcal{O} (L^ {-1}) +
    \sum_{k\in \set { N_j:N_j\leq M}}
    \abs{\sqrt{2E_{1,0}}\,\cos \varphi_k
    -\sqrt{2E_{1,0}}\,\cos
    \varphi_{k,\varepsilon}'
    +\mathcal{O} (\varepsilon L)}
    \\
    &\qquad
    =\mathcal{O} (L^ {-1}) +
    \mathcal{O} (\varepsilon L^2)
    +\sqrt{2E_{1,0}}\,\sum_{k\in \set { N_j:N_j\leq M}}
    \abs{\cos \varphi_k
    -\cos
    \varphi_{k,\varepsilon}'}.
\end {split}
\]
But using (d),
\[
\begin {split}
    \sum_{k\in \set { N_j:N_j\leq M}}
    \abs{\cos \varphi_k-\cos
    \varphi_{k,\varepsilon}'}
    \leq\sum_{ k=0} ^M\mathcal{O}  (\rho (\C/\gamma) ^k)
    =\mathcal{O}  (\rho (\C/\gamma) ^\lambda).
\end {split}
\]
Since $\varepsilon L^2=\mathcal{O}(\rho (\C/\gamma) ^\lambda) $,
this finishes the proof of Lemma~\ref{lem:gron2}.

\end {proof}

\begin {proof}[Proof of Lemma~\ref{lem:gron3}]

The proof is by induction.  We take $\varepsilon$ to be so small
that Equation~\eqref{eq:gron2} is satisfied. This is possible by
Equation~\eqref{eq:gron7}.

It is trivial to verify (a)-(f) for $ k=0$. So let $0<l<N\wedge
N_\varepsilon' $, and suppose that (a)-(f) have been verified for
all $k<l$.  We have three cases to consider:

\subsubsection*{Case 1: $l-1$ and $l\notin \set{N_{j,\varepsilon}'}$:}

In this case, verifying (a)-(f) for $ k=l$ is a relatively
straightforward application of the machinery developed in
Subsection~\ref{sct:Gronwall_map}, because for
$t_{l-1,\varepsilon}'\leq t\leq t_{l,\varepsilon}'$, $y_\varepsilon
(t) $ traces out the billiard orbit between $x_{l-1,\varepsilon}'$
and $x_{l,\varepsilon}'$ corresponding to free flight in the domain
$\mathcal{D}_1( Q_0) $. We make only two remarks.

First, as long as $\varepsilon$ is sufficiently small, it really is
true that $x_{l,\varepsilon}'=y_\varepsilon (t_{l,\varepsilon}')$
corresponds to a true collision point on $\partial\mathcal{D}_1(
Q_0) $.  Indeed, if this were not the case, then it must be that
$Q_\varepsilon(t_{l,\varepsilon}')> Q_0 $, and $y_\varepsilon
(t_{l,\varepsilon}')$ would have to correspond to a collision with
the side of the ``tube'' to the right of $ Q_0$.  But then
$x_{l,\varepsilon}''  =Fx_{l-1,\varepsilon}'\in\Omega_{ h_0}$ would
correspond to a collision with an immobile piston at $ Q_0$ and
would satisfy $d(x_k,x''_{k,\varepsilon})\leq\C\,\rho (\C/\gamma)
^k\leq \C\,\rho (\C/\gamma) ^\lambda =o(\gamma )$, using
Equations~\eqref{eq:gron1} and \eqref{eq:gron7}.  But $
x_k\notin\mathcal{N}_\gamma (\partial\Omega_{ h_0}) $, and so it
follows that when the trajectory of $y_\varepsilon(t) $ crosses the
plane $\set {Q= Q_0} $, it is at least a distance $\sim \gamma $
away from the boundary of the face of the piston, and its velocity
vector is pointed no closer than $\sim \gamma $ to being parallel to
the piston's face.  As $Q_\varepsilon(t_{l,\varepsilon}')-
Q_0=\mathcal{O} (\varepsilon L) =o(\gamma ) $, and it is
geometrically impossible (for small $\varepsilon$) to construct a
right triangle whose sides $ s_1,\: s_2$ satisfy $\abs{ s_1}\geq\sim
\gamma ,\:\abs{ s_2}\leq\mathcal{O} (\varepsilon L) $, with the
measure of the acute angle adjacent to $ s_1$ being greater than
$\sim \gamma$, we have a contradiction. After crossing the plane
$\set {Q= Q_0} $, $y_\varepsilon(t) $ must experience its next
collision with the face of the piston, which violates the fact that
$l\notin \set{N_{j,\varepsilon}'}$.

Second, $t_{l,\varepsilon}'-t_{l-1,\varepsilon}' =\zeta
x'_{l-1,\varepsilon}+\mathcal{O}(\varepsilon L)$, because
$v_{1,\varepsilon} =v_ {1,0} +\mathcal{O} (\varepsilon L) $.  See
Equation~\eqref{eq:h_div2}.  From Equation~\ref{eq:gron3},
$\abs{\zeta x_{l-1}-\zeta x'_{l-1,\varepsilon}} \leq \mathcal{O}
((\rho/\gamma)  \left (\C/\gamma\right) ^{l-1})$. As $t_l - t_{
l-1}=\zeta x_{l-1}$ and $\varepsilon L=\mathcal{O} ((\rho/\gamma)
\left (\C/\gamma\right) ^{l-1})$, we obtain (f).

\subsubsection*{Case 2: There exists $i$ such that $l=N_{i,\varepsilon}'$: }

For definiteness, we suppose that
$Q_\varepsilon(t_{l,\varepsilon}')\geq Q_0$, so that the left gas
particle collides with the piston to the right of $ Q_0$.  The case
when $Q_\varepsilon(t_{l,\varepsilon}')\leq Q_0$ can be handled
similarly.

We know that $ x_{l-1}, x_{l},x_{l+1}\notin \mathcal{N}_\gamma
(\partial\Omega_{ h_0})\cup\mathcal{N}_\gamma  (F^
{-1}\mathcal{N}_\gamma (\partial\Omega_{ h_0}))$.  Using the
inductive hypothesis and Equation~\eqref{eq:gron1}, we can define
\[
    x_{l,\varepsilon}''=Fx_{l-1,\varepsilon}',\qquad
    x_{l+1,\varepsilon}''=F^2x_{l-1,\varepsilon}',
\]
and $d( x_{l},x_{l,\varepsilon}'') \leq\C\,\rho (\C/\gamma) ^{l}$,
$d( x_{l+1}, x_{l+1,\varepsilon}'')\leq\C\,\rho (\C/\gamma) ^{l+1}$.
In particular, $x_{l,\varepsilon}''$ and $x_{l+1,\varepsilon}''$ are
both a distance $\sim \gamma $ away from $\partial\Omega_{ h_0} $.
Furthermore, when the left gas particle collides with the moving
piston, it follows from Equation~\eqref{eq:v_1Wchange} that the
difference between its angle of incidence and its angle of
reflection is $\mathcal{O} (\varepsilon) $.  Referring to
Figure~\ref{fig:collision}, this means that
$\varphi_{l,\varepsilon}' =\varphi_{l,\varepsilon}'' +\mathcal{O}
(\varepsilon) $. Geometric arguments similar to the one given in
Case 1 above show that the $y_\varepsilon$-trajectory of the left
gas particle has precisely one collision with the piston and no
other collisions with the sides of the gas container when the gas
particle traverses the region $Q_0\leq Q\leq
Q_\varepsilon(t_{l,\varepsilon}')$. Note that $x_{l,\varepsilon}' $
was defined to be the point in the collision cross-section $\Omega_{
h_0} $ corresponding to the return of the $y_\varepsilon$-trajectory
into the region $Q\leq Q_0$. See Figure~\ref{fig:collision}.  From
this figure, it is also evident that
$d(r_{l,\varepsilon}',r_{l,\varepsilon}'')\leq\mathcal{O}
(\varepsilon L/\gamma ) $.  Thus $d(
x_{l,\varepsilon}'',x_{l,\varepsilon}')=\mathcal{O} (\varepsilon
L/\gamma )$, and this explains the choice of $\rho(\varepsilon) $ in
Equation~\eqref{eq:gron6}.

\begin{figure}
    \begin {center}
    \setlength{\unitlength}{1.0 cm}
    \begin{picture}(15,10)
        \thicklines
        \put(1,1.5){\line(1,0){10.5}}
        \put(1,8.5){\line(1,0){10.5}}
        \put(11.5,1.5){\line(0,1){7.0}}
        \put(3.5,1.5){\vector(1,0){0.5}}
        \put(3.3,1.15){$r-$coordinate}
        \thinlines
        \put(1,1.5){\line(1,0){11}}
        \put(1,8.5){\line(1,0){11}}
        \put(11.5,1.4){\line(0,1){7.1}}
        \put(11.3,1.0){$Q_0$}
        \put(12,1.4){\line(0,1){7.1}}
        \put(11.9,1.0){$Q_\varepsilon(t_{l,\varepsilon}')$}
        \put(1.5,8){$\mathcal{D}_1(Q_0) $}
        \put(9.5,8.8){\line(1,0){4}}
        \put(11.5,8.8){\line(0,-1){0.15}}
        \put(13.5,8.8){\line(0,-1){0.15}}
        \put(9.5,8.8){\line(0,-1){0.15}}
        \put(  10.5,8.8){\line(0,1){0.15}}
        \put(10.2,9.1){$\gamma/2 $}
        \put(12.5,8.8){\line(0,1){0.15}}
        \put(12.2,9.1){$\gamma/2 $}
        \put(11.5,0.6){\line(1,0){0.5}}
        \put(11.75,0.6){\line(0,-1){0.15}}
        \put(  11.5,0.6){\line(0,1){0.15}}
        \put(  12,0.6){\line(0,1){0.15}}
        \put(11.4,0.1){$\mathcal{O} (\varepsilon L) $}
        \put(12.5,4.5){\line(0,1){1}}
        \put(12.5,4.5){\line(-1,0){0.15}}
        \put(12.5,5){\line(1,0){0.15}}
        \put(12.5,5.5){\line(-1,0){0.15}}
        \put(12.7,4.85){$\mathcal{O} (\varepsilon L/\gamma ) $}
        \put(12.5,6.5){\line(0,1){2}}
        \put(12.5,8.5){\line(-1,0){0.15}}
        \put(12.5,7.5){\line(1,0){0.15}}
        \put(12.5,6.5){\line(-1,0){0.15}}
        \put(12.7,7.45){$\gamma /2 $}
        \put(12.5,1.5){\line(0,1){2}}
        \put(12.5,3.5){\line(-1,0){0.15}}
        \put(12.5,2.5){\line(1,0){0.15}}
        \put(12.5,1.5){\line(-1,0){0.15}}
        \put(12.7,2.45){$\gamma  /2$}
    \small
        \put(12,5){\line(-1,-1){3.5}}
        \put(11.5,4.5){\vector(-1,1){0.8}}
        \put(11.5,4.5){\line(-1,1){4}}
        \put(11.65,4.4){$r_{l,\varepsilon}''$}
        \put(10.35,4.7){$\varphi_{l,\varepsilon}''$}
        \put(11.5,4.5){\circle*{.1}}
        \qbezier[24](10,4.5)(10.75,4.5)(11.5,4.5)
        \put(11.1,4.5){\vector(1,1){0.2}}
        \put(12,5){\vector(-1,1){1.3}}
        \put(12,5){\line(-1,1){3.5}}
        \put(11.65,5.4){$r_{l,\varepsilon}'$}
        \put(10.35,5.7){$\varphi_{l,\varepsilon}'$}
        \put(11.5,5.5){\circle*{.1}}
        \qbezier[24](10,5.5)(10.75,5.5)(11.5,5.5)
        \put(11.1,5.5){\vector(1,1){0.2}}
        \qbezier[12](11.5,5)(11.75,5)(12,5)
        \put(8.5,1.5){\circle*{.1}}
        \put(8.5,1.5){\vector(1,1){0.8}}
        \qbezier[24](8.5,1.5)(8.5,2.25)(8.5,3.0)
        \put(8.45,1.15){$r_{l-1,\varepsilon}'$}
        \put(8.57,2.6){$\varphi_{l-1,\varepsilon}'$}
        \put(8.5,1.9){\vector(1,-1){0.2}}
        \put(8.5,8.5){\circle*{.1}}
        \put(8.3,8.7){$r_{l+1,\varepsilon}'$}
        \put(7.5,8.5){\circle*{.1}}
        \put(6.95,8.7){$r_{l+1,\varepsilon}''$}
    \normalsize
    \end{picture}
    \end {center}
    \caption{An analysis of the divergences of orbits when $\varepsilon>0$
        and the left gas particle collides with the moving piston to the right of $Q_0$.  Note that the
        dimensions are distorted for visual clarity, but that $\varepsilon L$
        and $\varepsilon L/\gamma $ are both $o(\gamma ) $ as $\varepsilon\rightarrow
        0$.}  Furthermore, $\varphi_{l,\varepsilon}''\in(-\pi/2+\gamma/2,\pi/2-\gamma/2) $
        and $\varphi_{l,\varepsilon}' =\varphi_{l,\varepsilon}'' +\mathcal{O} (\varepsilon) $,
        and so $r_{l,\varepsilon}' =r_{l,\varepsilon}''+\mathcal{O}
        (\varepsilon L/\gamma ) $.  In particular, the
        $y_\varepsilon$-trajectory of the left gas particle has precisely
        one collision with the piston and no other collisions with the sides
        of the gas container when the gas particle traverses the region
        $Q_0\leq Q\leq Q_\varepsilon(t_{l,\varepsilon}')$
    \label{fig:collision}
\end{figure}
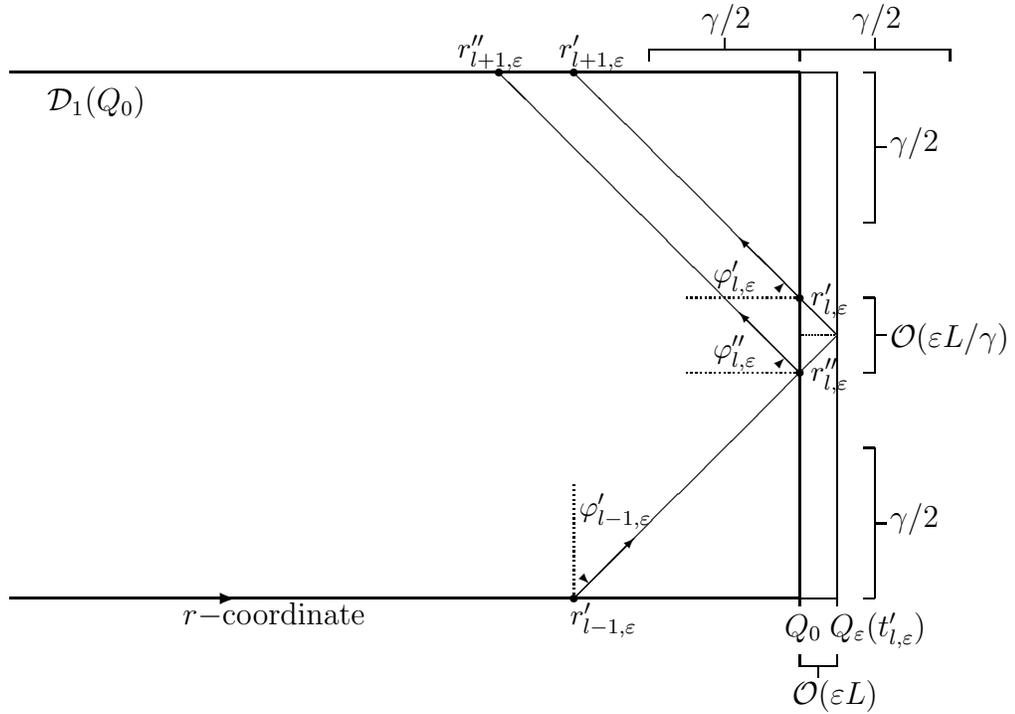

From the above discussion and the machinery of
Subsection~\ref{sct:Gronwall_map}, (a)-(e) now follow readily for
\emph{both} $ k=l$ and $ k=l+1$.  Furthermore, property (f) follows
in much the same manner as it did in Case 1 above.  However, one
should note that $t_{l,\varepsilon}'-t_{l-1,\varepsilon}' =\zeta
x'_{l-1,\varepsilon}+\mathcal{O}(\varepsilon L
)+\mathcal{O}(\varepsilon L/\gamma )$ and
$t_{l+1,\varepsilon}'-t_{l,\varepsilon}' =\zeta
x'_{l,\varepsilon}+\mathcal{O}(\varepsilon L
)+\mathcal{O}(\varepsilon L/\gamma )$, because of the extra distance
$\mathcal{O} (\varepsilon L/\gamma ) $ that the gas particle travels
to the right of $ Q_0$.  But $\varepsilon L/\gamma =\mathcal{O}
((\rho/\gamma) \left (\C/\gamma\right) ^{l-1})$, and so property (f)
follows.

\subsubsection*{Case 3: There exists $i$ such that $l-1=N_{i,\varepsilon}'$: }

As mentioned above, the inductive step in this case follows
immediately from our analysis in Case 2.

\end {proof}

\section{Generalization to a full proof\\ of Theorem~\ref{thm:dDpiston}}
\label{sct:generalization}

It remains to generalize the proof in Sections~\ref{sct:2dprep} and
\ref{sct:2dproof} to the cases when $ n_1, n_2\geq 1$ and $d=3$.

\subsection{Multiple gas particles on each side of the piston}\label{sct:multiple}

When $d=2$, but $n_1,n_2\geq 1$, only minor modifications are
necessary to generalize the proof above. As in
Subsection~\ref{sct:collisions}, one defines a stopping time $\tilde
{T}_\varepsilon$ satisfying $P\set {\tilde {T}_\varepsilon < T\wedge
T_\varepsilon} =\mathcal{O} (\varepsilon) $ such that for $0\leq
t\leq \tilde {T}_\varepsilon/\varepsilon$, gas particles will only
experience clean collisions with the piston.

Next, define $H(z) $ by
\[
    H(z) =
    \begin{bmatrix}
    W\\
    +2\sum_{j=1}^{n_1}\abs{v_{1,j}^{\perp }} \delta_{q_{1,j}^{\perp }=Q}
    -2\sum_{j=1}^{n_2}\abs{v_{2,j}^{\perp }} \delta_{q_{2,j}^{\perp }=Q}\\
    -2W\abs{v_{1,j}^{\perp }} \delta_{q_{1,j}^{\perp }=Q}\\
    +2W\abs{v_{2,j}^{\perp }} \delta_{q_{2,j}^{\perp }=Q}\\
    \end{bmatrix}.
\]
It follows that for $0 \leq t\leq \tilde
{T}_\varepsilon/\varepsilon$, $
    h_\varepsilon(t)-h_\varepsilon(0)=
    \mathcal{O}(\varepsilon)+\varepsilon\int_0^t
    H(z_\varepsilon(s))ds.
$ From here, the rest of the proof follows the same steps made in
Subsection~\ref{sct:main_steps}.  We note that at Step 3, we find
that $H(z) -\bar H(h(z)) $ divides into $n_1+ n_2 $ pieces, each of
which depends on only one gas particle when the piston is held
fixed.

\subsection{Three dimensions}\label{sct:higher_d}

The proof of Theorem~\ref{thm:dDpiston} in $d=3$ dimensions is
essentially the same as the proof in two dimensions given above. The
principal differences are due to differences in the geometry of
billiards.  We indicate the necessary modifications.

In analogy with Section~\ref{sct:billiard}, we briefly summarize the
necessary facts for the billiard flows of the gas particles when
$M=\infty $ and the slow variables are held fixed at a specific
value $h\in\mathcal{V} $.  As before, we will only consider the
motions of one gas particle moving in $\mathcal{D}_1 $.  Thus we
consider the billiard flow of a point particle moving inside the
domain $\mathcal{D}_1$ at a constant speed $\sqrt{2E_1} $.  Unless
otherwise noted, we use the notation from
Section~\ref{sct:billiard}.

The billiard flow takes place in the five-dimensional space
$\mathcal{M}^1=\{(q_1,v_1)\in\mathcal{TD}_1:q_1\in\mathcal{ D}_1,\;
\abs{v_1}=\sqrt{2E_1}\}/\sim$.  Here the quotient means that when
$q_1\in\partial\mathcal{ D}_1$, we identify velocity vectors
pointing outside of $\mathcal{D}_1$ with those pointing inside
$\mathcal{D}_1$ by reflecting orthogonally through the tangent plane
to $\partial\mathcal{D}_1$ at $ q_1$. The billiard flow preserves
Liouville measure restricted to the energy surface.  This measure
has the density $d\mu=dq_1dv_1/(8\pi E_1\abs{\mathcal{D}_1} ) $.
Here $dq_1$ represents volume on $\mathbb{R}^3$, and $ dv_1$
represents area on
$S^2_{\sqrt{2E_1}}=\set{v_1\in\mathbb{R}^3:\abs{v_1}=\sqrt{2E_1}}$.

The collision cross-section
$\Omega=\{(q_1,v_1)\in\mathcal{TD}_1:q_1\in\partial\mathcal{ D}_1,\;
\abs{v_1}=\sqrt{2E_1}\}/\sim$ is properly thought of as a fiber
bundle, whose base consists of the smooth pieces of
$\partial\mathcal{D}_1$ and whose fibers are the set of outgoing
velocity vectors at $q_1\in\partial\mathcal{ D}_1$.  This and other
facts about higher-dimensional billiards, with emphasis on the
dispersing case, can be found in~\cite{BalCheSzaTot_2003}.  For our
purposes, $\Omega$ can be parameterized as follows. We decompose
$\partial\mathcal{D}_1$ into a finite union $\cup_j \Gamma_j$ of
pieces, each of which is diffeomorphic via coordinates $r$ to a
compact, connected subset of $\mathbb{R}^2$ with a piecewise
$\mathcal{C} ^3$ boundary.  The $\Gamma_j$ are nonoverlapping,
except possibly on their boundaries. Next, if $(q_1,v_1)\in\Omega $
and $ v_1$ is the outward going velocity vector, let $\hat v =
v_1/\abs{v_1} $.  Then $\Omega$ can be parameterized by $\{x=(r,\hat
v)\}$. It follows that $\Omega$ it is diffeomorphic to $\cup_j
\Gamma_j\times S^{2 +}$, where $S^{2 +}$ is the upper unit
hemisphere, and by $\partial\Omega$ we mean the subset diffeomorphic
to $(\cup_j
\partial\Gamma_j\times S^{2 +})\bigcup (\cup_j \Gamma_j\times
\partial S^{2 +})$. If $x\in\Omega $, we let $\varphi\in [0,\pi /2]$
represent the angle between the outgoing velocity vector and the
inward pointing normal vector $n$ to $\partial\mathcal{D}_1$,
i.e.~$\cos\varphi=\langle \hat v, n\rangle$. Note that we no longer
allow $\varphi$ to take on negative values. The return map
$F:\Omega\circlearrowleft$ preserves the projected probability
measure $\nu $, which has the density $d\nu=\cos\varphi\, d\hat v \,
dr/(\pi\abs{\partial\mathcal{D}_1}) $. Here
$\abs{\partial\mathcal{D}_1}$ is the area of
$\partial\mathcal{D}_1$.

$F$ is an invertible, measure preserving transformation that is
piecewise $\mathcal{C} ^2$.  Because of our assumptions on
$\mathcal{D}_1$, the free flight times and the curvature of
$\partial\mathcal{D}_1$ are uniformly bounded.  The bound on $\norm
{DF(x)}$ given in Equation~\eqref{eq:2d_derivative_bound} is still
true.  A proof of this fact for general three-dimensional billiard
tables with finite horizon does not seem to have made it into the
literature, although see~\cite{BalCheSzaTot_2003} for the case of
dispersing billiards.  For completeness, we provide a sketch of a
proof for general billiard tables in Section~\ref{sct:d_bounds}.

We suppose that the billiard flow is ergodic, so that $F$ is
ergodic.  Again, we induce $F$ on the subspace $\hat\Omega$ of
$\Omega$ corresponding to collisions with the (immobile) piston to
obtain the induced map $\hat F:\hat\Omega\circlearrowleft$ that
preserves the induced measure $\hat \nu$.

The free flight time $\zeta:\Omega\rightarrow \mathbb{R}$ again
satisfies the derivative bound given in
Equation~\eqref{eq:2d_time_derivative}. The generalized
Santal\'{o}'s formula~\cite{Chernov1997} yields
\[
    E_\nu \zeta=\frac {4
    \abs{\mathcal{D}_1}} {\abs{v_1}\abs{\partial\mathcal{D}_1}}.
\]
If $\hat\zeta:\hat\Omega\rightarrow\mathbb{R} $ is the free flight
time between collisions with the piston, then it follows from
Proposition \ref{prop:inducing} that
\[
    E_{\hat\nu} \hat\zeta=\frac {4
    \abs{\mathcal{D}_1}}{\abs{v_1}\ell}.
\]

The expected value of $ \abs{v_1^\perp }$ when the left gas particle
collides with the (immobile) piston is given by
\[
    E_{\hat\nu} \abs{v_1^\perp }=E_{\hat\nu} \sqrt{2E_1}\cos\varphi=
    \frac{\sqrt{2E_1}}{\pi}\iint_{S^ {2+}} \cos^2\varphi\,d\hat v_1=
    \sqrt{2E_1}\frac{2}{3}.
\]

As a consequence, we obtain
\begin {lem}
\label{lem:ae_convergence_3d}

For $\mu-a.e.$ $y\in \mathcal{M}^1$,
\[
    \lim_{t\rightarrow\infty} \frac{1}{t}
    \int_0^t \abs{v_1^\perp (s)}\delta_{q_1^\perp (s)
    =Q}ds=
    \frac{E_1\ell}{3\abs{\mathcal{D}_1(Q)}}.
\]

\end {lem}
\noindent Compare the proof of Lemma~\ref{lem:ae_convergence}.

With these differences in mind, the rest of the proof of
Theorem~\ref{thm:dDpiston} when $d=3$ proceeds in the same manner as
indicated in Sections~\ref{sct:2dprep}, \ref{sct:2dproof} and
\ref{sct:multiple} above.  The only notable difference occurs in the
proof of the Gronwall-type inequality for billiards.  Due to
dimensional considerations, if one follows the proof of
Lemma~\ref{lem:gron1} for a three-dimensional billiard table, one
finds that \[\nu \mathcal{N}_\gamma  (F^ {-1}\mathcal{N}_\gamma
(\partial\Omega))=\mathcal{O}(\gamma  ^ {1-4\alpha}+\gamma  ^
{\alpha}).\]  The optimal value of $\alpha$ is $1/5$, and so $ \nu
\mathcal{N}_\gamma  (F^ {-1}\mathcal{N}_\gamma
(\partial\Omega))=\mathcal{O}(\gamma  ^ {1/5})$ as $\gamma
\rightarrow 0$.  Hence $\nu\mathcal{C}_{\gamma,\lambda} =\mathcal{O}
(\lambda \gamma  ^ {1/5})$, which is a slightly worse estimate than
the one in Equation~\eqref{eq:gron4}.  However, it is still
sufficient for all of the arguments in Section~\ref{sct:gron_ap},
and this finishes the proof.

\comment{
 can also be heuristically justified by
the procedure we used to justify the averaged equation when $d=2$.
Let $B^d=\{(x_1,\dots,x_d)\in\mathbb{R}^d :\sum_{i=1}^{d}x_i^2\leq
1\}$ denote the unit ball in $\mathbb{R}^d$, and let
$S^{d-1}=\{(x_1,\dots,x_d)\in B^d :\sum_{i=1}^{d}x_i^2= 1\}$ denote
the unit $(d-1)$-sphere.  Also let $(S^{d-1}) ^
+=\{(x_1,\dots,x_d)\in S^{d-1} :x_d\geq 0\}$.  Then if the piston is
held fixed, the expected flight time between collisions for the left
gas particle (with respect to the invariant billiard measure for the
billiard map induced on the subspace of collisions with the piston)
is
\[
    E_{\hat\nu} \hat\zeta=\frac {
    \abs{\mathcal{D}_1}}{\ell\sqrt{2E_1}}\frac {\abs{S^{d-1}}}
    {\abs{B^{d-1}}}.
\]
See \cite{CM06}. Furthermore,\marginal {I should say a word about
invariant measures/how to derive these equations} For future
reference, we observe that the expected value of $ \abs{v_1^\perp }$
when the left gas particle collides with the (immobile) piston is
given by
\[
    E_{\hat\nu} \abs{v_1^\perp }=
    \frac {\sqrt{2E_1}} {2d}\frac{\abs{S^{d-1}}}{\abs{B^{d-1}}}.
\]
} 


\section{Inducing maps on subspaces}
\label{sct:inducing}

Here we present some well-known facts on inducing measure preserving
transformations on subspaces. Let $F: (\Omega,
\mathfrak{B},\nu)\circlearrowleft$ be an invertible, ergodic,
measure preserving transformation of the probability space $\Omega$
endowed with the $\sigma$-algebra $\mathfrak{B}$ and the probability
measure $\nu$.  Let $\hat\Omega\in\mathfrak{B}$ satisfy $0<\nu
\hat\Omega<1$.  Define $R:\hat\Omega\rightarrow\mathbb{N}$ to be the
first return time to $\hat\Omega$, i.e.~$R\omega
=\inf\{n\in\mathbb{N}:F^n\omega \in\hat\Omega\}$.  Then if
$\hat{\nu} : =\nu(\cdot\cap\hat\Omega)/\nu\hat\Omega$ and
$\hat{\mathfrak{B}}: =\{B\cap\hat\Omega:B\in \mathfrak{B}\}$,
$\hat{F}: (\hat\Omega,
\hat{\mathfrak{B}},\hat{\nu})\circlearrowleft$ defined by
$\hat{F}\omega=F^{R\omega}\omega$ is also an invertible, ergodic,
measure preserving transformation~\cite{Pet83}. Furthermore
$E_{\hat{\nu}} R=\int_{\hat\Omega}
R\,d\hat{\nu}=(\nu\hat\Omega)^{-1}$.

This last fact is a consequence of the following proposition:
\begin {prop}
\label{prop:inducing}

If $\zeta:\Omega\rightarrow\mathbb{R}_{\geq 0}$ is in $L^1(\nu)$,
then $\hat\zeta =\sum_{n=0}^{R-1}\zeta\circ F^n$ is in
$L^1(\hat{\nu})$, and
\[
    E_{\hat{\nu}} \hat\zeta =\frac {1}{\nu\hat\Omega}
    E_{\nu}\zeta.
\]
\end{prop}

\begin {proof}
\[
\begin {split}
    \nu\hat\Omega \int_{\hat\Omega} \sum_{n=0}^{R-1}\zeta\circ F^n\,d\hat{\nu}
    &=
    \int_{\hat\Omega} \sum_{n=0}^{R-1}\zeta\circ F^n\,d\nu
    =
    \sum_{k=1}^\infty\int_{\hat\Omega\cap\{R=k\}}
    \sum_{n=0}^{k-1}\zeta\circ F^n\,d\nu
    \\
    &=
    \sum_{k=1}^\infty\sum_{n=0}^{k-1}\int_{F^n(\hat\Omega\cap\{R=k\})}
    \zeta\,d\nu
    =
    \int_{\Omega}\zeta\,d\nu,
\end {split}
\]
because $\{F^n(\hat\Omega\cap\{R=k\}):0\leq n< k<\infty\}$ is a
partition of $\Omega$.

\end {proof}

\section{Derivative bounds for the billiard map\\ in three dimensions}
\label{sct:d_bounds}

Returning to Section~\ref{sct:higher_d}, we need to show that for a
billiard table $\mathcal{D}_1\subset\mathbb{R}^3$ with a piecewise
$\mathcal{C} ^3$ boundary and the free flight time uniformly bounded
above, the billiard map $F$ satisfies the following: If $x_0\notin
\partial\Omega \cup F^{-1} (\partial\Omega) $, then
\begin{equation*}
    \norm {DF(x_0)}\leq\frac {\C} {\cos \varphi(Fx_0)}.
\end{equation*}

Fix $x_0= (r_0,\hat v_0)\in\Omega $, and let $x_1= (r_1,\hat
v_1)=Fx_0$.  Let $\Sigma$ be the plane that perpendicularly bisects
the straight line between $ r_0$ and $ r_1$, and let $r_{1/2} $
denote the point of intersection.  We consider $\Sigma$ as a
``transparent'' wall, so that in a neighborhood of $ x_0$, we can
write $F=F_2\circ F_1$.  Here, $F_1$ is like a billiard map in that
it takes points (i.e.~directed velocity vectors with a base) near $
x_0$ to points with a base on $\Sigma$ and a direction pointing near
$ r_1$. ($ F_1$ would be a billiard map if we reflected the image
velocity vectors orthogonally through $\Sigma$.)  $ F_2$ is a
billiard map that takes points in the image of $F_1$ and maps them
near $ x_1$. Let $x_{1/2} = F_1 x_0= F_2^ {-1} x_1 $. Then $ \norm
{DF(x_0)}\leq \norm {DF_1(x_0)}\norm {DF_2(x_{1/2})}$.

It is easy to verify that $\norm {DF_1(x_0)}\leq\C$, with the
constant depending only on the curvature of $\partial\mathcal{D}_1$
at $ r_0$.  In other words, the constant may be chosen independent
of $ x_0$. Similarly, $\norm {DF_2^ {-1}(x_1)}\leq\C$.  Because
billiard maps preserve a probability measure with a density
proportional to $\cos\varphi $, $\text {det}DF_2^ {-1}(x_1)=\cos
\varphi_{1}/\cos\varphi_{ 1/2} =\cos\varphi_1$.  As $\Omega$ is $4
$-dimensional, it follows from Cramer's Rule for the inversion of
linear transformations that
\[
    \norm {DF_2(x_{1/2})}\leq \frac {\C\norm {DF_2^ {-1}(x_1)}^3}
    {\text {det}DF_2^ {-1}(x_1)}\leq\frac {\C} {\cos\varphi_1},
\]
and we are done.

\backmatter

\bibliographystyle{alpha}
\addcontentsline{toc}{chapter}{Bibliography}

\end{document}